\documentclass[a4paper, 11pt]{article}
\pdfoutput=1
\usepackage{geometry}
\usepackage[nottoc]{tocbibind} 
\usepackage{amsmath}
\usepackage{amssymb}
\usepackage{amsthm}
\usepackage{sectsty}
\usepackage[font=small]{caption}
\usepackage{comment} 
\usepackage[all]{xy}
\usepackage{color}
\usepackage{amsfonts}
\usepackage[usenames,dvipsnames]{xcolor}
\usepackage{varioref}
\usepackage{tikz}
\usetikzlibrary{decorations.pathreplacing}
\usepackage{tikz-cd}
\usepackage{enumerate}
\usepackage[backref]{hyperref}
\usepackage[noabbrev,capitalise,nameinlink]{cleveref}
\hypersetup{colorlinks={true},linkcolor={blue},citecolor=red}

\newcommand{\selfconn}
{
 \tikz[baseline=-0.1ex]{\draw[densely dotted] (0, -0.05) arc (-90:90: 0.3cm and 0.15cm);
    }
}
\newcommand{\selftriconn}
{
 \tikz[baseline=0ex]{\draw[densely dotted] (0, -0.05) -- (0.3, 0.1);
 \draw[densely dotted] (0, 0.25) -- (0.3, 0.1);
 \draw[densely dotted] (0.3, 0.1) -- (0.55, 0.1);
 \node at (0.75, 0.1) {$s$};
 \draw[fill=black]  (0.55, 0.1) circle (0.035);
  \draw[fill=black]  (0.3, 0.1) circle (0.035);
    }
}
\newcommand{\connedge}{
\tikz[baseline=-0.1ex]{\draw[densely dotted] (0,0) -- (0.5,0);
    \node at (0,0.2) {$v$};
    \node at (0.5, 0.2) {$w$};
    }
}
\newcommand{\conntriedge}{
\tikz[baseline=-0.1ex]{\draw[densely dotted] (0,0) -- (0.7,0);
    \draw[densely dotted] (0.35, 0) -- (0.35, 0.2);
    \node at (0.35, 0.4) {$s$};
    \node at (0,0.2) {$v$};
    \node at (0.7, 0.2) {$w$};
    \draw[fill=black]  (0.35, 0.2) circle (0.035);
     \draw[fill=black]  (0.35, 0) circle (0.035);
    }
}

\usepackage{stmaryrd}

\newtheorem{theorem}{Theorem}[subsection]
\newtheorem{lemma}[theorem]{Lemma}
\newtheorem{proposition}[theorem]{Proposition}
\newtheorem{corollary}[theorem]{Corollary}

\theoremstyle{definition}
\newtheorem{definition}[theorem]{Definition}
\newtheorem{example}[theorem]{Example}
\newtheorem{convention}[theorem]{Convention}

\theoremstyle{remark}
\newtheorem{remark}[theorem]{Remark}

\numberwithin{equation}{subsection}
\renewcommand\thesection{\arabic{section}}
\renewcommand\thesubsection{\arabic{section}.\arabic{subsection}}

\newcommand{\ad}{\mathop{\mathrm{ad}}\nolimits}

\newcommand{\Aut}{\mathop{\mathrm{Aut}}\nolimits}

\newcommand{\bZ}{\mathbb{Z}}
\newcommand{\bQ}{\mathbb{Q}}
\newcommand{\bR}{\mathbb{R}}
\newcommand{\bC}{\mathbb{C}}
\newcommand{\calC}{\mathcal{C}}
\newcommand{\calCV}{\mathcal{C}^{\vee}}
\newcommand{\calLie}{\mathcal{C}\mathcal{L}ie}
\newcommand{\calLieV}{{\mathcal{C}\mathcal{L}ie}^{\vee}}

\newcommand{\Conf}{\mathop{\mathrm{Conf}}\nolimits}
\newcommand{\CorH}{\mathop{\mathrm{Cor}_{\mathcal{H}}}\nolimits}
\newcommand{\CorHod}{\mathop{\mathrm{Cor}_{\mathrm{Hod}}}\nolimits}
\newcommand{\CRC}{\mathop{\mathbb{C}^{\times}_{\bC/\bR}}\nolimits}

\newcommand{\Der}{\mathop{\mathrm{Der}}\nolimits}

\newcommand{\End}{\mathop{\mathrm{End}}\nolimits}
\newcommand{\Ext}{\mathop{\mathrm{Ext}}\nolimits}
\newcommand{\ext}{\mathop{\mathrm{ext}}\nolimits}

\newcommand{\gr}{\mathop{\mathrm{gr}}\nolimits}

\newcommand{\Hod}{\mathop{\mathrm{Hod}}\nolimits}

\newcommand{\HSR}{\mathop{\mathrm{HS}_{\mathbb{R}}}\nolimits}
\renewcommand{\Im}{\mathop{\mathrm{Im}}\nolimits}

\newcommand{\Ker}{\mathop{\mathrm{Ker}}\nolimits}

\newcommand{\MHSR}{\mathop{\mathrm{MHS}_{\mathbb{R}}}\nolimits}

\newcommand{\PM}{\mathop{\mathcal{PM}}\nolimits}

\newcommand{\nil}{\mathop{\mathrm{nil}}\nolimits}

\newcommand{\ord}{\mathop{\mathrm{ord}}\nolimits}

\renewcommand{\Re}{\mathop{\mathrm{Re}}\nolimits}

\newcommand{\sign}{\mathop{\mathrm{sign}}\nolimits}
\newcommand{\sgn}{\mathop{\mathrm{sgn}}\nolimits}

\newcommand{\SmProj}{\mathop{\mathrm{SmProj}}\nolimits}

\newcommand{\Tr}{\mathop{\mathrm{Tr}}\nolimits}

\newcommand{\twi}{\mathop{\mathrm{twistor}}\nolimits}

\newcommand{\Vect}{\mathop{\mathrm{Vect}}\nolimits}

\newcommand{\vol}{\mathop{\mathrm{vol}}\nolimits}

\newcommand{\Id}{\mathop{\mathrm{Id}}\nolimits}

\newcommand{\LieH}{\mathop{\mathrm{Lie_{\Hod}}}\nolimits}

\newcommand\blfootnote[1]{
  \begingroup
  \renewcommand\thefootnote{}\footnote{#1}
  \addtocounter{footnote}{-1}
  \endgroup
}

\makeatletter
\newcommand{\dates}[1]{%
  \let\@@@oldtitle\@title%
  \gdef\@title{\@@@oldtitle\footnotetext{\emph{Date:} #1.}}%
}
\makeatother

\title{\textbf{Quantum master equation and Hodge correlators}}

\author{Hisatoshi Kodani and Yuji Terashima}
\date{}
\begin{document}
\maketitle


\abstract
We give a generalization of Goncharov's Hodge correlator twistor connection. 
Our generalized version is a connection 1-form with values in a DG Lie algebra of uni-trivalent graphs which may have loops and satisfies some Maurer--Cartan equation.  
This connection and the Maurer--Cartan equation can be viewed as an arithmetic analogue of effective action and quantum master equation respectively in non-acyclic Chern--Simons perturbation theory associated with the trivial local system.

\setcounter{tocdepth}{2}

\tableofcontents

\section{Introduction}
Hodge correlators introduced by Goncharov are complex numbers given as certain linear combinations of integrations assigned to a punctured Riemann surface $X \setminus S$ where $X$ is a Riemann surface of genus $g \geq 0$ and $S \subset X$ is a finite set of points in $X$ (\cite{GoncharovHodge1}). 
They contain a wide variety of arithmetic functions such as classical and elliptic polylogarithms, multiple zeta values, and special values of several $L$-functions. 
They provide an alternative description of the standard mixed $\bR$-Hodge structure on the nilpotent completion of the fundamental group of a punctured Riemann surface given by Chen’s iterated integral due to Hain--Zucker (\cite{HZ87}). 
As shown by Goncharov himself, the Hodge correlators can be thought of as tree part of an asymptotic expansion of some quantum field theory and provide a collection of examples of periods given by Feynman diagrams (cf. \cite{Bro17}, \cite{Mar10}). 
In addition, Hodge correlators can be lifted to their motivic avatar, called motivic correlators, which have potential applications to number theory through lifted linear relations from Hodge correlators. 
Recently, Hodge and motivic correlators have been further developed by several authors as in \cite{goncharov2022motivic}, \cite{rudenko2022goncharov}, \cite{malkin2020motivic}, \cite{malkin2020shuffle}.

The dual of Hodge correlators in a family of Riemann surfaces $X \rightarrow B$ with sections canonically gives rise to a connection $\nabla_{\mathcal{G}}= d_{B} + d_{\bC_{\twi}} + \mathbf{G}$, called the Hodge correlator twistor connection, on a vector bundle over $B \times \bC^2$ where $\bC^2$ denotes the twistor plane.
Its connection 1-form $\mathbf{G}$  valued in a Lie algebra of cyclic words of cohomology groups is called the Hodge correlator class and its coefficients are given by integrations associated with uni-trivalent planar tree graphs.
Among other things, the restriction $\mathbf{G}'$ of the Hodge correlator class to the twistor line $\bC_{\twi} \subset \bC^2$ satisfies a Maurer--Cartan equation 
\begin{equation}\label{eq:01}
   ((d_{B} + d_{\bC_{\twi}})\otimes \Id)  \mathbf{G}' + \mathbf{G}' \wedge \mathbf{G}' =0
\end{equation}
which encodes the Griffiths transversality condition needed to define variations of mixed $\bR$-Hodge structures in a canonical way. 
This distinguished feature enables us to take the Green operator, the dual of Hodge correlator map,  as canonical generators of Hodge Galois Lie algebra governing mixed $\bR$-Hodge structure whose other generators were originally given by Deligne in \cite{De_MHSR}. 
Such properties of Hodge correlators can be regarded as a certain realization of ``Hodge field theory'' dreamed up by Beilinson--Levin (\cite{BL94}).

In the present article, we construct one of the generalized versions of the Hodge correlator class and twistor connection by pursuing further quantum field theoretic aspects. 
More precisely, our construction is given from the viewpoints of Chern--Simons perturbation theory or configuration space integrals founded by Axelrod--Singer and Kontsevich (\cite{AS94}, \cite{Kon94}), which was further developed by  Costello, Cattaneo--Mn\"ev, and Iacovino (\cite{Cos07}, \cite{CM}, \cite{I}) in the context of BV formalism and quantum master equations. 
See \cite{BN95}, \cite{Fuk96}, and \cite{Wa18} for other interesting aspects of Chern--Simons perturbation theory as like relations to knot invariants and Morse homotopy.
Now, briefly recall that the effective action $S^{\mathrm{eff}}$ of Chern--Simons theory in family over the unit interval $I=[0,1]$ satisfies the quantum master equation
\begin{equation}\label{eq:02}
    d_I S^{\mathrm{eff}}  + \{S^{\mathrm{eff}}, S^{\mathrm{eff}}\} + \hbar \Delta_{\mathrm{BV}}S^{\mathrm{eff}}= 0
\end{equation}
where $S^{\mathrm{eff}}$ is defined as generating series of Chern--Simons perturbative coefficients of a closed smooth 3-manifold $M$ valued in the polynomial ring of  some cohomology group of $M$ with local coefficients, $\Delta_{\mathrm{BV}}$ is the BV Laplacian and $\{-, -\} $ is the Gerstenhaber bracket on the polynomial ring. 
By taking $\hbar \to 0$, the above equation reduces to the follwing equation as \eqref{eq:01}
\begin{equation}\label{eq:03}
    d_I S^{\mathrm{eff}}_{\mathrm{tree}} + \{S^{\mathrm{eff}}_{\mathrm{tree}}, S^{\mathrm{eff}}_{\mathrm{tree}}\} = 0.
\end{equation}
Note that $S^{\mathrm{eff}}$ is a series of integrations associated with uni-trivalent graphs and $S^{\mathrm{eff}}_{\mathrm{tree}}$ is that with uni-trivalent tree graphs. 
Since the tree part $S^{\mathrm{eff}}_{\mathrm{tree}}$ carries information of Massey products in cohomologies, \eqref{eq:02} and  \eqref{eq:03} imply $S^{\mathrm{eff}}$ can be regarded as a quantum generalization of Massey products. 

By thinking of the Maurer--Cartan equation of Hodge correlator class \eqref{eq:01} as an analogue of the equation \eqref{eq:03} for tree part of non-acyclic perturbative Chern--Simons theory around trivial connection, we define a series $\mathbf{S}^{\mathrm{eff}}$ with values in a DG Lie algebra of uni-trivalent graphs for a family of punctured Riemann surfaces as a generalization of $\mathbf{G}$. The restriction of $\mathbf{S}^{\mathrm{eff}}$ to the twistor line $\bC_{\twi}$ satisfies the following equation (Theorem \ref{thm:main} (2)):
\begin{equation}\label{eq:04}
        ((d_{B} + d_{\bC_{\twi}}) \otimes \Id) \mathbf{S}^{\mathrm{eff}}{}' + \frac{1}{2}[\mathbf{S}^{\mathrm{eff}}{}', \mathbf{S}^{\mathrm{eff}}{}'] + \hbar(\Id \otimes \delta) \mathbf{S}^{\mathrm{eff}}{}'=0 
    \end{equation}
where $\mathbf{S}^{\mathrm{eff}}{}'$ denotes the restriction of $\mathbf{S}^{\mathrm{eff}}$ to the twistor line $\bC_{\twi}$, and $[-,-]$ and $\delta$ are Lie bracket and differential on the DG Lie algebra of uni-trivalent graphs respectively.
Our construction respects the original approach by Goncharov in several parts but, if necessary, we also employ the idea of compactification of configuration spaces especially to apply Stokes theorem. Since the propagators used in Goncharov theory do not extend smoothly to the compactified configuration spaces, more careful treatments are needed in several arguments than usual ones of configuration space integrals.

As other results, we show that our $\mathbf{S}^{\mathrm{eff}}{}'$  recovers the original Hodge correlator class $\mathbf{G}$ in a canonical way (Theorem \ref{thm:tree_red}). 
An explicit relation between one loop term of $\mathbf{S}^{\mathrm{eff}}{}'$ and  2-dimensional Chern--Simons perturbation theory is investigated in Proposition \ref{prop:6.2.1}. 
Moreover, by taking enhanced moduli space  $\mathcal{M}_{0,4}'$ of compact Riemann surfaces of genus $0$ with $4$ points and information of tangential base point as a family of punctured Riemann surfaces, we prove that $\mathbf{S}^{\mathrm{eff}}{}'$ is master homotopic to the formal KZ connection 1-form introduced by Drinfel'd (\cite{Dri90}) in Proposition \ref{prop:M0n'_GHC}. It is interesting to relate this to the work by Rossi--Willwacher (\cite{RW14}).

Finally, we give a remark. 
We note that there is another motivation of the present article coming from arithmetic topology, an area of mathematics based on an analogy between low dimensional topology and number theory (cf. \cite{Mo12}). In our previous works (\cite{KMT17}, \cite{KT23}) with Morishita, we studied some arithmetic analogue of Milnor and Orr invariants which are tree level perturbative invariants of links in quantum topology in the context of Galois action on the fundamental group of a punctured rational curve. One of our purposes is to consider an analogue of loop invariants in this context, since loop invariants were missing there. In some sense, our $\mathbf{S}^{\mathrm{eff}}$ may be regarded as de Rham version of generating series of speculative arithmetic analogues of both tree and loop invariants in this context. 

The organization of this article is as follows. 
In Section \ref{section:2}, we give a brief review of the theory of Hodge correlators by Goncharov. 
In Section \ref{section:3}, we recall the compactification of configuration spaces by Axelrod--Singer, and then we set up compactified configuration space of punctured surfaces which we use in the main part of the article. 
In Section \ref{section:4}, we develop an appropriate graph complex for our purpose and give a detailed explanation of some DG Lie algebra structures of decorated uni-trivalent graphs. 
In Section \ref{section:5}, after recalling the basics of the quantum master equation in BV formalism, we construct a formal effective action which is a framework of algebraic structures of generating series of integrations associated with uni-trivalent graphs. Then, we give a generalized version of the Hodge correlator twistor connection using the tools we set up in the previous sections. 
In Section \ref{section:6}, we study some properties of our twistor connection. For example, we show its relation to the original twistor connection,  discuss its 1-loop terms more explicitly, and study how it relates to the KZ connection. 
In Appendix \ref{section:appendix}, we recall Goncharov's work on some quantum field theory related to Hodge correlator integrals. Then, we explain the 1-loop term in our twistor connection can be written by asymptotic coefficients of the theory similar to the case of tree diagrams.

\subsection*{Acknowledgments}
The authors would like to express their appreciation to Hidekazu Furusho, Akishi Ikeda, Kohei Iwaki, Bingxiao Liu, Toshiki Matsusaka, Masanori Morishita, Tatsuro Shimizu, and Tadayuki Watanabe for their helpful communications. 
H.K. is sincerely grateful to Masanori Morishita and Hiroaki Nakamura, and Koji Tasaka for giving him opportunities to give talks on our early results of the present article at conferences Low Dimensional Topology and Number Theory XIII, and Various Aspects of Multiple Zeta Values 2022 respectively. 
This work is supported by  JSPS KAKENHI Grant Number 21K03240 and 22H01117.

\subsection*{Notation}
The ring of integers is denoted by $\mathbb{Z}$. 
The fields of rational numbers,  real numbers, and complex numbers are denoted by $\mathbb{Q}, \mathbb{R}$ and $\mathbb{C}$ respectively. 
For a subring $R$ of $\mathbb{R}$ and an integer $n$, we set $R(n):= (2\pi i)^n \cdot R \subset \mathbb{C}$.

For a finite set $S$, $|S|$ stands for the number of elements of $S$. 

For a vector space $V$, its dual vector space is denoted by $V^{\vee}$. For a vector space $V$, the tensor algebra of $V$ is denoted by $T(V)$. 
For an associative algebra $A$ over a field $k$ with an algebra morphism $\epsilon: A \rightarrow k$, we set the cyclic envelope of $A$ as  $\mathcal{C}(A):= A^{+}/[A^{+},A^{+}]$ where $A^{+}:= \Ker(\epsilon)$ is the augmentation ideal. 
In particular, when $A=T(V)$, the cyclic envelope $\mathcal{C}(T(V))$ of $T(V)$ will be denoted by $\mathcal{C}T(V)$ for simplicity.

For a smooth compact Riemann surface $X$, the space of holomorphic (resp. anti-holomorphic) differential forms is denoted by $\Omega^1_X$ (resp. $\overline{\Omega}^1_X$). 
We let $\mathcal{D}_{X}^{\ast, \ast}$ the Dolbeault complex of currents on $X$. For a smooth manifold $M$, the space of differential forms are denoted by $\Omega^{\bullet}(M)$ or $\mathcal{A}^{\bullet}_M$.

Let $V$ be a graded vector space over a field. For a homogeneous element $v$, we denote by $|v|$ its homogeneous degree. 
For two graded algebras $M, N$ over a field, let $M \hat{\otimes} N$ denote their graded tensor product. Then, its product structure is given by $(a \hat{\otimes} b)(a' \hat{\otimes} b') = (-1)^{|b||a'|}(aa' \hat{\otimes} bb')$. 
For an integer $n \in \mathbb{Z}$, we denote by $V[n]$ the graded vector space $V$ whose degree is shifted by $n$, i.e., $V[n]^d = V^{d+n}$.

For a commutative ring $R$ with unit, the constant sheaf with value  $R$ on a topological space $X$ is denoted by $\underline{R}_X$ or simply $\underline{R}$ when there is no confusion. 
For a continuous map $f : X \rightarrow Y$ between two topological spaces $X$ and $Y$ and a sheaf $\mathcal{F}$ on $X$, $f_{\ast} \mathcal{F}$ stands for the direct image sheaf of $\mathcal{F}$ along $f$ on $Y$. 
Similarly, $q$-th higher direct images are denoted by $R^qf_{\ast}$.

Let $A$ be an object of an abelian category. 
For a decreasing filtration $F^{\bullet}A$ of $A$, its associated graded is denoted by $\gr_F A = \bigoplus_p \gr_F^p A$. 
Similarly, for an increasing filtration $W_{\bullet}A$ of $A$, its associated graded is denoted by $\gr^W A = \bigoplus_p \gr^W_p A$. 

For a finite-dimensional vector space $H$ over $\mathbb{Q}$ and a finite set $S$, $L_{H,S}$ and $A_{H,S}$ denotes the free Lie algebra and the free associative algebra generated by $H \oplus \mathbb{Q}[S]$ respectively.

\section{Review of Hodge correlators}\label{section:2}
This section reviews the basics of Hodge correlators and Hodge correlator twistor connection introduced by Goncharov in \cite{GoncharovHodge1}. 
See also \cite{malkin2020shuffle} and \cite{malkin2020motivic} for references for genus 0 and 1 cases respectively.

\subsection{Cyclic Lie coalgebra $\calCV_{X, S^{\ast}}$ and Hodge correlator Lie coalgebra $\calLieV_{X, S^{\ast}}$}\label{subsection:2.1}
When a compact Riemann surface $X$ of genus $g \geq 0$ with punctures $S = \{s_0, s_1, \ldots, s_{m}\}$ is given, there are associated cyclic Lie coalgebras $\calCV_{X, S^{\ast}}$ and  $\calLieV_{X, S^{\ast}}$. 
Here, $S^{\ast} := S \setminus \{s_0\}$. 
This section recalls their definition and some properties following \cite{GoncharovHodge1}.

For a finite dimensional vector space $V$, its completed tensor algebra $T(V)$ is defined as $\oplus_{n \geq 0}^{\infty} V^{\otimes n}$. 
We denote by $T^{+}(V):= \oplus_{n \geq 1}^{\infty} V^{\otimes n}$ the degree $\geq 1$ part of $T(V)$. 
Let $[T^{+}(V), T^{+}(V)]$ be the subspace of $T^{+}(V)$ generated by commutators $[v,w]=v \otimes w - w \otimes v$ for $v, w \in T^{+}(V)$. 
Note that $[T^{+}(V), T^{+}(V)]$ is not an ideal of $T^{+}(V)$. The \textit{cyclic envelope} $\mathcal{C}T(V)$ of $V$ is defined as the quotient vector space
\begin{equation}
    \mathcal{C}T(V) := \frac{T^{+}(V)}{[T^{+}(V), T^{+}(V)]}.
\end{equation}

Let $\mathbb{C}[S^{\ast}]$ denote the $\mathbb{C}$-vector space spanned by symbols $\{s\}$ for each $s \in S^{\ast}$. 
Then, we set $V_{X, S^{\ast}} := H_1(X; \mathbb{C}) \oplus \mathbb{C}[S^{\ast}]$, where $H_1(X; \mathbb{C})$ and $\mathbb{C}[S^{\ast}]$ are endowed with weights $-1$ and $-2$  respectively. 
Their dual objects are denoted by $V_{X, S^{\ast}}^{\vee} := H^1(X; \mathbb{C}) \oplus \mathbb{C}[S^{\ast}]$.

Then, we define $\calCV_{X,S^{\ast}}$ and $\calC_{X,S^{\ast}}$ as the cyclic envelopes of $V_{X, S^{\ast}}^{\vee}$ and $V_{X, S^{\ast}}$ respectively
\begin{equation}\label{eq:2.1.2}
    \calCV_{X,S^{\ast}}:= \mathcal{C}T(V_{X, S^{\ast}}^{\vee}) \quad \text{and}\quad  \calC_{X,S^{\ast}}:=\mathcal{C}T(V_{X, S^{\ast}}).
\end{equation}
Elements of $\calCV_{X, S^{\ast}}$ are given by linear combinations of cyclic words in $V_{X, S^{\ast}}^{\vee}$ and such cyclic words are denoted, for instance,  as $W=\calC(\{s_1\} \otimes \alpha \otimes \{s_3\} \otimes \{s_1\}\otimes \beta)$ for $\alpha, \beta \in H^1(X;\mathbb{C})$ and $s_i \in S^{\ast}$. This can also be depicted so that each letter in $W$ is put on the circle counterclockwise as Figure \ref{fig:cyclic_word}.

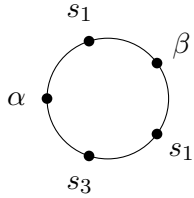
\begin{figure}[h]
\captionsetup{margin=2cm}
\centering
\begin{tikzpicture}
\draw (0,0) circle (0.8cm);
\foreach \a in {1,2,3,4, 5}{
\fill (\a*360/5 + 36 : 0.8cm) circle (2pt);
}
\draw (-1*360/5 + 36: 1.2cm) node{$s_1$};
\draw (-2*360/5 + 36: 1.2cm) node{$s_3$};
\draw (-3*360/5 + 36: 1.2cm) node{$\alpha$};
\draw (-4*360/5 + 36: 1.2cm) node{$s_1$};
\draw (-5*360/5 + 36: 1.2cm) node{$\beta$};
\end{tikzpicture}
\caption{Pictorial expression of the cyclic word $W=\calC(\{s_1\} \otimes \alpha \otimes \{s_3\} \otimes \{s_1\}\otimes \beta)$.}
\label{fig:cyclic_word}
\end{figure}

Next, we define $\calLieV_{X, S^{\ast}}$ and $\calLie_{X, S^{\ast}}$ as follows:
\begin{equation}
    \calLieV_{X, S^{\ast}} := \frac{\calCV_{X,S^{\ast}}}{\text{shuffle relations}} \quad \text{and} \quad \calLie_{X, S^{\ast}} := (\calLieV_{X, S^{\ast}})^{\vee}
\end{equation}
where  \textit{shuffle relations} (also called \textit{ first shuffle relations} in \cite{malkin2020shuffle} and \cite{malkin2020motivic}) is the subspace of $\calCV_{X,S^{\ast}}$ generated by elements
\begin{equation}
    \sum_{\sigma \in \Sigma_{p,q}} (v_0 \otimes v_{\sigma(1)} \otimes \cdots \otimes v_{\sigma(p+q)}), \quad p, q \geq 1
\end{equation}
where $\Sigma_{p,q}$ denotes the set of $(p,q)$-shuffles, i.e., permutations $\sigma \in S_{p+q}$ of $p+q$ elements such that $\sigma(1) < \sigma(2) < \cdots < \sigma(p)$ and $\sigma(p+1) < \sigma(p+2) < \cdots < \sigma(p+q)$. 
Note that we have exact sequences of vector spaces $\calCV_{X,S^{\ast}} \rightarrow \calLieV_{X, S^{\ast}} \rightarrow 0$ and $0 \rightarrow \calLie_{X, S^{\ast}} \rightarrow \calC_{X,S^{\ast}}$ by definition.
\begin{example}[shuffle relations]
Some simple shuffle relations for $(p,q)$ are given as follows:
    \begin{enumerate}[(1)]
        \item $\Sigma_{1,1}:$ $v_0 \otimes v_1 \otimes v_2 + v_0 \otimes v_2 \otimes v_1 =0.$
        \item $\Sigma_{2,1}:$ $v_0 \otimes v_1 \otimes v_2 \otimes v_3 + v_0 \otimes v_1 \otimes v_3 \otimes v_2 + v_0 \otimes v_2 \otimes v_3 \otimes v_1=0$.
    \end{enumerate}
\end{example}

Now, we recall the Lie coalgebra structure of $\calCV_{X, S^{\ast}}$. The Lie cobracket $\delta$ on $\calCV_{X, S^{\ast}}$ is defined as the sum of two Lie cobrackets
\begin{equation}
    \delta = \delta_{S^{\ast}} + \delta_{\mathrm{Cas}} : \calCV_{X, S^{\ast}} \rightarrow \wedge^2 \calCV_{X, S^{\ast}}
\end{equation}
where $\delta_{S^{\ast}}$ and $\delta_{\mathrm{Cas}}$ are defined as following (i) and (ii). Let $W \in \calCV_{X, S^{\ast}}$ be a cyclic word given by
\begin{equation}
    W = \mathcal{C}(\gamma_0 \otimes \gamma_1 \otimes \cdots \otimes \gamma_n).
\end{equation}
We express $W$ as a disk on whose boundary circle  $\gamma_0, \ldots, \gamma_n$ are put counterclockwise as Figure \ref{fig:cyclic_word}.

\noindent
(i) The map $\delta_{S^{\ast}}$: Take a point on the boundary circle of the disk labeled by a puncture $s$ in $S^{\ast}$ and consider a line inside the disk from the chosen point to any point on a segment on the circle adjacent to two labeled points. 
Then, along the line, cut the disk into two parts $C_1$ and $C_2$ which share only the point $s$. Here, $C_1$ denotes the parts lying on the counterclockwise of $s$. 
This yields the term $C_1 \wedge C_2$ which is one of the components of $\delta_{S^{\ast}}W$ as Figure \ref{fig:1.1.2}. The map $\delta_{S^{\ast}}W$ is defined as the sum of these coproduct terms over all such cuts. 
Since $W$ is cyclic word, by taking a point of $S^{\ast}$, say $\gamma_0$, then the map $\delta_{S^{\ast}}$ can also be written explicit way as follows:
\begin{equation}
    \delta_{S^{\ast}} W = \mathrm{Cycle}_{0,1, \ldots, n}\left( \sum_{i=1}^n \mathcal{C}(\gamma_0 \otimes \gamma_i \otimes \cdots \otimes \gamma_n) \wedge \mathcal{C}(\gamma_0 \otimes \gamma_1 \otimes \cdots \otimes \gamma_{i-1})\right)
\end{equation}
where $\mathrm{Cycle}_{0,1, \ldots, n}$ denotes the sum over all the cyclic permutations of  $\{0, 1, \ldots, n\}$ preserving cyclic word $W$.

\noindent
(ii) The map $\delta_{\mathrm{Cas}}$: Consider a line inside the disk starting at a point adjacent to two labeled points $y_1$ and $z_1$ and ending at a point adjacent to two labeled points $y_2$ and $z_2$. 
Then, as in case (i), cut the disk along the line into two circles $C_1$ and $C_2$ where $y_1$ and $z_2$ lie on $C_1$ and $y_2$ and $z_1$ lie on $C_2$. 
Then, insert points labelled by $\alpha$ and $\alpha^{\vee}$ between $y_1$ and $z_2$ on $C_1$ and $y_2$ and $z_1$ on $C_2$ respectively. We denote by $C_1'$ and $C_2$ the resulting cyclic words respectively. 
By taking summation over $\alpha$ and $\alpha$ in a fixed symplectic basis $\{\alpha_i\}_{i=1}^{g}$ and its symplectic dual $\{\alpha_i^{\vee}\}_{i=1}^{g}$ of $H^1(X;\bC)$ respectively, we obtain the coproduct term $C_1'\wedge C_2'$. 
The map $\delta_{\mathrm{Cas}} W$ is obtained as the sum of these products over all such cuts.
More explicitly, $\delta_{\mathrm{Cas}} W$ is given as follows:
\begin{equation}
     \delta_{\mathrm{Cas}} W = \sum_{i=0}^n \sum_{j=0}^n \sum_{k=1}^{g} \mathcal{C}(\gamma_i \otimes \cdots \otimes \gamma_{j-1} \otimes \alpha_k) \wedge \mathcal{C}(\gamma_j \otimes \cdots \otimes \gamma_{i-1} \otimes \alpha_k^{\vee})
\end{equation}

\begin{figure}[h]
\captionsetup{margin=2cm}
 \centering
\begin{tikzpicture}
\begin{scope}
\draw (0,0) circle (1cm);
\foreach \a in {1,4,7}{
\draw (\a*360/7 + 2*360/7: 1.4cm) node{$s_{\a}$};
}
\foreach \a in {2,3,5,6}{
\draw (\a*360/7 + 2*360/7: 1.4cm) node{$\omega_{\a}$};
}
\foreach \a in {1,2,...,7}{
\fill (\a*360/7  : 1cm) circle (2pt);
}
\draw[densely dotted] (1*360/7 + 2*360/7: 1cm) -- (-15*360/14 - 2*360/14: 1cm);
\node at (3, 0) {$\mapsto$};
\end{scope}
\begin{scope}[xshift=5.5cm]
\draw (0,0) circle (0.8cm);
\foreach \a in {1,2,3}{
\fill (\a*360/3 -30: 0.8cm) circle (2pt);
}
\draw (-1*360/3-30: 1.2cm) node{$s_{1}$};
\draw (-2*360/3-30: 1.2cm) node{$\omega_3$};
\draw (-3*360/3-30: 1.2cm) node{$\omega_2$};
\node at (1.75, 0) {$\wedge$};
\end{scope}
\begin{scope}[xshift=9cm]
\draw (0,0) circle (0.8cm);
\foreach \a in {1,2,3,4, 5}{
\fill (\a*360/5: 0.8cm) circle (2pt);
}
\draw (-1*360/5: 1.2cm) node{$\omega_{5}$};
\draw (-2*360/5: 1.2cm) node{$s_4$};
\draw (-3*360/5: 1.2cm) node{$s_1$};
\draw (-4*360/5: 1.2cm) node{$s_7$};
\draw (-5*360/5: 1.2cm) node{$\omega_{6}$};
\end{scope}
\end{tikzpicture}
 \caption{One the components of the differential $\delta_{S^{\ast}}$.}
 \label{fig:1.1.2}
\end{figure}
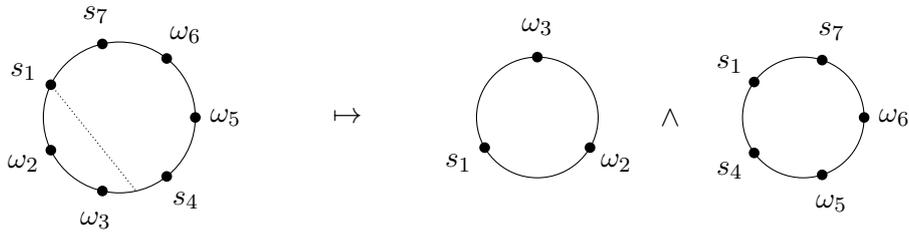

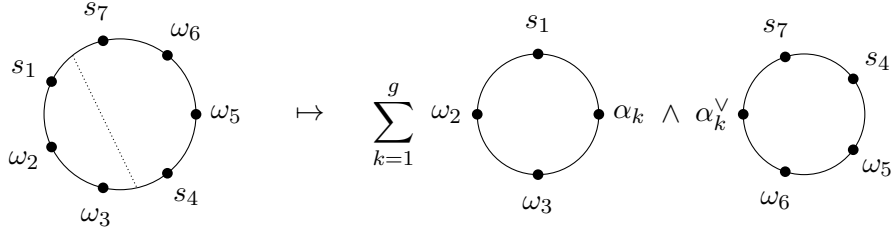
\begin{figure}[h]
\captionsetup{margin=2cm}
 \centering
\begin{tikzpicture}
\begin{scope}
\draw (0,0) circle (1cm);
\foreach \a in {1,4,7}{
\draw (\a*360/7 + 2*360/7: 1.4cm) node{$s_{\a}$};
}
\foreach \a in {2,3,5,6}{
\draw (\a*360/7 + 2*360/7: 1.4cm) node{$\omega_{\a}$};
}
\foreach \a in {1,2,...,7}{
\fill (\a*360/7  : 1cm) circle (2pt);
}
\draw[densely dotted] (-7*360/14 - 2*360/14: 1cm) -- (-15*360/14 - 2*360/14: 1cm);
\node at (2.5, 0) {$\mapsto$};
\end{scope}

\begin{scope}[xshift=5.5cm]
\draw (0,0) circle (0.8cm);
\foreach \a in {1,2,3,4}{
\fill (\a*360/4 : 0.8cm) circle (2pt);
}
\draw (-1*360/4: 1.2cm) node{$\omega_{3}$};
\draw (-2*360/4: 1.2cm) node{$\omega_2$};
\draw (-3*360/4: 1.2cm) node{$s_1$};
\draw (-4*360/4: 1.2cm) node{$\alpha_{k}$};
\node at (1.75, 0) {$\wedge$};
\node at (-1.9,-0.125) {$\displaystyle \sum_{k=1}^{g}$};
\end{scope}
\begin{scope}[xshift=9cm]
\draw (0,0) circle (0.8cm);
\foreach \a in {1,2,3,4, 5}{
\fill (\a*360/5 + 36 : 0.8cm) circle (2pt);
}
\draw (-1*360/5 + 36: 1.2cm) node{$\omega_{5}$};
\draw (-2*360/5 + 36: 1.2cm) node{$\omega_6$};
\draw (-3*360/5 + 36: 1.2cm) node{$\alpha_{k}^{\vee}$};
\draw (-4*360/5 + 36: 1.2cm) node{$s_7$};
\draw (-5*360/5 + 36: 1.2cm) node{$s_4$};
\end{scope}
\end{tikzpicture}
 \caption{One of the components of the differential $\delta_{\mathrm{Cas}}$.}
 \label{fig:1.1.1}
\end{figure}

\subsection{Green functions}\label{section:2.2}
Next, we recall the basics of (normalized) Green functions on compact Riemann surfaces which serve as parts of the building blocks of the Hodge correlators. 
Since they can be expressed as a combination of Arakelov Green functions, we provide a quick review of them too.
\subsubsection{Arakelov Green functions}
This subsection recalls some facts on the Arakelov Green functions. 
For more information, see, for instance, \cite{Ara74} and \cite{Wen91}.

Let $X$ be a compact Riemann surface and $\Omega^{p, q}(X)$ denote the space of complex-valued $(p,q)$-forms. 
Let $\partial:\Omega^{\bullet, \bullet}(X) \rightarrow \Omega^{\bullet+1, \bullet}(X)$ and $\bar{\partial}: \Omega^{\bullet, \bullet}(X) \rightarrow \Omega^{\bullet, \bullet+1}(X)$ be the Dolbeault operators on $X$. 
Let us fix a volume element $\mu \in \Omega^2(X)$ on $X$ normalized so that
\begin{equation}
	\int_X \mu = 1.
\end{equation}
Then, the \textit{Arakelov Green function} $G_{\mathrm{Ar}}(z,w)$ on $X \times X$ is characterized by the following conditions (i), (ii), (iii) and (iv) (\cite{Ara74}, see also \cite{Wen91}):
\footnote{The normalization of Green functions here is different from \cite{Ara74} and \cite{Wen91} but same as \cite{GoncharovHodge1}.}
\begin{enumerate}[(i)]
\item $\exp( G_{\mathrm{Ar}}(z,w))$ is a nonnegative smooth function, defined on $\Conf_2(X)$.
\item $\exp( G_{\mathrm{Ar}}(z,w))$ has a first-order zero on the diagonal $\Delta_X$.
\item $\displaystyle \overline{\partial}_z \partial_z G_{\mathrm{Ar}}(z,w) = -p_1^{\ast} \mu$.
\item $G_{\mathrm{Ar}}(z,w) = G_{\mathrm{Ar}}(w,z)$.
\item $\displaystyle \int_{z \in X} G_{\mathrm{Ar}}(z,w) \mu(z) = 0$.
\end{enumerate}
Here, $\Conf_2(X)$ denotes the 2-points configuration space of $X$, 
that is, 
$\Conf_2(X):= (X \times X) \setminus \Delta_X$ where $\Delta_X \subset X \times X$ is the diagonal subspace, 
the subscript $z$ in $\partial_z$ and $\overline{\partial}_z$ indicates that  operators acts on the first variable of $X \times X$, 
and $p_1 : X \times X\rightarrow X$ denotes the projection map $(z,w) \mapsto z$. 
Here are examples of the Arakelov Green functions for projective line $X = \mathbb{P}^1(\mathbb{C})$ and elliptic curves $X = E_{\tau}$.
\begin{example}\label{example:arakelov}
(1) Let $X = \mathbb{P}^1(\mathbb{C})$ and $z$ be an affine coordinate on $\mathbb{P}^1$. Fix a volume element as
	\begin{equation}
		\mu(z) = \frac{1}{2\pi} \frac{dz \wedge d\bar{z}}{(1 + |z|^2)^2}.\end{equation}
	Then
	\begin{equation}
		G_{\mathrm{Ar}}(z,w) = \frac{1}{2\pi \sqrt{-1}} \log\left( \frac{|z-w|}{\sqrt{(1 + |z|^2)(1 +|w|^2)}}\right) .
	\end{equation}
(2) Let $X = E_\tau = \mathbb{C} / (\mathbb{Z} \oplus \mathbb{Z} \tau)$ with $\tau \in \mathbb{C}$ such that $\Im(\tau) >0$ and $z$ be a holomorphic coordinate inherited from that of $\mathbb{C}$. Fix a volume element as
\begin{equation}
	\mu(z) = \frac{1}{\Im(\tau)} dz \wedge d\bar{z}.
\end{equation}
Then, we set
\begin{equation}
	G_{\mathrm{Ar}}(z,w)  = G_{\mathrm{Ar}}(z-w) 
\end{equation}
where $G_{\mathrm{Ar}}(z)$ is given by 
\begin{equation}
	G_{\mathrm{Ar}}(z) = \frac{\tau - \bar{\tau}}{2 \pi \sqrt{-1}} \sum_{\gamma \in (\mathbb{Z} \oplus \mathbb{Z} \tau)\setminus \{(0,0)\}} \frac{\exp(2 \pi \sqrt{-1} \Im(z \bar{\gamma})/\Im(\tau))}{|\gamma|^2}.
\end{equation}
\end{example}

\subsubsection{(Normalized) Green functions}

Then, we recall the definition of the Green function on a genus $g \geq 0$ compact Riemann surface $X$. Let $\mathcal{D}^{\bullet}(X)=\bigoplus_{k} \bigoplus_{i+j=k} \mathcal{D}^{i,j}(X)$ be the graded vector space of currents on $X$, i.e., differential forms on $X$ with distribution coefficients. Since $X$ is compact, Dirac delta 2-current $\delta_a \in \mathcal{D}^2(X)$ for any point $a \in X$ can be regarded as a normalized volume element of $X$ by evaluating it with constant function $1 \in \Omega^0(X)$.

To define the Green functions, we recall the canonical Hermitian structure on the space $\Omega^1_X$ of holomorphic differential 1-forms on $X$ defined by
\begin{equation}
\langle \alpha, \beta \rangle := \frac{\sqrt{-1}}{2} \int_X \alpha \wedge \bar{\beta}. 
\end{equation}
Let $\{ \alpha_1, \ldots, \alpha_g \}$ be an orthonormal basis of $\Omega^1_X$ with respect to the Hermitian metric, that is,
\begin{equation}
	\langle \alpha_i, \alpha_j \rangle = \frac{\sqrt{-1}}{2} \int_X \alpha_i \wedge \bar{\alpha}_j =\delta_{ij}.
\end{equation}

\begin{definition}[Green function] \label{def:Green_function}
Let $\partial = \partial_z + \partial_w$ and $\bar{\partial}_z + \bar{\partial}_w$ be the Dolbeault operators on $\Omega^{\bullet, \bullet}(X \times X)$. For a fixed volume element $\mu \in \mathcal{D}^2(X)$, the Green function $G_{\mu}(z,w)$ is defined as a solution to the following differential equation:

\begin{equation}\label{eq:def_eq_green_func}
	\bar{\partial} \partial G_{\mu}(z,w) = \delta_{\Delta_X} - p_1^{\ast} \mu - p_2^{\ast}\mu - \frac{\sqrt{-1}}{2} \left(\sum_{k=1}^g p_1^{\ast}\alpha_k \wedge p_2^{\ast}\bar{\alpha}_k + p_2^{\ast} \alpha_k \wedge p_1^{\ast}\bar{\alpha}_k \right)
\end{equation}
where $\delta_{\Delta_X} \in \mathcal{D}^2(X)$ is the Dirac delta 2-current supported at the diagonal subspace $\Delta_X \subset X \times X$ and $p_i: X \times X \rightarrow X$ is the projection map onto the $i$-th component $(i=1,2)$. 
\end{definition}

\begin{remark}
\begin{enumerate}[(1)]
\item Note that the equation \eqref{eq:def_eq_green_func} is independent of the choice of an orthonormal basis of $\Omega_X^1$. Thus, except for the volume element $\mu$, the equation \eqref{eq:def_eq_green_func} is canonically constructed from the structure of Riemann surface $X$.
\item The above differential equation determines $G_{\mu}(z,w)$ uniquely up to an additive constant.
\item The above differential equation should be understood in the distribution (current) sense. The Green function $G_{\mu}(z,w)$ is smooth on $\Conf_2(X)$ but has singularity of logarithm function near the diagonal $\Delta_{X}$. Since $G_{\mu}(z,w)$ is an integrable function, one may regard the smooth function $G_{\mu}(z,w) \in \Conf_2(X)$ with singularity of logarithm function on $\Delta_X$ as the 0-current on $X \times X$ satisfying the defining equation.
\end{enumerate}
\end{remark}

\begin{definition}
By taking a normalized volume element $\mu$ as the Dirac 2-current $\delta_a$ supported at $\{a\}$ for some point $a \in X$, we set
\begin{equation}
	G_{a}(z,w) := G_{\delta_a}(z,w)
\end{equation}
\end{definition}

As explained in \cite[\S 3.1]{GoncharovHodge1}, the Green function can be expressed as a combination of Arakelov Green functions:
\begin{equation}\label{eq:G_a_by_Ara}
	G_a(z,w) = G_{\mathrm{Ar}}(z,w) - G_{\mathrm{Ar}}(a,w) - G_{\mathrm{Ar}}(z,a) + C
\end{equation}
where $C$ is a constant. Since the Arakelov Green function $G_{\mathrm{Ar}}(z,w)$ has singularities of the form $\log|z|$ at the divisor $z=w$, the above formula implies that the Green function $G_a(z,w)$ has singularities of logarithm function at the divisors $z=w$, $z=a$, and $w=a$.   

Recall that there is a way to normalize Green functions $G_a(z,w)$. Let $v \in T_a X$ be a non-zero tangent vector of $X$ at $a \in X$ and $t$ be a local holomorphic coordinate at $a$ such that $dt (v) =1$. Then, there is the unique soution $G_v(z,w)$ to \eqref{eq:def_eq_green_func} with $\mu= \delta_a$ such that $G_v(z,w)-(2\pi \sqrt{-1})^{-1} \log|t|$ vanishes at $z=a$ (and hence at $w=a$), i.e.,  $\lim_{t \to 0} (G_v(t,w) - (2\pi \sqrt{-1})^{-1} \log|t|) = 0$. We call $G_v(z,w)$ the Green function normalized by a non-zero tangent vector $v \in T_aX$.

\begin{example}
For $X = \mathbb{P}^1(\mathbb{C})$ and $a = \infty$, one can show that
\begin{equation}
	G_{\infty}(z,w) = \frac{1}{2\pi \sqrt{-1}} \log |z-w| +C
\end{equation}	
for some constant $C$ by applying \eqref{eq:G_a_by_Ara} with $a= \infty$ for $G_{\mathrm{Ar}}(z,w)$ given in Example \ref{example:arakelov}(1). The constant $C$ depends on the choice of a tangent vector $v \in T_{\infty}X$.
\end{example}

\begin{remark}
    The Green function $G_{a}$ satisfies the differential equation
    \begin{equation}
        \bar{\partial} \partial G_{a}(z,w) = \delta_{\Delta_X}  - \frac{\sqrt{-1}}{2} \left(\sum_{k=1}^g p_1^{\ast}\alpha_k \wedge p_2^{\ast}\bar{\alpha}_k + p_2^{\ast} \alpha_k \wedge p_1^{\ast}\bar{\alpha}_k \right)
    \end{equation}
    on $(X - \{a\}) \times (X - \{a\})$. In particular, for $X=\mathbb{P}^1(\bC)$, we have 
    \begin{equation}
        \bar{\partial} \partial G_{a}(z,w) = \delta_{\Delta_X}
    \end{equation}
    on $(\mathbb{P}^1(\bC)  - \{a\}) \times (\mathbb{P}^1(\bC) - \{a\})$.
\end{remark}

\begin{convention}
    In the sequel, we always consider Green functions $G_{a}$ associated with volume elements of the forms $\mu = \delta_a$ for some point $a \in X$ or its normalized version $G_{v}$ using a tangent vector $v \in T_aX$.
\end{convention}

%
%

\subsection{Hodge correlator integrals and Hodge correlator maps}\label{subsection:2.3}

Let us recall Hodge correlators as integrals associated with planar trivalent graphs decorated by cyclic words of the cohomology group of a punctured Riemann surface. 
For details,  the reader is referred to \cite{GoncharovHodge1, malkin2020shuffle, malkin2020motivic}.

Let $X$ be a smooth compact Riemann surface and $S$ be a finite set of points in $X$ with a distinguished point $s_0$, and fix a tangent vector $v \in T_{s_0}X$. 
Set $S^{\ast}:= S \setminus \{s_0\}$ and $V_{X, S^{\ast}}^{\vee}:= \mathbb{C}[S^{\ast}] \oplus (\Omega^1_X \oplus \bar{\Omega}^1_X)$, where $\Omega_X^1$ (resp. $\bar{\Omega}_X^1$) be the spaces of holomorphic (resp. anti-holomorphic) differential 1-forms on $X$. 
Let $\calCV_{X,S^{\ast}}$ be the cyclic envelope of $V_{X, S^{\ast}}^{\vee}$ defined as \eqref{eq:2.1.2}. 

The {\it Hodge correlators} (or {\it Hodge correlator integrals}) $\CorH$ assigned to the data $(X, S, v)$ take a cyclic word $W \in \calCV_{X,S^{\ast}}$ of $V_{X, S^{\ast}}$ as input and give some complex number as output via summation of integrations of some currents constructed from differentials of Green function $G_v$ associated with planar trivalent tree graphs decorated by $W$. For terminology on graphs used in the present article, see Section \ref{section:4.1}.

Suppose that a cyclic word $W = \mathcal{C}(x_1 \otimes x_2 \otimes \cdots \otimes  x_n) \in \calCV_{X,S^{\ast}}$ for some integer $n \geq 2$ is given. Then, the associated Hodge correlator $\CorH(W)$ is defined separately in the following cases (1) and (2); (1) case that $n=2$, (2) case that $n \geq 3$.

(1) When $n=2$, suppose that  $x_1=s_1, x_2=s_2 \in S^{\ast}$ with $s_1 \neq s_2$, i.e., $W= \mathcal{C}(\{s_1\}\otimes \{s_2\})$. Then we set
\begin{equation}
	\mathrm{Cor}_{\mathcal{H}}(\mathcal{C}(\{s_1\}\otimes \{s_2\})) := G_{v} (s_1, s_2).
\end{equation}
In this case, the corresponding planar tree decorated by $W=\mathcal{C}(\{s_1\}\otimes \{s_2\})$ is given as Figure \ref{fig:2.3.1}. Here, the outer circle depicted by the solid line is not considered as a graph but just for representing the cyclic ordering of $\mathcal{C}(\{s_1\}\otimes \{s_2\})$. Unless otherwise stated, we always assume that the outer circle is oriented counterclockwise. 
\begin{figure}[h]
\captionsetup{margin=2cm}
 \centering
\begin{tikzpicture}
\draw (0,0) circle (1cm);
\foreach \a in {1,2}{
\draw (-\a*360/2 - 2*360/2: 1.4cm) node{$s_{\a}$};
}
\draw[densely dotted] (-1,0) -- (1,0);
\foreach \a in {1,2}{
\fill (\a*360/2  : 1cm) circle (2pt);
}
\end{tikzpicture}
\caption{Decorated planar tree graph associated with a cyclic word $\mathcal{C}(\{s_1\}\otimes \{s_2\})$.}
\label{fig:2.3.1}
\end{figure}
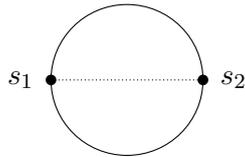

\noindent
For other types of cyclic words $W=\mathcal{C}(x_1\otimes x_2)$, we set $\mathrm{Cor}_{\mathcal{H}}(\mathcal{C}(x_1\otimes x_2)) =0$.

(2) Case that $n \geq 3$. Suppose that a cyclic word $W = \mathcal{C}(x_1 \otimes x_2 \otimes \cdots \otimes  x_n)$ contains $p$ purely holomorphic form and $q$ purely anti-holomorphic forms with $p,q \geq 0$. Draw a disk whose boundary circle is oriented counterclockwise. 
Then, put $x_1, \ldots, x_n$ on the boundary circle disjointly so that the orientation of the circle is compatible with the cyclic ordering of $x_1, \ldots, x_n$. 
Consider an embedding of trivalent tree $T$ whose univalent vertices are mapped to the $n$ boundary points $x_1,\ldots, x_n$ and the rest of $T$ onto the interior of the disk (see Figure \ref{fig:2.3.2} for the case of $n=4$) up to ambient isotopy fixing the boundary circle. Here, a trivalent tree means a tree graph whose vertices are of valency one or three (cf. Section \ref{section:4.1}). 
The resulting (equivalent class of) embedded trivalent tree is called a planar trivalent tree decorated by $W$. 

The tree $T$ has $n-2$ internal (i.e., trivalent) vertices and $2n-3$ edges. 
Let $V_T, V_T^{\mathrm{ext}}, V_T^{\mathrm{int}}$ be the sets of vertices, external (i.e., univalent) vertices and internal vertices of $T$ respectively. We set $V_T^{\mathrm{ext, S^{\ast}}}$ the set of external vertices decorated by elements of $S^{\ast}$.
We endow $V_T^{\mathrm{int}} \cup V_T^{\mathrm{ext,S^{\ast}}}$ with any ordering and $V_T^{\mathrm{ext}}, V_T^{\mathrm{int}}$ with the induced one. Let $E_T$ be the set of edges of $T$. Number the edges  $E_T=\{E_1, \ldots, E_{r-1}, E_{r},\ldots, E_{2n-3}\}$ so that $E_1, \ldots, E_{r-1}$ are internal edges (i.e., edges connecting two internal vertices) or external edges (i.e., edges connected to an external vertex) decorated by elements of $S^{\ast}$ and $E_{r},\ldots, E_{2n-3}$ are external edges decorated by 1-forms in $\Omega^1_X$ or $\bar{\Omega}^1_X$. 

For each edge $E=(v_i,v_j) \in E_T$, we define the associated projection map $p_E : X^{V_T^{\mathrm{int}} \cup V_T^{\mathrm{ext,S^{\ast}}}}\rightarrow X^2$ so that $p_E(x_{1},\ldots, x_{n+1})=(x_{i}, x_j)$. 
We define the current $G_E \in \mathcal{D}^{\bullet}(X^{V_{T}})$ associated with $E$ by the following rule: 
When $E = (v_i,v_j)$ is an internal edge, set $G_E:= p^{\ast}_E(G_v (x_1, x_2))$. 
When $E=(v_i, v_j)$ is an external edge decorated by  $s \in S^{\ast}$ at $v_j$, set $G_E := p^{\ast}_{E}(G_v(x_1, s))$. 
When $E = (v_i, w)$ is an external edge decorated by $\omega \in \Omega^1_X$ or $\bar{\Omega}^1_X$ at $w$, set $\omega_E:=G_E:=  p_{x_{v_i}}^{\ast} \omega$. 
Here, $p_{x_{v_i}}: X^{V_T^{\mathrm{int}} \cup V_T^{\mathrm{ext,S^{\ast}}}} \rightarrow X$ is given by $p_{x_v}(x_1,\ldots, x_{n+1}) = x_i \in X$. 
Then, we define a current of degree $2n-4$  by
\begin{equation}\label{eq:kappa}
	\kappa_T := c_T \cdot \mathrm{Or}_T \cdot G_{E_1} d^{\mathbb{C}} G_{E_2} \wedge \cdots \wedge d^{\mathbb{C}} G_{E_{r-1}} \wedge \omega_{E_{r}} \wedge \cdots \wedge \omega_{E_{2n-3}}
\end{equation}
where we set $d^{\mathbb{C}}:=\partial - \bar{\partial}$, $c_T$ is a constant defined by
\begin{equation}
    c_T:= (-2)^r \binom{r}{\frac{1}{2}(r+p-q)}^{-1},
\end{equation}
and $\mathrm{Or}_T$ is the sign difference between orientations on $E_T$ determined by the embedding of $T$ and by the chosen numbering of $E_T$ (for details, see Section \ref{section:4.3}). 
Note that $\kappa_T$ is independent of the choice of the ordering on the set of edges. Consider the projection map
\begin{equation}
	p_T : X^{V_T^{\mathrm{int}} \cup V_T^{\mathrm{ext,S^{\ast}}}} \rightarrow X^{V_T^{\mathrm{ext}, S^{\ast}}},
\end{equation}
and set 
\begin{equation}\label{eq:kappa_int}
\begin{split}
	\mathrm{Cor}_{\mathcal{H}}(W)&:= \sum_{T} p_{T \ast}(\kappa_T) \\
 &= \sum_{T}c_T \cdot \mathrm{Or}_T  \int_{X^{V^{\mathrm{int}}_T}} G_{E_1} d^{\mathbb{C}} G_{E_2} \wedge \cdots \wedge d^{\mathbb{C}} G_{E_{r-1}} \wedge \omega_{E_{r}} \wedge \cdots \wedge \omega_{E_{2n-3}} .
\end{split}
\end{equation}
Here, the summation runs over all planar trivalent tree graphs decorated by the cyclic word $W$. 
We give a remark on the meaning of $p_{T \ast}$. 
In general, for a proper map $p: X \rightarrow Y$ between smooth manifolds $M$ and $N$, there are pullback $p^{\ast}: \Omega^{p}_c(N) \rightarrow \Omega^{p}_c(M)$ on the spaces of differential forms with compact support and its transpose $p_{\ast}: \mathcal{D}^{\dim M - p}(M) \rightarrow \mathcal{D}^{\dim N - p}(N)$ on the spaces of currents. 
For $J \in \mathcal{D}^{\dim M - p}(M)$, the map $p_{\ast}(J)$ is defined by $p_{\ast}(J)(\alpha) := J(p^{\ast}(\alpha))$ for $\alpha \in \Omega^{p}_c(N)$. 
In our situation, $p_T$ can be considered as a proper map $X^{V_T^{\mathrm{int}}} \rightarrow \{\ast\}$ since points in $S^{\ast}$ are fixed points on $X$. 
Therefore, combining the fact that $\kappa_T$ is $2|V_T^{\mathrm{int}}|=2n-4$-current on $X^{V_{T}}$ (\cite[Lemma 3.3]{GoncharovHodge1}), the map $p_{T \ast}: \mathcal{D}^{2n-4}(X^{V_T^{\mathrm{int}}}) \rightarrow  \mathcal{D}^{0}(\{\ast\}) \simeq \bC$ assigns $\kappa_T$ to some number defined by integration as in \eqref{eq:kappa_int}.
\begin{remark}
The original definition of $\mathrm{Cor}_{\mathcal{H}}(W)$ in \cite{GoncharovHodge1} uses a polydifferential operator by which the appearance of the integrand $\kappa_T$ in \eqref{eq:kappa} becomes more complicated. Our definition here follows the one given in \cite{malkin2020shuffle, malkin2020motivic} since it has a more simplified form.
\end{remark}

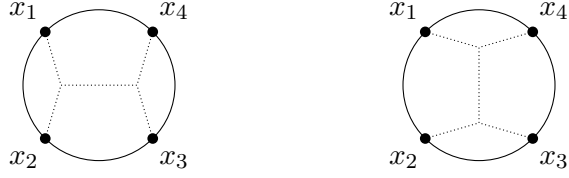
\begin{figure}[h]
\captionsetup{margin=2cm}
 \centering
\begin{tikzpicture}
\begin{scope}
\draw (0,0) circle (1cm);
\foreach \a in {1,2,3,4}{
\draw (\a*360/4 +45: 1.4cm) node{$x_{\a}$};
}
\foreach \a in {1,2, 3, 4}{
\fill (\a*360/4 + 45 : 1cm) circle (2pt);
}
\draw[densely dotted] (-0.5,0) -- (0.5,0);
\draw[densely dotted] (1*360/4 + 45 : 1cm) -- (-0.5,0);
\draw[densely dotted] (2*360/4 + 45 : 1cm) -- (-0.5,0);
\draw[densely dotted] (3*360/4 + 45 : 1cm) -- (0.5,0);
\draw[densely dotted] (4*360/4 + 45 : 1cm) -- (0.5,0);
\end{scope}
\begin{scope}[xshift=5cm]
\draw (0,0) circle (1cm);
\foreach \a in {1,2,3,4}{
\draw (\a*360/4 +45: 1.4cm) node{$x_{\a}$};
}
\foreach \a in {1,2, 3, 4}{
\fill (\a*360/4 + 45 : 1cm) circle (2pt);
}
\draw[densely dotted] (0,-0.5) -- (0,0.5);
\draw[densely dotted] (1*360/4 + 45 : 1cm) -- (0,0.5);
\draw[densely dotted] (2*360/4 + 45 : 1cm) -- (0,-0.5);
\draw[densely dotted] (3*360/4 + 45 : 1cm) -- (0,-0.5);
\draw[densely dotted] (4*360/4 + 45 : 1cm) -- (0,0.5);
\end{scope}
\end{tikzpicture}
\caption[Decorated planar trivalent tree graphs associated with a cyclic word $W=\mathcal{C}(x_1 \otimes x_2 \otimes x_3 \otimes x_4)$]{Decorated planar trivalent tree graphs associated with a cyclic word $W=\mathcal{C}(x_1 \otimes x_2 \otimes x_3 \otimes x_4)$. Note that when $n=4$ there are only two types of planar trivalent trees decorated by $W$.}
\label{fig:2.3.2}
\end{figure}

\begin{example}[Hodge correlators and polylogarithms (\cite{GoncharovHodge1}, \cite{Levin2001}, see also \cite{malkin2020shuffle})]
We consider the case of $X = \mathbb{P}^1(\mathbb{C})$ with $s_0 = \infty$ and recall that Hodge correlator integral at some specific cyclic words give polylogarithms. For this, first, recall polylogarithms and their single-valued versions. Let $Li_n(z)$ be the Euler $n$-logarithm for $n \in \mathbb{N}$, i.e., $Li_n(z)$  is a multi-valued analytic function on $\mathbb{P}^1(\mathbb{C}) \setminus \{0,1,\infty\}$ defined as the analytic continuation of
\begin{equation}
    Li_n(z) = \sum_{k=1}^{\infty} \frac{z^k}{k^n} , \quad |z| <1.
\end{equation}
There is a single valued cousin $\mathcal{L}_n(z)$ of polylogarithms introduced by  Zagier (\cite{zagier90}) by following formula:
	\begin{align}
	\mathcal{L}_n(z) & =  \pi_n\left(\sum_{k=0}^{n-1} \frac{2^k B_k}{k!} Li_{n-k}(z)\log^k |z| \right).
\end{align}
Here, for each complex number $z \in \bC$, $\pi_n$ is defined by
\begin{equation}
    \pi_n(z) = \begin{cases}
         \Re(z) & (\text{$n$:odd}),\\
         \sqrt{-1} \Im(z) & (\text{$n$:even}),
    \end{cases}
\end{equation}
and $B_k$ denotes $k$-th Bernoulli number defined by
\begin{equation}
	\frac{2t}{\exp(2t) - 1} = \sum_{k=0}^{\infty} \frac{2^k B_k}{k!} t^k.
\end{equation}
Then, the relations between Hodge correlators and polylogarithms are described as follows:
\begin{enumerate}[(1)]
\item For a trivalent planar tree decorated by a cyclic word $\mathcal{C}(\{z_1\} \otimes\{z_2\})$ as Figure \ref{fig:planar_trees} (A), the associated Hodge correlator integral is given by
\begin{equation}
\begin{split}
    \mathrm{Cor}_{\mathcal{H}}(\mathcal{C}(\{z_1\} \otimes\{z_2\}) &= G_{v}(z_1, z_2) = \frac{1}{2 \pi \sqrt{-1}} \log |z_1 - z_2|\\
    &=- \frac{1}{2\pi \sqrt{-1}} \mathrm{Re}(Li_1(1+z_2 - z_1)) =  - \frac{1}{2\pi \sqrt{-1}} \mathcal{L}_1(1 + z_2 - z_1).
    \end{split}
\end{equation}
for a chosen $v \in T_{\infty} X$. In particular, for $z_1 = 1, z_2 = z$, we have
\begin{equation}
    \mathrm{Cor}_{\mathcal{H}}(\mathcal{C}(\{1\} \otimes\{z\}) =  - \frac{1}{2\pi \sqrt{-1}} \mathcal{L}_1(z).
\end{equation}
\item Next, consider a decorated trivalent planar tree as Figure \ref{fig:planar_trees} (B). The associated Hodge correlator integral is then given by 
\begin{equation}
\begin{split}
     \mathrm{Cor}_{\mathcal{H}}(\mathcal{C}(\{z_1\} \otimes\{z_2\}\otimes \{z_3\}) &= -  \frac{1}{8}  \int_{x \in \mathbb{C}}\frac{1}{(2\pi \sqrt{-1})^3} \log| x - z_0| d^{\mathbb{C}} \log| x- z_1| \wedge d^{\mathbb{C}} \log| x- z_2|\\
     &=-\frac{1}{(2 \pi \sqrt{-1})^2} \mathcal{L}_2\left(\frac{z_3 - z_1}{z_2 - z_1} \right).
     \end{split}
\end{equation}
In particular, for $z_1= 0$, $z_2 = 1, z_3 = z$, we have
\begin{equation}
     \mathrm{Cor}_{\mathcal{H}}(\mathcal{C}(\{0\} \otimes\{1\}\otimes \{z\}) = -\frac{1}{(2 \pi \sqrt{-1})^2} \mathcal{L}_2(z).
\end{equation}
\item More generally, consider the cyclic word $W_n$ of length $n+1$ defined by
\begin{equation}
    W_n = \mathcal{C}(\{1\} \otimes\{z\}\otimes \{0\} \otimes \cdots \otimes \{0\})\quad (\text{$n-1$ iteration of $\{0\}$}).
\end{equation}
Then, we have
    \begin{equation}
	\mathrm{Cor}_{\mathcal{H}}(W_n) = - \frac{1}{(2\pi \sqrt{-1})^n} \mathbb{L}_n(z)
\end{equation}
Here, $\mathbb{L}_n(z)$ is the Levin's polylogarithm function (\cite[Definition 4.2.1]{Levin2001}) defined as follows:
\begin{equation}
\mathbb{L}_{n}(z) = \sum_{\substack{0 \leq k \leq n-2 \\ k \ \text{even}}} \frac{2^k (n-2)!(2n - k -3)!}{(2n - 3)!(k+1)! (n-k-2)!} \mathcal{L}_{n-k}(z)\log^k|z|.
\end{equation}
Note that only the decorated trivalent planar tree as Figure \ref{fig:planar_trees} (C) contribute to $\mathrm{Cor}_{\mathcal{H}}(W_n)$.
\end{enumerate}
\end{example}

\begin{figure}[h]
\captionsetup{margin=2cm}
 \centering
\begin{tikzpicture}
\begin{scope}
\draw (0,0) circle (1.5cm); 
\draw[densely dotted] (-1.5,0) -- (1.5,0); 
\node at (1.5,0) {$\bullet$};
\node at (1.5+0.5,0) {$z_2$};
\node at (-1.5, 0) {$\bullet$};
\node at (-1.5-0.5, 0) {$z_1$};
\node at (0, -2.3) {(A)};
\end{scope}
\begin{scope}[xshift=4.5cm]
\draw (0,0) circle (1.5cm);
\node at (0,0) {$\bullet$};
\foreach \a in {1,2,3}{
\draw (\a*360/3 -30: 1.5cm) node {$\bullet$};
\draw[densely dotted] (\a*360/3 -30: 1.5cm) -- (0,0);
}
\draw (1*360/3-30: 1.9cm) node{$z_1$};
\draw (2*360/3-30: 1.9cm) node{$z_2$};
\draw (3*360/3-30: 1.9cm) node{$z_3$};
\node at (0, -2.3) {(B)};
\end{scope}
\begin{scope}[xshift=9.0cm]
	\draw (0,0) circle (1.5cm);
	\coordinate (a0) at (-1.5,0);
 \node at (-1.9, 0) {$1$};
	\coordinate (a1) at (1.5,0);
 \node at (1.9, 0) {$z$};
	\node at (a0) {$\bullet$};
	\node at (a1) {$\bullet$};
	\draw[densely dotted] (-1.0, 0) -- (-1.0, 1.1);
	\draw[densely dotted] (-0.5, 0) -- (-0.5, 1.4);
	\draw[densely dotted] (1.0, 0) -- (1.0, 1.1);
	\draw[densely dotted] (0.5, 0) -- (0.5, 1.4);
 \foreach \x in {(-1.0, 0),(-1.0, 1.1),(-0.5, 0),(-0.5, 1.4),(1.0, 0),(1.0, 1.1),(0.5, 0) ,(0.5, 1.4)  }{
 \node at \x {$\bullet$};
 }
	\node [above] at (-1.0, 1.3) {$0$};
	\node [above] at (-0.5, 1.6) {$0$};
	\node [above] at (1.0, 1.3) {$0$};
	\node [above] at (0.5, 1.6) {$0$};
	\node [above] at (0, 1.6) {$\cdots$};
	\node at (0, 0.75) {$\cdots$};
	\draw[densely dotted] (a0) -- (a1);
 \node at (0, -2.3) {(C)};
\end{scope}
\end{tikzpicture}
\caption{Examples of decorated trivalent planar tree graphs.}
\label{fig:planar_trees}
\end{figure}
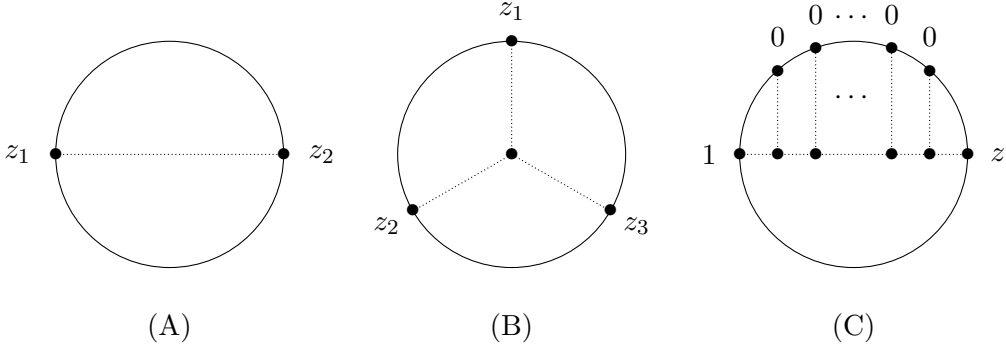

The Hodge correlators satisfy a family of \textit{(first) shuffle relations} as follows (\cite[Proposition 3.5]{GoncharovHodge1}): For $v_0, v_1, \ldots, v_{p+q} \in V_{X,S^{\ast}}^{\vee}$ and $p, q \geq 1$, we have
\begin{equation}\label{eq:first_shuffle}
\sum_{\sigma \in \Sigma_{p,q}} \CorH(\mathcal{C}( v_0 \otimes v_{\sigma(1)} \otimes v_{\sigma(2)} \otimes \cdots \otimes v_{\sigma(p+q)})) = 0.
\end{equation}

\begin{remark}
    In \cite{malkin2020shuffle} and \cite{malkin2020motivic}, Malkin gives another type of shuffle relations, called \textit{second shuffle relations} for Hodge correlators when $X = \mathbb{P}^1(\bC)$ or elliptic curve and their relations under degeneration of elliptic curve to a nodal projective line. 
    For $X = \mathbb{P}^1(\bC)$, it is conjectured the first and the second shuffle relations give all the linear relations among the Hodge correlators. 
    For more details on this direction, see [loc.cit.].
\end{remark}

Since the Hodge correlators $\CorH(\mathcal{C}( v_0 \otimes v_1 \otimes \cdots \otimes v_n))$ satisfy cyclic relation by definition
\begin{equation}
    \CorH(\mathcal{C}( v_0 \otimes v_1 \otimes \cdots \otimes v_n)) = \CorH(\mathcal{C}(v_1 \otimes v_2 \otimes \cdots \otimes v_n \otimes v_0)), 
\end{equation}
together with the (first) shuffle relation \eqref{eq:first_shuffle}, Hodge correlators induces the following map, called the \textit{Hodge correlator map},  
\begin{equation} 
    \CorH : \calLieV_{X, S^{\ast}}(1) \rightarrow \mathbb{C}; \quad \mathcal{C}( v_0 \otimes v_1 \otimes \cdots \otimes v_n)(1) \mapsto \CorH(\mathcal{C}( v_0 \otimes v_1 \otimes \cdots \otimes v_n)).
\end{equation}
Here, we need to take the Tate twist $\calLieV_{X, S^{\ast}}(1):= \calLieV_{X, S^{\ast}} \otimes H_2(X)$ of $\calLieV_{X, S^{\ast}}$ in order the Lie cobracket $\delta$ to be a morphism of mixed $\bR$-Hodge structures. 
Note that $H_2(X) \simeq \bR(1)$. 
In the sequel, we often call values of $\CorH$ at some elements of $\calLieV_{X, S^{\ast}}(1)$ the \textit{Hodge correlator integrals}.

The dual to the Hodge correlator map is called the \textit{Green operator} and is denoted by
\begin{equation} \label{eq:green_operator}
    \mathbf{G}_{v} := (\CorH)^{\vee} \in \calLie_{X, S^{\ast}}(-1).
\end{equation}
Note that it has the decomposition by Hodge bigrading
\begin{equation}\label{eq:2.3.21}
    \mathbf{G}_{v} =\sum_{p, q\geq 1}  \mathbf{G}_{p,q}
\end{equation}
where $\mathbf{G}_{p,q}$ is of Hodge bidegree $(-p, -q)$ with $p,q \geq 1$. Moreover, they satisfy the relations
\begin{equation}\label{eq:2.3.22}
     \overline{\mathbf{G}}_{p,q} = - \mathbf{G}_{q,p}.
\end{equation}

\subsection{Mixed $\bR$-Hodge structures and Hodge correlators}

This section recalls the upgraded version of Hodge correlators in Hodge theoretical viewpoints following \cite{GoncharovHodge1}. See also \cite{malkin2020shuffle} and  \cite{malkin2020motivic}.

\subsubsection{Hodge Galois Lie algebra}

Let $\HSR$ and $\MHSR$ be the categories of $\bR$-Hodge structures and of $\bR$-mixed Hodge structures respectively. The category $\HSR$ is equivalent to the category of representation of the algebraic group $\mathbb{S} = \mathrm{Res}^{\bC}_{\bR} \mathbb{G}_m$, where  $\mathrm{Res}^{\bC}_{\bR} \mathbb{G}_m$ is the restriction of scalars of the multiplicative group $\mathbb{G}_m$ from $\bC$ to $\bR$ (see \cite[Chapter 2]{mixed_hodge}), sometimes called Deligne torus. Following \cite{GoncharovHodge1}, we denote it by $\mathbb{C}^{\times}_{\bC/\bR}$ by abuse of notation.

The category $\MHSR$ is a Tannakian category with a canonical fiber functor $\omega_{\Hod}: \MHSR \rightarrow \mathrm{Vect}_{\bR}$ which assigns each real mixed Hodge structure to its underlying real vector space. The \textit{Hodge Galois group} $G_{\Hod}$ is defined as the automorphism group of the fiber functor $\omega_{\Hod}$ preserving tensor structures: $G_{\Hod}:= \Aut^{\otimes}(\omega_{\Hod})$. The Hodge Galois group is a pro-algebraic group over $\bR$ and the category $\MHSR$ is equivalent to the category of representations of the pro-algebraic group $G_{\Hod}$, i.e., $\omega_{\Hod}$ induces the equivalence
\begin{equation}
    \omega_{\Hod}: \MHSR \overset{\sim}{\rightarrow} \text{$G_{\Hod}$-modules}.
\end{equation}

We have now two canonical functors $\gr^W: \MHSR \rightarrow \HSR$ and $\iota: \HSR \rightarrow \MHSR$ with $\gr^W \circ \iota = \Id$ so that we obtain the induced split exact sequence
\begin{equation}\label{eq:gal_split}
   \begin{tikzcd}
0 \arrow[r] & U_{\Hod} \arrow[r] & G_{\Hod} \arrow[r, "p"] & \bC^{\times}_{\bC/\bR} \arrow[r] & 0.
\end{tikzcd}
\end{equation}
Here, we set $U_{\Hod}:= \Ker p$ that is a real pro-unipotent group.  The \textit{Hodge Galois Lie algebra} $\LieH$ is defined as the Lie algebra of $U_{\Hod}$. By split exact sequence \eqref{eq:gal_split}, $\bC^{\times}_{\bC/\bR}$ acts on $U_{\Hod}$ via the section $s: \CRC \rightarrow G_{\Hod}$ so that $U_{\Hod}$ and $\LieH$ are endowed with $\bR$-Hodge structures, and so $\LieH$ is a Lie algebra in the category $\HSR$. Thus, the $G_{\Hod}$-modules are nothing else but the $\LieH$-module in the category  $\HSR$ and $\LieH = \mathrm{Der}^{\otimes}(\gr^W)$, the Lie algebra of tensor derivations of the functor $\gr^W: \MHSR \rightarrow \HSR$ . In this way, the category $\MHSR$ is equivalent to the category of $\LieH$-module in the category  $\HSR$. 

Since for any simple object $H \in \HSR$ the higher Ext-groups vanish $\Ext^{p}_{\MHSR}(\bR(0), H) = 0$ ($p \geq 2$) (cf. \cite[Proposition 3.35]{mixed_hodge}), we conclude that the Hodge Galois group $\LieH$ is a free Lie algebra and, in fact, it is generated by
\begin{equation}
    \bigoplus_{[H]} \Ext^1_{\MHSR}(\bR(0), H)^{\vee} \otimes H
\end{equation}
where the sum runs over all the isomorphism classes of simple objects $H$ in $\HSR$. For details on some algebraic facts of the Tannakian category described above, see \cite[Appendice A]{DG05}.

\begin{example}
    Let $X$ be a compact Riemann surface and $S \subset X$ be a finite set of points on $X$. Let $s_0 \in S$ be a distinguished point and $v_0 \in T_{s_0} X$ be a distinguished tangent vector at $s_0$. Then, it is known that the pronilpotent completion $\pi_1^{\nil}(X\setminus S, v_0)$ of the fundamental group $\pi_1(X\setminus S, v_0)$ carries a real mixed Hodge structure which depends on $v_0$ so that there is a map
    \begin{equation}
        \LieH \rightarrow \Der(\gr^W \pi_1^{\nil}(X\setminus S, v_0)).
    \end{equation}
   In particular, when a real mixed Hodge structure is described by \textit{special derivations} of $\gr^W \pi_1^{\nil}(X\setminus S, v_0)$ which acts by 0 on the small loop around $s_0$ and preserve the conjugacy classes of all the small loops around $s \in S \setminus \{s_0\}$, there is the corresponding map
   \begin{equation}\label{eq:2.4.5}
        \LieH \rightarrow \Der^S(\gr^W \pi_1^{\nil}(X\setminus S, v_0)).
    \end{equation}
    Here $\Der^S(\gr^W \pi_1^{\nil}(X\setminus S, v_0))$ denotes the Lie algebra of special derivations of $\gr^W \pi_1^{\nil}(X\setminus S, v_0)$.
\end{example}

\subsubsection{Hodge Galois Lie algebra, period map, and Hodge correlators}
The Hodge Galois Lie algebra $\LieH$ is a free Lie algebra generated by certain element $g_{p,q}$ $(p,q \geq 1)$ which satisfies the only relation $\bar{g}_{p,q} = - g_{q,p}$ by the result of Deligne (\cite{De_MHSR}). One of the important features of Hodge correlators is that the Green operator $\mathbf{G}_v = \sum_{p,q \geq 1} \mathbf{G}_{p,q}$ gives a canonical generator $\{\mathbf{G}_{p,q}\}$ of  $\LieH$, which generalizes the previous result by Levin for the mixed $\bR$-Hodge--Tate structures (the case that $p=q$) (\cite{Levin2001}).

Recall that by choosing generators $n_{p,q}$ of $\LieH \otimes \bC$ with $\overline{n}_{p,q} = - n_{q,p}$, we obtain a period map
\begin{equation}
    p_{[H]}: \LieH^{\vee} \rightarrow i \bR
\end{equation}
for simple $\bR$-Hodge structure $H$ as follows. The chosen generator of $\LieH$ defines an inclusion $\bigoplus_{[H]} \Ext^1_{\MHSR}(\bR(0), H)^{\vee} \otimes H \rightarrow \LieH$. By taking its dual, we have a projection
\begin{equation}
    \LieH^{\vee} \rightarrow \bigoplus_{[H]} \Ext^1_{\MHSR}(\bR(0), H) \otimes H^{\vee}.
\end{equation}
Since $\Ext^1_{\MHSR}(\bR(0), H) = \mathrm{Coker}(W_0 H \oplus F^0(W_0 H) \rightarrow W_0H_{\bC})$ for any $\bR$-Hodge structure $H$ (cf. \cite[Theorem 3.31]{mixed_hodge}), we have $\Ext^1_{\MHSR}(\bR(0), H) = H \otimes \bC / H =  H \otimes_{\bR} i \bR $ for simple $\bR$-Hodge structure $H$. Combining these with canonical pairing $H\otimes H^{\vee} \rightarrow \bR$, one obtains the desired period map
\begin{equation}
    p_{[H]} : \LieH^{\vee} \rightarrow \bigoplus_{[H]} \Ext^1_{\MHSR}(\bR(0), H) \otimes H^{\vee} = \bigoplus_{[H]}  H\otimes i\bR  \otimes H^{\vee} \rightarrow i \bR.
\end{equation}
While originally certain generating set of $\LieH$ is defined by Deligne in \cite{De_MHSR} as above,  Goncharov proposed to use the Green operator $\mathbf{G}_{v}$ as a canonical generating set instead of Deligne's one by the following reasons. When describing variations of real Hodge structures, the Griffiths transversality condition is needed to define them. The Green operator $\mathbf{G}_{v}$ satisfies an explicit Maurer--Cartan differential equation which encodes this condition canonical manner, while this condition for Deligne's generators is complicated and difficult to write down.

Therefore, in the above sense,  we call the period map $p_{[H]}$ defined by choosing $\mathbf{G}_{v}$ as generators of  $\LieH$ \textit{canonical period map} and denote it by $\mathcal{P}$.

Since there is an injective morphism $\calLie_{X, S^{\ast}} \rightarrow \Der(\gr^W \pi_1^{\nil}(X\setminus S, v_0)) \otimes \bC$ whose image is the special derivations of $\gr^W \pi_1^{\nil}(X\setminus S, v_0) \otimes \bC$ (cf. \cite[\S 8.2]{GoncharovHodge1}), the Green operator $\mathbf{G}_{v}$ defines a real mixed Hodge structures on $\pi_1^{\nil}(X\setminus S, v_0)$. Therefore, it gives rise to the \textit{Hodge correlator morphism} of Lie coalgebras
\begin{equation}\label{eq:Hodge_correlator_morph}
    \mathrm{Cor}_{\Hod}: \calLie_{X, S^{\ast}}^{\vee}(1) \rightarrow \LieH^{\vee},
\end{equation}
the dual of the map \eqref{eq:2.4.5} induced by the action of $\mathbf{G}_{v}$. Note that, for  $x \in \calLie_{X, S^{\ast}}^{\vee}(1)$ of type $(p,q)$, $\mathrm{Cor}_{\Hod}(x)$  is the element of $\LieH^{\vee}$ induced by a framing $\bR(p,q) + \bR(q,p) \rightarrow \gr^W_{p+1} \pi_1^{\nil}(X\setminus S, v_0)$ provided by $x$ (cf. \cite{BGSV}, \cite[\S 1.4]{Gon97}).

\begin{theorem}{\cite[Theorem 1.12, Theorem 1.14]{GoncharovHodge1}}\label{thm:Gon1} \begin{enumerate}[(1)]
\item Let $\mathcal{P}$ denote the canonical period map. Then, the following  diagram commutes, that is, for $x \in \calLie_{X, S^{\ast}}^{\vee}(1)$, we have $\CorH(x) = \mathcal{P} \circ \CorHod(x)$:
\begin{figure}[h]
    \centering
    \begin{tikzcd}
 \calLie_{X, S^{\ast}}^{\vee}(1) \arrow[rd, "\CorH", swap] \arrow[r, "\CorHod"] & \LieH^{\vee}\otimes \mathbb{C}  \arrow[d, "\mathcal{P}"]\\
& \bC
\end{tikzcd}
\end{figure}
\item The mixed $\bR$-Hodge structure on $\pi_1^{\nil}(X\setminus S, v_0)$ determined by the Green operator $\mathbf{G}_v = \CorH^{\vee}$ coincides with the standard mixed $\bR$-Hodge structure on $\pi_1^{\nil}(X\setminus S, v_0)$.
\end{enumerate}
\end{theorem}

When the data $(X, S,v_0)$ varies, there is a generalization of Theorem \ref{thm:Gon1} for variations of mixed $\bR$-Hodge structures as follows. 
Let $\mathcal{M}_{g,n}'$ be the enhanced moduli space of Riemann surfaces of genus $g$ with $n=|S|$ punctures and additional information of tangent vectors at a distinguished point.
Varying the data $(X, S,v_0)$ over $\mathcal{M}_{g,n}'$, the family $\mathbf{V}:=L_{X_{/B},S^{\ast},v_0}$ of free Lie algebras $L_{X,S^{\ast},v_0}$ generated by $\gr^W H^1(X \setminus S)$ forms a variation of $\bR$-Hodge structures over $\mathcal{M}_{g,n}'$. Then, a family of Green operator $\mathbf{G}_{v}$ over $\mathcal{M}_{g,n}'$ is an element of endomorphism $\End(\mathbf{V}_{\infty})$ and determines a mixed $\bR$-Hodge structure on $\mathbf{V}$ fiberwise. Here, $\mathbf{V}_{\infty}:= \mathbf{V} \otimes_{\mathbb{C}} C_{\mathcal{M}_{g,n}'}^{\infty}$ and $C_{\mathcal{M}_{g,n}'}^{\infty}$ denotes the sheaf of smooth complex functions on $\mathcal{M}_{g,n}'$. Provided further an operator-valued 1-form $g_{0,0}$, the pair $(\mathbf{G}_{v_0}, g_{0,0})$ can be transformed into a flat connection $\nabla_{\mathcal{G}}$, called \textit{twister connection}, on $\pi^{\ast} \mathbf{V_{\infty}}$ where $\pi: \mathcal{M}_{g,n}' \times \bC_{\twi} \rightarrow \mathcal{M}_{g,n}'$ and $\bC_{\twi} \subset \bC^2$ denotes the twistor line. 
\begin{theorem}{\cite[Theorem 1.12]{GoncharovHodge1}}\label{thm:Gon2}
\begin{enumerate}[(1)]
\item A restriction $\nabla^{1/2}_{\mathcal{G}}$ of $\nabla_{\mathcal{G}}$ to $\left(\frac{1}{2}, \frac{1}{2}\right)\in \bC_{\twi}$ is a flat connection on $\mathbf{V_{\infty}}$ which defines a variation of mixed $\bR$-Hodge structures over $\mathcal{M}_{g,n}'$.
\item  This variation is isomorphic to the standard variation of mixed $\bR$-Hodge structures on $\pi_1(X \setminus S, v_0)$.
\end{enumerate}
\end{theorem}
One of the goals in the sequel is to give one of the generalized versions of the twistor connection $\nabla_{\mathcal{G}}$ by using general decorated uni-trivalent graphs unrestricted to planar trivalent trees.

\begin{remark}
   One can upgrade Hodge correlators to \textit{motivic correlators} using conjectural abelian category $\mathcal{MM}_F$ of mixed motives over a number field $F$. Relations on Hodge correlators can be lifted to those on motivic correlators provided that some conditions are satisfied (\cite[Lemma 3]{malkin2020motivic}). Under the assumption of the motivic formalism, the motivic correlators allow us to translate these relations in the Hodge correlators to those in $l$-adic realization. For details, see \cite[\S 1.11]{GoncharovHodge1} and \cite[\S 2.2.5]{malkin2020motivic}.
\end{remark}

\section{Compactified configuration spaces and Green functions} \label{section:3}
Here, we recall and prepare some basic results around the compactification of configuration spaces. To begin with, let us review the notion of compactified configuration space introduced by Axelrod--Singer (Section \ref{ssecion:3.1}). Then, we set up compactified configuration space of closed surfaces with marked points, which is needed in subsequent sections, as a two-dimensional analogue of Bott--Taubes's compactified configuration spaces of manifolds containing embedded submanifolds (Section \ref{secion:BT_cmpct}, \ref{ssecion:3.3}, \ref{section:boundary_strata}). Finally, we argue for a smooth extension of some differential of Green functions to compactified configuration spaces (Section \ref{ssecion:3.5}).

\subsection{Compactification of configuration spaces by Axelrod--Singer}\label{ssecion:3.1}
To begin with, we recall the compactification configuration spaces as manifold with corners introduced by Axelrod--Singer as a differential geometric analogue of the algebra geometric compactification by Fulton--MacPherson. For more details see \cite[\S 5]{AS94} and \cite{FM94}.

Let $M$ be an oriented closed smooth manifold of dimension $m$. For a positive integer $V$, we set $\underline{V}:=\{1,\ldots, V\}$.
Then, by definition of the direct product, we have $M^{V}=M^{\underline{V}} = \mathrm{Maps}(\underline{V},M)=\prod_{i\in \underline{V}}M_{i}$ where $\mathrm{Maps}(\underline{V},M)$ is the set of all maps from $\underline{V}$ to $M$ and $M_{i}=M$ is just a copy of $M$. 
For a subset $S \subset \underline{V}$ with $|S| \geq 2$, we denote by $\Delta_{S}\simeq M$ the smallest diagonal in  $M^{S}=\mathrm{Map}(S,M)$ consisting of constant maps from $S$ to $M$. 
Let $\pi_S: M^{\underline{V}} \rightarrow M^S$ be the projection map which maps $(x_v)_{v \in \underline{V}}$ to $(x_s)_{s \in S}$ by forgetting $x_v$ for $v \notin S$. 
We set $\overline{\Delta}_S \subset M^{\underline{V}}$ as the inverse image of $\Delta_S$ under the projection $\pi_S$, i.e., $\overline{\Delta}_S := \pi_S^{-1}(\Delta_S)$

Note that $\Delta_{S}$ is a closed embedded submanifold of $M^{S}$. 
We denote by $B\ell(M^{S},\Delta_{S})$ the geometric blow-up of $M^{S}$ along $\Delta_{S}$ with the blow-down map $q_S: B\ell(M^{S},\Delta_{S}) \rightarrow M^S$. It is compact manifold with boundary whose interior is $M^S \setminus \Delta_{S}$ and boundary is the sphere bundle $S\nu_{\Delta_{S}}$ of the normal bundle $\nu_{\Delta_{S}}$ to $\Delta_{S}$. The restriction of the boundary of the blowdown map $\partial q_S:= q_S|_{\partial B\ell(M^{S},\Delta_{S})}: \partial B\ell(M^{S},\Delta_{S})\rightarrow \Delta_{S} \subset M^S$ is given by the bundle projection map $S\nu_{\Delta_{S}}\rightarrow \Delta_{S}$.

For a positive integer $V$, we consider the $V$-points configuration space of $M$ defined by
\begin{equation}
	\mathrm{Conf}_V(M):= \{ (x_1, \ldots, x_V) \in M^{\underline{V}} \mid x_i \neq x_j (i \neq j)\}.
	\label{eq:1.2.2}
\end{equation}
For a subset $S \subset \underline{V}$ with $|S| \geq 2$, the projection map $\pi_S: M^{\underline{V}}\rightarrow M^S$ restricts to the map (by abuse of notation) $\pi_S:\Conf_V(M) \rightarrow  B\ell(M^{S},\Delta_{S})$ whose image lies on the interior of $B\ell(M^{S},\Delta_{S})$. 
Setting $\Phi_V:= \iota \times \prod\pi_S$ where $\iota:\mathrm{Conf}_V(M)\rightarrow M^{\underline{V}}$ is the canonical inclusion and the product runs over all $S \subset \underline{V}$ with $|S|\geq 2$, we obtain the following injective smooth map:
\begin{equation}
	\Phi_{V}:\mathrm{Conf}_V(M)\rightarrow M^{\underline{V}} \times \prod_{S 
\subset \underline{V}, |S| \geq 2} B\ell(M^S, \Delta_S)=:\mathcal{B}_V(M).
\label{eq:1.2.3}
\end{equation}
Note that the target space of $\Phi_{V}$ is a product of compact manifolds with 
boundary, i.e., compact manifold with corners.

The compactification $C_n(M)$ of $\mathrm{Conf}_n(M)$ is defined 
as the closure of the image of $\mathrm{Conf}_n(M)$ via $\Phi_{n}$ 
equipped with the induced smooth structure, i.e.,
\begin{equation}
	C_V(M) := \overline{\Phi_{V}(\mathrm{Conf}_V(M))} \subset M^{\underline{V}} \times \prod_{S 
\subset \underline{V}, |S| \geq 2} B\ell(M^S, \Delta_S)=\mathcal{B}_V(M).
\label{eq:1.2.4}
\end{equation}

Next, we recall the characterization of $C_{V}(M) \subset \mathcal{B}_V(M)$ as a point set. 
For this, we prepare some notations. 
A point in $\mathcal{B}_V(M)$ is a pair $(\vec{x}, \{\vec{x}_{B,S}\})$ where $\vec{x} \in M^{\underline{V}}$ and $\vec{x}_{B,S} \in B\ell(M^S, \Delta_S)$ for each $S \subset \underline{V}$ with $|S| \geq 2$. 
We set $\vec{x}_S:=q_S(\vec{x}_{B,S}) \in M^S$, the image of $\vec{x}_{B,S}$ under the blowdonw map $q_S: B\ell(M^{S},\Delta_{S}) \rightarrow M^S$. 
If $\vec{x}_{S} \notin \Delta_S$, $\vec{x}_{B,S}$ coincides with $\vec{x}_{S}$, that is, $\vec{x}_{B,S}= \vec{x}_{S}$. Otherwise, $\vec{x}_{B,S} $ is an element of the fiber of $S\nu_{\Delta_S}$ at $\vec{x}_{S}$, that is, $\vec{x}_{B,S} \in S\nu_{\Delta_S}|_{\vec{x}_{S}}$. When $\vec{x}_{S}=(z,\ldots, z) \in \Delta_S$ for some $z \in M$, we have an isomorphism $\nu_{\Delta_S}|_{\vec{x}_{S}} \simeq (T_zM)^S/T_zM$ where $T_zM$  acts on $(T_zM)^S$ by overall translations. Then, the fiber of sphere bundle $S\nu_{\Delta_S}|_{\vec{x}_{S}}$ is identified with $\mathbb{R}_+\backslash ((T_zM)^S/T_zM - \{0\})$ where the action of $\mathbb{R}_+$ is given by dilations. Its any element is called \textit{screen} for $S$ at $z$ and written as a class $[\vec{u}_S]$ represented by $\vec{u}_S=(u_s)_{s \in S} \in (T_zM)^{S}$ such that not all of $u_s$ are the same. If $M$ is endowed with a Riemannian metric $g^{TM}$, elements of the fiber $S\nu_{\Delta_S}|_{\vec{x}_{S}}$ are given by $\vec{u}_S=(u_s)_{s \in S} \in (T_zM)^{S}$ with norm $||\vec{u}_S||_{g^{TM}} \equiv 1$ and $\sum_{s \in S} u_s =0$. For such $\vec{u}_S=(u_s)_{s \in S}$, it can be regarded as vector $\vec{u}_S \in (T_zM)^{\underline{V}}$ by setting $u_{S,v} = 0$ for $v \notin S$.

As described in \cite[\S 5.2]{AS94},  $C_{V}(M)$ is characterized as the subset of $\mathcal{B}_V(M)$ consisting of points $(\vec{x}, \{\vec{x}_{B,S}\})$ satisfying the following conditions (C1) and (C2):
\begin{enumerate}[(C1)]
\item $\vec{x}_S = \vec{x}|_{S}$ for $S \subset \underline{V}, |S| \geq 2$ where $\vec{x}_S=q_S(\vec{x}_{B,S})$.
\item(Compatibility condition for screens) Suppose that $S' \subset S$ with $|S| \geq 2$, $\vec{x}|_{S} = (z, \ldots, z) \in \Delta_S$, $S'$-components of $u_S$ are not all equal. Then, $[u_{S'}] = [u_S|_{S'}]$
\end{enumerate}

Now we describe the structure of stratified space of $C_{V}(M)$. Recall that subsets $S_1, S_2$ of $\underline{V}$ are called \textit{nested} if they are either disjoint or else one contains the other. Then, $C_{V}(M)$ has the decomposition 
\begin{equation}
	C_V(M) = \bigsqcup_{\mathcal{S}} \partial_{\mathcal{S}} C_V(M)^o
\end{equation}
where $\mathcal{S}$ runs over collections of nested subsets $S$ with $|S| \geq 2$ of $\underline{V}$ and  $\partial_{\mathcal{S}} C_V(M)^o$ is the open strata consisting of the elements $(\vec{x}, \{\vec{x}_{B,S}\}) \in C_V(M)$ satisfying the following conditions (i), (ii) and (iii):
\begin{enumerate}[(i)]
\item $\vec{x}|_S \in \Delta_S$ if and only if $S \subset S'$ for some $S' \in \mathcal{S}$.
\item For a subset $S' \in \underline{V}$ with $|S'|\geq 2$, if $S$ is the smallest set in $\mathcal{S}$ such that $S$ contains $S'$, then $[\vec{u}_{S'}]=[\vec{u}_S|_{S'}]$.
\item For $S_1, S_2 \in \mathcal{S}$ with $S_1 \subsetneq S_2$, all $S_1$-components of $u_{S_2}$ are all the same.
\end{enumerate}

Then, the structure of the manifold with corners on $C_V(M)$ is described by Axelrod--Singer  as follows (\cite[\S5.3, \S 5.4]{AS94}):
\begin{enumerate}[(S1)]
    \item $\partial_{\mathcal{S}} C_V(M)^o$ is a smooth (noncompact) manifold of codimension $|\mathcal{S}|$ in $C_V(M)$, i.e., of dimension $\dim(M)\cdot V - |\mathcal{S}|$, in particular, $M(\emptyset)^o = \Conf_V(M)$.
    \item The closed strata  $\partial_{\mathcal{S}} C_V(M)$, the closure of $\partial_{\mathcal{S}} C_V(M)^o$ in $C_V(M)$,  is given by
    \begin{equation}
        \partial_{\mathcal{S}} C_V(M) = \bigcup_{\mathcal{T} \supseteq \mathcal{S}} \partial_{\mathcal{T}} C_V(M)^o.
    \end{equation}
    \item The codimension-$k$ boundary to $C_V(M)$ denoted by $\partial_k C_V(M)$ is the union of the codimension-$k$ closed strata,
    \begin{equation}
        \partial_k C_V(M) = \bigcup_{\mathcal{S}; |\mathcal{S}|=k} \partial_{\mathcal{S}} C_V(M).
    \end{equation}
    \item $\partial_k C_V(M) \setminus \partial_{k+1}C(V)$ is the open set in the codimension-$k$ boundary $\partial_k C_V(M)$ given by
    \begin{equation}
        \partial_k C_V(M) \setminus \partial_{k+1}C(V) = \bigcup_{\mathcal{S}; |\mathcal{S}|=k} \partial_{\mathcal{S}} C_V(M)^o.
    \end{equation}
\end{enumerate}

Thus, in particular, a codimension $1$ open stratum $\partial_{\mathcal{S}} C_V(M)^o$ for $\mathcal{S}=\{S\}$  is given by the fiber product
\begin{equation}
   \partial_{\mathcal{S}} C_V(M)^o \simeq \Conf_{V - |S| +1}(M) \times_{\Delta_S} \left(\Conf_{|S|}(T M)/TM \rtimes \mathbb{R}_+ \right)_{/\Delta_S}
\end{equation}
where $\Conf_{V - |S| +1}(M) \rightarrow M \simeq \Delta_S$ is $(V - |S| +1)$-th projection  and  $\left(\Conf_{|S|}(T M)/TM \rtimes \mathbb{R}_+ \right)_{/\Delta_S}$ denotes the vector bundle over $\Delta_S$ whose fiber at $(z,\ldots,z) \in \Delta_S$ is given by $\Conf_{|S|}(T_zM)/T_zM \rtimes \mathbb{R}_+$.

Finally, we recall the explicit description of coordinates on $C_V(M)$. Especially, the coordinates of the codimension-1 open stratum are important in the subsequent sections. Taking a Riemannian metric $g^{TM}$, for a collection $\mathcal{S}$ of nested subsets containing at least two elements of $\underline{V}$, we define a map 
\begin{equation}\label{eq:3.1.9}
    \psi: \partial_{\mathcal{S}} C_V(M)^o \times [\bR_{\geq 0}]^{\mathcal{S}} \rightarrow M^V; \quad (c, t)\mapsto (\underline{x}_1(c,t), \ldots, \underline{x}_V(c,t))
\end{equation}
where $c=(\vec{x},\{[\vec{u}_S]; S \in \mathcal{S} \}) \in \partial_{\mathcal{S}} C_V(M)^o$ with $\sum_{i \in S} u_{S,i} =0$ and $||\vec{u}_S|| =1$, and
\begin{equation}
    \begin{split}
        \underline{x}_i(c,t):= \exp_{x_i} \left(\sum_{S \in \mathcal{S}} \left(\prod_{S' \in \mathcal{S}; S' \supset S} t_{S'} \right) u_{S,i} \right) \quad (i=1,\ldots, V).
    \end{split}
\end{equation}

Then, there exist an open neighbourhood $U$ of a point $c^{(0)} \in \partial_{\mathcal{S}} C_V(M)^o \subset C_V(M)$ and an open neighbourhood $W$ of $0 \in [\bR_{\geq 0}]^{\mathcal{S}}$ such that the restriction $\psi_0:=\psi|_{U \times W \setminus \partial W}$ maps $U \times W \setminus \partial W$ diffeomorphicaly onto its image. Moreover, $\psi_0$ extends continuously to a homeomorphism $\psi_B: U \times W \rightarrow \Im(\psi_B) \subset C_V[M]$ onto its image with the following properties:
\begin{itemize}
\item $\psi_B(c, 0) = c$
\item  Let $\mathcal{S}' \subset \mathcal{S}$ be a subset. If $S$-component $t_S$ of $\vec{t}$ is $0$ for all $S \in \mathcal{S}'$, then $\psi_B(c, \vec{t}) \in \partial_{\mathcal{S}'} C_V(M)^0$.
\end{itemize}

Then, it turns out that the set of such maps $\{ (\psi_B, U \times W)\}_{c^{(0)}}$ where $c^{(0)}$ varies over $C_V(M)$ forms a system of coordinates on $C_V(M)$ which gives it a structure of manifolds with corners. This structure is independent of the choice of Riemannian metrics. With this coordinate system, the inclusion map $C_V(M) \hookrightarrow \mathcal{B}_V(M)$ becomes smooth map and $\partial_{\mathcal{S}} C_V(M)^o$ becomes an open subset of smooth part of codimension  $|\mathcal{S}|$ boundary of $C_V(M)$.

Finally, we focus on the case of $V=2$. Let $\Delta\subset 
M\times M$ denotes the diagonal. Then, we have
\begin{equation}
C_2(M) = B\ell(M^{2},\Delta) = (M \times M \setminus \Delta) \cup 
S\nu_{\Delta}
\label{eq:1.2.5}
\end{equation}
with blow down map $q : C_2(M) \rightarrow M^2$ given by identity on the interior and  bundle projection $q_{\partial}:=q|_{S\nu_{\Delta}}: S\nu_{\Delta} \rightarrow \Delta$ on the boundary.

Note that the unit normal bundle $S\nu_{\Delta}$ is identified with the unit tangent bundle $S(TM)$ by
\begin{equation}
	S\nu_{\Delta} \overset{\sim}{\rightarrow} S(TM), \quad ((x,x), 
(v,-v)) \mapsto (x, v).
\label{eq:1.2.6}
\end{equation}
Then $\partial C_2(M)= \partial_{\{\{1,2\}\}}  C_2(M)\simeq S\nu_{\Delta}\simeq S(TM)$. A local coordinate around a point $ ((x,x), (v,-v)) \in \partial C_2(M)$ is given by
\begin{equation}
    \psi:  \partial C_2(M) \times \mathbb{R}_{\geq 0} \rightarrow C_2(M); (((x,x),(v,-v)), t) \rightarrow (\exp_x(tv), \exp_x(-tv)).
\end{equation}

\subsection{Compactification of configuration spaces with embedding information}\label{secion:BT_cmpct}
Next, we recall the Bott-Taubes' way of compactification of configuration spaces of a manifold with disjointly embedded compact submanifolds (\cite[Appendix]{BT94}).

Let $M$ be an oriented closed smooth manifold. Let $N_1, \ldots, N_k \subset M$ are compact submanifolds dosjointly embedded into $M$. Take positive integers $n_1, \ldots, n_k$. Then, there is the canonical map
\begin{equation}\label{eq:3.2.1}
	\iota : \prod_{i=1}^k C_{n_i}(N_i) \rightarrow C_{\sum_{i=1}^k n_i}(M).
\end{equation}
For $t \geq 0$, let $C_t(M; (N_1, n_1),\ldots, (N_k, n_k))$ denote the pullback bundle of the projection map $\pi: C_{\sum_{i=1}^k n_i + t}(M) \rightarrow C_{\sum_{i=1}^k n_i}(M)$ forgetting the last $t$ points by $\iota$ given in \eqref{eq:3.2.1}:
\begin{equation}\label{eq:3.2.2}
	\begin{tikzcd}
		C_t(M; (N_1, n_1),\ldots, (N_k, n_k))  \arrow[r] \arrow[d] & C_{\sum_{i=1}^k n_i + t}(M) \arrow[d, "\pi"] \\
		 \prod_{i=1}^k C_{n_i}(N_i) \arrow[r, "\iota"] & C_{\sum_{i=1}^k n_i}(M).
	\end{tikzcd}
\end{equation}
\begin{proposition}{\cite[Proposition A.3]{BT94}} \label{prop:BT}
	The space $C_t(M; (N_1, n_1),\ldots, (N_k, n_k))$ is a compact manifold with corners with the projection induced smooth maps to $C_{\sum_{i=1}^k n_i + t}(M)$ and $\prod_{i=1}^k C_{n_i}(N_i)$. In fact, the map $C_t(M; (N_1, n_1),\ldots, (N_k, n_k)) \rightarrow \prod_{i=1}^k C_{n_i}(N_i)$ is strata compatible (for definition, see loc.cit.). Furthermore, if a smooth manifold $\mathcal{K}$ parametrizes a family of embeddings of $\bigsqcup_{i=1}^k N_i \hookrightarrow M$, then the parametrized version of $C_t(M; (N_1, n_1),\ldots, (N_k, n_k))$ is a smooth manifold with corners which fibers over $\mathcal{K}$.
\end{proposition}

\subsection{Compactification of configuration spaces of a closed surface with marked points}\label{ssecion:3.3}
Now, we explain a version of compactified configuration spaces of closed surfaces with marked points mainly used in the present article.

Let us apply the procedure in Section \ref{secion:BT_cmpct} to our situation with a little modification. Let $X$ be a closed surface and $S := \{s_0, s_1,\ldots, s_k\} \subset X$ be a finite set of distinct points on $X$. Take a distinguished point $s_0 \in S$ and a tangent vector $v \in T_{s_0} X$ at $s_0$. Then, we define a  compactification $C_2(X)_{(s_0, v)}$ of two point configuration space $\Conf_2(X)$ with respect to $(s_0,v)$ as the pullback of $\pi_3: C_3(X) \rightarrow X$ by the canonical embedding
\begin{equation}
	\iota: \{s_0\} \hookrightarrow X
\end{equation}
with the distinguished element $v$ on the boundaries corresponding to $x_1=s_0$ and $x_2 = s_0$. If there is no fear of confusion, we often denote $C_2(X)_{(s_0,v)}$ by $C_2(X)_{s_0}$ or $C_2(X)$ for simplicity. Similarly, we define $C_n(X;(s_1,\ldots, s_k))_{(s_0,v)}$ as the pullback of $\pi: C_{|S| + 1 + n}(X) \rightarrow C_{|S|}(X)$ by the canonical embedding
\begin{equation}
	\iota_{S,(s,v)}: \{s_1\} \times \cdots \{s_k\} \times \{s_0\} \hookrightarrow C_{|S|+1}(X)
\end{equation}
with the distinguished element $v$ on the boundaries corresponding to $x_i=s$ $(1\leq i \leq n)$.

Then, by Proposition \ref{prop:BT}, we deduce the following.
\begin{corollary}
\begin{enumerate}[(1)]
\item The space $C_n(X;(s_1,\ldots, s_k))_{(s_0,v)}$ is compact manifolds with corners. Moreover, the induced map $C_n(X;(s_1,\ldots, s_k))_{(s_0,v)} \rightarrow \{s_1\} \times \cdots \{s_k\} \times \{s_0\} $ is a fibration map. 
\item Furthermore, if we consider a family of compact Riemann surfaces $X \rightarrow B$ with a smooth divisor $S \subset X$ in $X$ over $B$. Let $s_0, s_1,\ldots, s_n:B \rightarrow X$ be a section corresponding to a finite number of intersections of $S$ with $X$ at each fiber. Then, parametrized version of $C_n(X;(s_1,\ldots, s_k))_{(s_0,v)}$ is a smooth manifold with corners which fibers over $B$.
\end{enumerate}
\end{corollary}

\begin{proof}
(1) The first statement is just a translation of Proposition \ref{prop:BT}. The second one follows from the fact that $C_n(X;(s_1,\ldots, s_k))_{(s_0,v)} \rightarrow \{s_1\} \times \cdots \{s_k\} \times \{s_0\} $ is smooth by 	Proposition \ref{prop:BT} and $\{s_1\} \times \cdots \{s_k\} \times \{s_0\} $ is 0-dimensional manifold. 

\noindent
(2) The parametrized version of (1) follows immediately by considering \eqref{eq:3.2.2} with fibers over $B$.
\end{proof}

\begin{remark}\label{rem:3.3.2}
    The space $C_n(X;(s_1,\ldots, s_k))_{(s_0,v)}$ may be regarded as $n$-point configuration space of punctured Riemann surface $X \setminus S$ where $s_0$ plays a role of the point at  infinity. For succinctness, $C_n(X;(s_1,\ldots, s_k))_{(s_0,v)}$ is often denoted by $C_n(X;(s_1,\ldots, s_k))_{s_0}$ or $C_n(X;(s_1,\ldots, s_k))$. Similarly, it's interior or boundary open stratum is denoted by $\Conf_n(X;(s_1,\ldots, s_k))_{(s_0,v)}$, $\Conf_n(X;(s_1,\ldots, s_k))_{s_0}$, or $\Conf_n(X;(s_1,\ldots, s_k))$.
\end{remark}

\subsection{Boundary Strata of compactified configuration spaces}\label{section:boundary_strata}
With the same notations as the previous section, we recall that the codimension-1 open strata $\partial  C_n(X;(s_1,\ldots, s_k))_{(s_0,v)}^0$ is classified into the following three cases:
\begin{itemize}
    \item \textit{principal strata}: strata corresponding to collisions of exactly two points away from $s_0$.
    \item \textit{hidden strata}: strata corresponding to collisions of more than two points away from $s_0$.
    \item \textit{infinite strata}: strata corresponding to collisions of points with $s_0$.
\end{itemize}

\begin{remark}
    In \cite{BT94}, anomalous strata (anomalous face in loc.cit.), corresponding to the collision of all points, are considered separately from hidden strata. In our case, different from their case (3-dimensional case), anomalous strata do not yield any problem and we can treat them as a special case of hidden strata. Therefore, we use the classification of the boundary strata of compactified configuration spaces as above.
\end{remark}

\subsection{Smooth extension of differentials of Green functions to compactified configuration spaces}\label{ssecion:3.5}
Here, we consider an extension of some differentials of Green functions to compactified configuration spaces as in \cite{AS94} where this type of consideration first appeared. We use the same notation and setting as Section \ref{section:2.2}.

\begin{theorem}\label{thm:ext_green}
Let $X$ be a compact Riemann surface and $G_{\mathrm{Ar}}(x,y) \in \mathcal{D}^0(X^2)$ be the Arakelov Green function for some normalized volume element on $X$. We set $d^{\bC}:= \partial - \bar{\partial}: \Omega^{\bullet}(X^2) \rightarrow \Omega^{\bullet +1}(X^2)$. Then, the smooth differential one form $d^{\mathbb{C}}  G_{\mathrm{Ar}}(x,y) \in \Omega^1(\Conf_2(X))$ extends smoothly to the differential one form $\widetilde{d^{\mathbb{C}} G_{\mathrm{Ar}}(z,w)} \in \Omega^1(C_2(X))$ on $C_2(X)$ with following boundary condition
\begin{equation}
	\frac{1}{2}\widetilde{d^{\mathbb{C}} G_{\mathrm{Ar}}(z,w)}|_{\partial C_2(M)} = \omega + q^{\ast}_{\partial} (\psi)
\end{equation}
where $\omega \in \Omega^1(S(TX))$ is a fiberwise volume form of $S^1$, $\psi \in \Omega^1(X)$ is a one form on $X$, and $q_{\partial}: \partial C_2(X) \simeq S(TX) \rightarrow X$ is the blowdown map restricted to the boundary $\partial C_2(X)$.
\end{theorem}

\begin{proof}
To begin with, we prepare some notations for coordinates near the diagonal. Since $X$ is a Riemann surface, we can assume that $X$ is endowed with a conformal metric $g^{TX}$ compatible with the holomorphic structure. For a tangent vector $v$ of $T_x X$ at $x \in X$, we denote by $||v||$ the norm of $v$ with respect to $g^{TX}$, and for $x,y \in X$, we denote by $d(x,y)$ the distance between $x$ and $y$ with respect to $g^{TX}$.  Let us take $\epsilon>0$ to be smaller than the injectivity radius of $TX$ and let
\begin{equation}
	N(\epsilon) = \{(z,u) \in TX \ | \ ||u|| < \epsilon \} \subset TX
\end{equation}
be a subbundle of $TX$ whose fibers consist of tangent vectors with length less than $\epsilon$. We define a map
\begin{equation}
	E: N(\epsilon)  \longrightarrow X \times X; \quad (z,u) \mapsto (\exp_z(u), \exp_z(-u))
\end{equation}
which is a diffeomorphism onto its image containing $\Delta_X$.  Putting
\begin{equation}
	N(0) = \{(z,u) \in TX \ | \ u=0 \} \subset TX,
\end{equation}
the restriction $E|_{N(\epsilon)\setminus N(0)}$ of $E$ to $N(\epsilon)\setminus N(0)$ induces a diffeomorphism
\begin{equation}
	E|_{N(\epsilon)\setminus N(0)} : N(\epsilon)\setminus N(0)  \overset{\sim}{\longrightarrow} E(N(\epsilon)) \setminus \Delta_X \subset \Conf_2(X) \subset X \times X
\end{equation}

Note that one can identify $N(\epsilon)\setminus N(0)$ with $S(TX) \times (0,\epsilon)$ by 
\begin{equation}
	N(\epsilon)\setminus N(0)  \overset{\sim}{\longrightarrow} S(TX) \times (0,\epsilon); \quad (z,u) \mapsto \left(\left(z, \frac{u}{||u||}\right), ||u||\right)
\end{equation}

Using local coordinates $z=(z^1, z^2)$ on an open subset $z \in U$ in $X$, we define local coordinates on $N \cap TU$ via
\begin{equation}
	N \cap TU \rightarrow \mathbb{R}^2 \times \mathbb{R}^2;\quad (z,u) \mapsto (z^1,z^2, u^1, u^2)
\end{equation}
where $u = u^1 \frac{\partial}{\partial z^1} + u^2 \frac{\partial}{\partial z^2}$.

Recall that in a holomorphic coordinate near the diagonal the Green function is expressed as the sum of the singular part and the smooth part:
\begin{equation}
	G_{\mathrm{Ar}} (x,y) =  \frac{1}{2\pi \sqrt{-1}} \cdot  \log\left(\frac{1}{|x-y|}\right) + \eta 
\end{equation}
where $\eta$ is a smooth function on $X \times X$. We assume that the metric is written as $g^{TX} = \lambda(x,y) (dx^2 + dy^2)$ in this coordinate in terms of a positive smooth function $\lambda$.  Then, using the coordinates defined above, $(x,y)=((x^1, x^2), (y^1,y^2))$ can be regarded as a function on $(z,u)$. 

By direct computation, the singular part of $d^{\mathbb{C}} G_{\mathrm{Ar}}(x,y)$ is given by
\begin{align}
	d^{\mathbb{C}} G_{\mathrm{Ar}}(x,y) =  \frac{1}{\pi} \cdot \left( \frac{(x^1 - y^1)(dx^2 - dy^2)}{|x-y|^2} - \frac{(x^2 - y^2)(dx^1 - dy^1)}{|x-y|^2} \right) + d^{\mathbb{C}} \eta
\end{align}
where we used the following equation:
\begin{equation}
\begin{split}
d^{\mathbb{C}}_z & = \partial_z - \partial_{\bar{z}} \\
	& = \frac{1}{2}\left(\frac{\partial}{\partial x} - \sqrt{-1} \frac{\partial}{\partial y} \right) \otimes e(dx + \sqrt{-1} dy) -  \frac{1}{2}\left(\frac{\partial}{\partial x} + \sqrt{-1} \frac{\partial}{\partial y} \right) \otimes e(dx - \sqrt{-1} dy)\\
	&= - \sqrt{-1} \left( \frac{\partial}{\partial y}\otimes e(dx) - \frac{\partial}{\partial x} \otimes e(dy)\right).
 \end{split}
\end{equation}
Here, $e$ denotes the exterior product operation. Next, after pulling back this form along the map
\begin{equation}
	F : S(TX) \times (0,\epsilon) \overset{\sim}{\rightarrow} N(\epsilon) \setminus N(0) \rightarrow   \Conf_2(X); ((z,v),r) \mapsto (\exp_z(rv), \exp_z(-rv)),
\end{equation}
we check whether  $F^{\ast}(d^{\mathbb{C}} G_{\mathrm{Ar}}(x,y))$ can extend smoothly on $S(TX) \times [0, \epsilon)$, keeping in mind the following commutative diagram:
\begin{equation}
\begin{tikzcd}
S(TX) \times (0, \epsilon)  \arrow[r, hook] \arrow[d, "F"] & S(TX) \times [0, \epsilon) \arrow[r] & B\ell(X^2, \Delta_X) \supset   \partial C_2[X]\simeq S(TX) \arrow[d, "q"] \\
\Conf_2(X) & &  X^2 \supset \Delta_X \simeq X
\end{tikzcd}
\end{equation}

Since $|x-y| \overset{x, y \to z}{\sim}  \frac{1}{\lambda(z)}( d(z,x) + d(z,y)) = \frac{2}{\sqrt{\det(g_{pq}(z))}} ||rv||_z$ (the equality follows from \cite[Lemma 2.3.A.3]{Jos06}), the pullback of the singular part of $d^{\mathbb{C}} G_{\mathrm{Ar}}(x,y)$ becomes
\begin{equation}
	\frac{1}{\pi} \cdot  \sqrt{\det(g_{pq}(z))} \cdot \left( \frac{F^{\ast}(x^1 - y^1) F^{\ast}(dx^2 - dy^2)}{4 r^2||v||^2} - \frac{F^{\ast}(x^2 - y^2) F^{\ast}(dx^1 - dy^1)}{4 r^2 ||v||^2} \right).
\end{equation}

Since 
\begin{equation}
	\frac{\partial}{\partial u} x(z, 0) = \frac{\partial}{\partial u} y(z, 0) =  \Id
\end{equation}
and 
\begin{equation}
	\frac{\partial}{\partial z} \left( x(z, 0) - y(z,0)\right) = 0,
\end{equation}
we see that $F^{\ast}(dx^i - dy^i) \sim 2 d(rv^i)$ near $r=0$ and
\begin{equation}
\lim_{r \to 0} \frac{F^{\ast}((x-y)}{r} = \lim_{r \to 0} \frac{(\exp_z(rv) - \exp_z(-rv))}{r} = 2 v
\end{equation}
(cf. \cite[Lemma 2.3.A.2]{Jos06}). Therefore, we finally get the pullback of the singular part of $id^{\mathbb{C}} G_{\mathrm{Ar}}(x,y)$ near $r=0$ as
\begin{equation}
	 \frac{1}{\pi} \cdot   \sqrt{\det(g_{pq}(z))} \left(v^1 dv^2 - v^2 dv^1\right) + O(r)
\end{equation}
and hence we conclude that $F^{\ast}(d^{\mathbb{C}} G_{\mathrm{Ar}}(x,y))$ extends smoothly on $S(TX) \times [0,\epsilon)$. The boundary value of this part is given by its restriction to $S(TX) \times \{0\}$ as follows:
\begin{equation}
\frac{1}{2\pi} \cdot   \sqrt{\det(g_{pq}(z))} \left(v^1 dv^2 - v^2 dv^1\right)=:\omega  \in \Omega^1(S(TX))
\end{equation}
and this form is a fiberwise volume form of $S^1$ normalized so that $\int_{S(T_zX)} \omega = 1$.
Thus, the whole boundary value is given as
\begin{equation}\label{eq:extension}
\begin{split}
	 \frac{1}{2}  \lim_{r \to 0} F^{\ast}\left(d^{\mathbb{C}} G_{\mathrm{Ar}}(x,y)\right) & =  \frac{1}{2} F^{\ast}\left(d^{\mathbb{C}} G_{\mathrm{Ar}}(x,y)\right)|_{S(TX) \times \{0\}} \\
  & = \omega +  q^{\ast}_{\partial}\psi
  \end{split}
\end{equation}
where we put $\psi:= (\frac{1}{2} d^{\mathbb{C}}\eta)|_{\Delta_X}\in \Omega^1(\Delta_X) \simeq \Omega^1(X)$. Finally, note that we give a proof of \eqref{eq:extension} for the chosen local coordinate but this property itself does not depend on the choice of the coordinate system. This completes the proof.
\end{proof}

\begin{corollary}\label{cor:extended_prop}
Let $X$ be a compact Riemann surface with a tangential base point $v \in T_{s_0}X$ at $s_0$ for some point $s_0 \in X$. Let $C_2(X)_{(s_0,v)}$ denote the compactification of $\Conf_2(X)$ with respect to $(s_0,v)$. Then, the smooth differential one form $d^{\mathbb{C}}  G_v(x,y) \in \Omega^1(\Conf_2(X))$ extends smoothly to the differential one form $\widetilde{d^{\mathbb{C}} G_v(z,w)} \in \Omega^1(C_2(X)_{(s_0, v)})$ with following boundary condition:

\begin{enumerate}[(1)]
\item  On the boundary component $\partial_{\{\{1,2\}\}} C_2(X)_{(s_0, v)}\simeq S(TX)$ corresponding to $x_1 = x_2$,
\begin{equation}
	\frac{1}{2}  \widetilde{d^{\mathbb{C}} G_v(z,w)}|_{\partial_{\{\{1, 2\}\}} C_2(X)} = \omega +  q^{\ast}|_{\partial_{\{\{1, 2\}\}} C_2(X)} (\psi)
\end{equation}
where $\omega \in \Omega^1(S(TX))$ is a fiberwise normalized volume form of $S^1$ and $\psi \in \Omega^1(X)$ is a one form on $X$. 
\item On the boundary component $\partial_{\{\{i, s_0\}\}} C_2(X)_v \simeq S(T_{s_0}X)$ corresponding to $x_i = s_0$ for $i=1,2$,
\begin{equation}
	\frac{1}{2} \widetilde{d^{\mathbb{C}} G_v(z,w)}|_{\partial_{\{\{i, s_0\}\}} C_2(X)} = \vol_{S^1} 
\end{equation}
where $\vol_{S^1} \in \Omega^1(S(T_{s_0} X))\simeq\Omega^1(S^1) $ is a volume form of $S^1$.\end{enumerate}
Moreover, the differential of $\widetilde{d^{\mathbb{C}} G_v(z,w)}$ is given by
\begin{equation}
	\frac{1}{2} d \widetilde{d^{\mathbb{C}} G_v(z,w)} =  - \frac{\sqrt{-1}}{2} \left(\sum_{k=1}^g p_1^{\ast}\alpha_k \wedge p_2^{\ast}\bar{\alpha}_k + p_2^{\ast} \alpha_k \wedge p_1^{\ast}\bar{\alpha}_k \right).
\end{equation}

\end{corollary}

\begin{proof}
This follows immediately by applying Theorem \ref{thm:ext_green} to $G_{s_0} (z,w) = G_{\mathrm{Ar}}(z,w) -  G_{\mathrm{Ar}}(s_0,w) -  G_{\mathrm{Ar}}(z,s_0) + C$. For (ii), note that $q^{\ast}|_{\partial_{\{\{i, s_0\}\}} C_2(X)} (\psi|_{x_i=s_0}) = 0$ for any $\psi \in \Omega^1(X)$.
\end{proof}

\begin{remark}
\begin{enumerate}[(1)]
\item From Corollary \ref{cor:extended_prop}, $\frac{1}{2} d^{\mathbb{C}} G_{s_0}(z,w)$ can be regraded as a 2-dimensional Chern--Simons propagator for $X \setminus\{s_0\}$. More precisely, it may be thought of as a 2-dimensional propagator for surfaces punctured at $s_0$ obtained from those for closed surfaces given in \cite{CW} by a similar trick used in \cite{BC} to construct a closed propagator aimed at defining invariants for rational homology 3-spheres.
\item As mentioned in \cite[Page 33]{Jos06}, in general, the exponential map is not holomorphic and only the differentiable structure of $X$ is preserved. However, we will use the exponential map and compactification of configuration spaces of $X$ just to argue the convergence of integrals of differential forms. Therefore, it does not cause any trouble in most of the present article.
\end{enumerate}
\end{remark}

\section{Graph complex}\label{section:4}
This section introduces an appropriate graph complex for our purpose of generalizing the Hodge correlator twistor connection.
\subsection{(Decorated) Graphs}\label{section:4.1}
In this section, we prepare some terminologies of graphs we use in the present article. Throughout this article, graphs are always assumed to be \textit{finite} graphs, i.e., the sets of vertices and edges are both finite sets. 

\begin{definition}
\begin{enumerate}[(1)]
    \item A \textit{uni-trivalent graph} $\Gamma$ is a graph with vertices whose valencies are univalent or trivalent. Sometimes, a uni-trivalent graph is also called a \textit{trivalent} graph for simplicity.
    \item An edge which connects a single vertex is called \textit{self-loop}. If two distinct vertices are connected by exactly one edge, then this edge is said to be \textit{regular}. An edge of a connected graph is called \textit{cycle edge} if the graph remains connected after removing the edge.   
    \item For a graph $\Gamma$, $V_{\Gamma}$ and $E_{\Gamma}$ denote the sets of  vertices and edges of $\Gamma$ respectively.
    \item For a graph $\Gamma$, a non-univalent (resp. univalent) vertex of $\Gamma$  is  also called \textit{internal (resp. external) vertex} of $\Gamma$. The set of internal (resp. external) vertices of $\Gamma$ is denoted by $V^{\mathrm{int}}_{\Gamma}$ (resp. $V^{\mathrm{ext}}_{\Gamma}$) and $V_{\Gamma}$ has the decomposition $V_{\Gamma} = V_{\Gamma}^{\mathrm{int}} \sqcup V_{\Gamma}^{\mathrm{ext}}$.
    \item  For a graph $\Gamma$,  an \textit{external edge} is an edge connecting a univalent vertex and an internal vertex, and an \textit{internal edge} is an edge connecting two internal vertices. Then, there is a decomposition $E_{\Gamma} = E^{\mathrm{ext}}_{\Gamma} \sqcup E_{\Gamma}^{\mathrm{int}}$ where $E^{\mathrm{ext}}_{\Gamma}$ is the set of external edges of $\Gamma$ and $E_{\Gamma}^{\mathrm{int}}$ is the set of internal edges of $\Gamma$.
    \item A graph $\Gamma$ is said to be \textit{ordered} if $E_{\Gamma}$ is ordered.
    \item For a graph $\Gamma$, its \textit{number of loops} $l(\Gamma)$ is defined as the first Betti number of (geometric realization of) $\Gamma$.
\end{enumerate}
\end{definition}

\begin{remark}
    For a connected graph $\Gamma$, the Euler characteristic of $\Gamma$ implies $|V_{\Gamma}| - |E_{\Gamma}| = 1 - l(\Gamma)$.
\end{remark}

\begin{definition}{(decoration of a graph)}
Let $\Gamma$ be a graph and $R$ be a set.  A \textit{decorated graph} is a graph $\Gamma$ equipped with a map $V_{\Gamma}^{\mathrm{ext}} \rightarrow R$ called \textit{decoration} or \textit{$R$-decoration}.
\end{definition}

\begin{example}
Let $X$ be a compact Riemann surface with a non-zero tangent vector $v$ at a point $s_0 \in X$. Let $S^{\ast}=\{s_1, \ldots, s_k\}$ be a collection of distinct points of $X$ different from $s_0$. As in Section \ref{subsection:2.1}, put
\begin{equation}
	V^{\vee}_{X, S^{\ast}} = \mathbb{C}[S^{\ast}] \oplus (\Omega^1_X \oplus \overline{\Omega}^1_X).
\end{equation}
Then, one can consider $V^{\vee}_{X, S^{\ast}}$-decorated graphs. Let $\Gamma$ be a graph with $V^{\vee}_{X, S^{\ast}}$-decoration given by $\mathsf{d}: V_{\Gamma}^{\ext} \rightarrow V^{\vee}_{X, S^{\ast}}$. Now, the direct sum decomposition of $V^{\vee}_{X, S^{\ast}}$ induces the associated decomposition of $V_{\Gamma}^{\ext}$ given by
\begin{equation}\label{decomposition_ext}
    V_{\Gamma}^{\ext} = V_{\Gamma}^{\ext, S^{\ast}} \sqcup V_{\Gamma}^{\ext, \mathrm{sp}}
\end{equation}
where we set
\begin{equation}
V_{\Gamma}^{\ext, S^{\ast}}:= \mathsf{d}^{-1}(\mathbb{C}[S^{\ast}]),\quad   V_{\Gamma}^{\ext, \mathrm{sp}}:= \mathsf{d}^{-1}(\Omega^1_X \oplus \overline{\Omega}^1_X).
\end{equation}
Elements of $V_{\Gamma}^{\ext, S^{\ast}}$ are called \textit{$S^{\ast}$-decorated} external vertices and those of $V_{\Gamma}^{\ext, \mathrm{sp}}$ are called \textit{special decorated} external vertices. Several examples of $V^{\vee}_{X, S^{\ast}}$-decorated trivalent graphs are given in Figure \ref{fig:4.1.1}. 

\begin{figure}[h]
\captionsetup{margin=2cm}
\begin{center}
\begin{tikzpicture}
\node at (0,-1.8) {$\Gamma_1$};
\node at (5,-1.8) {$\Gamma_2$};
\node at (10,-1.8) {$\Gamma_3$};
\begin{scope}
	\coordinate (a3) at (1.5*0.707106, -1.5*0.707106) node [below right] at (a3) {$\alpha_1$};
	\coordinate (a0) at (-1.5*0.707106, -1.5*0.707106) node [below left] at (a0) {$s$};
	\coordinate (a1) at (-1.5*0.707106, 1.5*0.707106) node [above left] at (a1) {$s$};
	\coordinate (a2) at (1.5*0.707106, 1.5*0.707106) node [above right] at (a2) {$s'$};
	\node at (a0) {$\bullet$};
	\node at  (a1) {$\bullet$};
	\node at (a2) {$\bullet$};
	\node at  (a3) {$\bullet$};
\coordinate (y) at (0.6,0) node [right] at (y) {};
\coordinate (x) at (-0.6,0) node [left] at (x) {};
\node at (x) {$\bullet$};
\node at (y) {$\bullet$};
\draw[densely dotted] (x) -- (y);
\draw[densely dotted] (x) -- (a0);
\draw[densely dotted] (x) -- (a1);
\draw[densely dotted] (y) -- (a2);
\draw[densely dotted] (y) -- (a3);
\end{scope}

\begin{scope}[xshift=5cm]
	\coordinate (a0) at (-1.5,0) node [left] at (a0) {$s$};
	\coordinate (a1) at (1.5,0) node [right] at (a1) {$s$};
	\node at  (a0)  {$\bullet$};
	\node at (a1) {$\bullet$};
	\draw[densely dotted] (0,0) circle (0.8cm);
	\draw[densely dotted] (a0) -- (-0.8,0);
	\draw[densely dotted] (0.8,0) -- (a1);
	\draw[densely dotted] (0,0.8) -- (0, -0.8);
 \node at (-0.8, 0) {$\bullet$};
 \node at (0.8, 0) {$\bullet$};
 \node at (0, 0.8) {$\bullet$};
 \node at (0, -0.8) {$\bullet$};
\end{scope}

\begin{scope}[xshift=10cm]
	\coordinate (a1) at (0, 1.5) node [above] at (a1) {$s$};
	\coordinate (a0) at (-1.50*.8660254, -1.5*0.5) node [below left] at (a0) {$\alpha_1$};
	\coordinate (a2) at (1.50*.8660254, -1.5*0.5) node [below right] at (a2) {$\bar{\alpha}_1$};
	\node at (a1) {$\bullet$};
	\node at  (a0) {$\bullet$};
	\node at (a2) {$\bullet$};
 \draw[densely dotted] (a0) -- (a2);
	\draw[densely dotted] (0,-0.75) -- (0, -0.15);
	\draw[densely dotted] (0,0.85) -- (a1);
	\draw[densely dotted] (0, 0.35) circle (0.5);
 \node at (0, -0.75) {$\bullet$};
 \node at (0, -0.15) {$\bullet$};
 \node at (0, 0.85) {$\bullet$};
\end{scope}
\end{tikzpicture}	
\end{center}
\caption[Examples of $V^{\vee}_{X, S^{\ast}}$-decorated trivalent graphs] {$V^{\vee}_{X, S^{\ast}}$-decorated trivalent graphs. Here, decorations are given by $s, s' \in S^{\ast}$, $\alpha_1 \in \Omega^1_X$  and $\overline{\alpha}_1 \in \overline{\Omega}^1_X$}
\label{fig:4.1.1}
\end{figure}
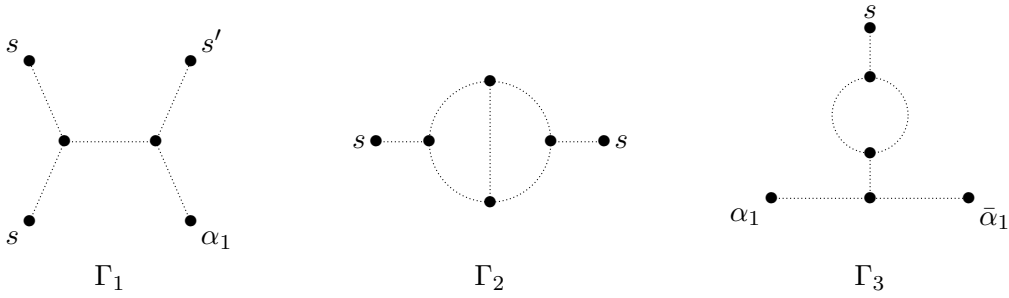
\end{example}

\begin{convention}\label{conv:sym_dec}
    For connected $R$-decorated graphs $\Gamma_1, \Gamma_2$ with decorations $W_1, W_2:V_{\Gamma}^{\mathrm{ext}} \rightarrow R$ respectively. Then, we have the induced $R$-decoration on the disjoint union $\Gamma_1\sqcup \Gamma_2$ of $\Gamma_1$ and $\Gamma_2$. We denote it by $W_1 \sqcup W_2$ and call it the \textit{symmetric product} of $W_1$ and $W_2$. Note that $W_2 \sqcup W_1 = W_1 \sqcup W_2$ as a map $V_{\Gamma_1}^{\mathrm{ext}}\sqcup V_{\Gamma_2}^{\mathrm{ext}}  \rightarrow R$.
\end{convention}

\subsection{Automorphism group of (Feynman) graphs}\label{ss:automorphism_graph}
Here, we recall some basics of the notion of the automorphism group of a Feynman graph. Roughly speaking,  a Feynman graph is a graph represented as partitions of the set of half-edges. Such partitions correspond to the edge set and the vertex set. The automorphisms of a Feynman graph are permutations of half-edges preserving such partitions.\footnote{The authors learned the definitions of Feynman graphs and their automorphisms given in the present article from Hiroyuki Fuji and Akishi Kato (indirectly). Besides, the version here is partially based on the work \cite[\S 6.6]{KL23} with Bingxiao Liu.}

The following definition of a Feynman graph is not a standard one but is a version adapted for our later usage.

\begin{definition}[Feynman graph]
\begin{enumerate}[(1)]

\item  Let $n$ be a non-negative integer. For each $k \geq 3$, we set $n_k$ be a non-negative integer and suppose that $n_k =0$ for all $k \geq 3$ but only finitely many $k$. Let $h_i$ be a set of $\sum_{k \geq 3} n_k$ elements, called the \textit{set of internal half-edges} and $h_e$ be a set of $n$-elements, called the \textit{set of external half-edges}. Then, we consider partitions $P$ of $h:= h_i \cup h_e$ as follows:
    \begin{itemize}
        \item a partition into pairs of half-edges which we call \textit{edges} and denote it by $E(h,P)$.
    \item a partition of $h_i$ into $n_k$ sets of cardinality $k$ which we call \textit{internal vertices} with valency $k$ $(k \geq 3)$ and denote it by $V^{\mathrm{int}}(h,P)$,
    \item a partition of $h_e$ into $n$ sets of cardinality $1$ which we call \textit{external vertices} with valency $1$ and denote it by $V^{\mathrm{ext}}(h,P)$.
    \end{itemize}
The set $E(h,P)$ is further decomposed into two disjoint sets $E(h,P)=E^{\mathrm{int}}(h,P)\sqcup E^{\mathrm{ext}}(h,P)$ where $E^{\mathrm{int}}(h,P)$ consists of pairs of internal half-edges and $E^{\mathrm{ext}}(h,P)$ consists of pairs of internal and external half-edges or those of external half-edges.

From such partitions, we can obtain a topological graph $\Gamma(h, P)$ by assigning pairs of half-edges in $E(h,P)$ to edges and sets of cardinality $k$ in  $V^{\mathrm{int}}(h,P) \cup  V^{\mathrm{ext}}(h,P)$ to vertices of valency $k$ respectively. If $\Gamma(h_1, P_1)$, $\Gamma(h_2, P_2)$ are two topological graphs, they are said to be equivalent to each other if there is a bijection between $h_1$ and $h_2$ which maps the partitions $P_1$ of $h_1$ to the ones $P_2$ of $h_2$. Then, such an equivalence class is called a \textit{Feynman graph}.

\item With the same situation as (1), in addition to the partitions $P$ of $h$, we further impose information of an $R$-decoration for a set $R$, i.e., a map $W:V^{\mathrm{ext}}(h,P) \rightarrow R$.
For two  $\Gamma(h_1, P_1)$, $\Gamma(h_2, P_2)$ topological graphs with information of an $R$-decoration of external vertices as above, they are said to be equivalent to each other if there is a bijection between $h_1$ and $h_2$ which maps the partitions $P_1$ of $h_1$ to the ones $P_2$ of $h_2$ preserving the $R$-decoration of the set of external vertices. Then, such an equivalence class is called a \textit{decorated Feynman graph}.

\end{enumerate}
\end{definition}

\begin{definition}[Automorphism group of a (Feynman) graph]\label{def:auto_feynman}
\begin{enumerate}[(1)]
\item Let $\Gamma$ be a Feynman graph represented by $\Gamma(h, P)$. An \textit{automorphism} of $\Gamma$ is a permutation of $h$ preserving its partitions $P$. We denote by $\Aut(\Gamma)$  the group of automorphisms of $\Gamma$.
\item Similarly, let $R$ be a set and let $\Gamma$ be an $R$-decorated Feynman graph represented by $\Gamma(h, P)$. An \textit{automorphism} of $\Gamma$ is a permutation of $h$ preserving its partitions $P$ and the $R$-decoration of external vertices. We denote by $\Aut(\Gamma)$  the group of automorphisms of $\Gamma$.
\end{enumerate}
\end{definition}

Note that an automorphism of a Feynman graph is well-defined, i.e., it is independent of the choice of representatives as a pair $(h, P)$. In the sequel, a Feynman graph and the automorphism of a (decorated) Feynman graph will be referred to as just a (topological) graph and the automorphism of a (decorated) graph respectively.
\begin{convention}
    For a $V_{X, S^{\ast}}^{\vee}$-decorated graph $\Gamma$, the underlying decorated Feynman graph $\Gamma(h,P)$ for $\Gamma$ is taken so that the decoration of the set of external vertices $V^{\mathrm{ext}}_{\Gamma}$ in \eqref{decomposition_ext} corresponds to that of $V^{\mathrm{ext}}(h,P)$ in Definition \ref{def:auto_feynman}(2).
\end{convention}

\begin{example}
Let us consider decorated uni-trivalent graphs given in Figure \ref{fig:4.1.1}. Then, the numbers of automorphism groups $|\Aut(\Gamma_i)|$ for decorated graphs $\Gamma_i$ in Figure \ref{fig:4.1.1} are given by
\begin{equation}
    |\Aut(\Gamma_1)| = 2, |\Aut(\Gamma_2)| = 4, |\Aut(\Gamma_3)| = 2.
\end{equation}
\end{example}

\subsection{Orientation torsor of graphs}\label{section:4.3}
We recall the notion of the orientation torsor of graphs following \cite[\S 3.2]{GoncharovHodge1}.

For a finite set $\mathcal{X}=\{x_1, \ldots, x_{|\mathcal{X}|}\}$, there is a $\mathbb{Z}/2\mathbb{Z}$-torsor $\mathrm{or}_{\mathcal{X}}$ called the \textit{orientation torsor} whose elements are expressions like
\begin{equation}
    \pm x_1 \wedge \cdots \wedge x_{|\mathcal{X}|}
\end{equation}
subject to the relation, for any permutation $\sigma \in S_{|\mathcal{X}|}$ on the set $\mathcal{X}$,
\begin{equation}
    \sign(\sigma) x_{\sigma(1)} \wedge \cdots \wedge x_{\sigma(|\mathcal{X}|)} = x_1 \wedge \cdots \wedge x_{|\mathcal{X}|}.
\end{equation}
The orientation torsor $\mathrm{or}_{\mathcal{X}}$ consists of two equivalent classes represented respectively by 
\begin{equation}
    x_1 \wedge \cdots \wedge x_{|\mathcal{X}|},\quad \text{and} \quad - x_1 \wedge \cdots \wedge x_{|\mathcal{X}|}.
\end{equation}
For a connected graph $\Gamma$, the \textit{orientation torsor} $\mathrm{or}_{\Gamma}$ of $\Gamma$ is defined as the orientation torsor $\mathrm{or}_{E_{\Gamma}}$ of $E_{\Gamma}$, the set of edges of $\Gamma$.  If $\Gamma$ is a connected graph endowed with an ordering on the set $E_{\Gamma}$, i.e., a fixed bijection $E_{\Gamma} \overset{\sim}{\rightarrow} \{1,\ldots, |E_{\Gamma}|\}$, then we have a canonical choice of an element of orientation torsor associated with the ordering of $E_{\Gamma}$, i.e., $E_1 \wedge \cdots \wedge E_{|E_{\Gamma}|}$. Suppose that for a connected graph $\Gamma$, we give two different orderings on $E_{\Gamma}$. Two such orderings are said to be equivalent if their associated canonical elements of orientation torsors coincide. Then, such an equivalence class can be identified with a pair $(\Gamma; \mathrm{Or}_{\Gamma})$ of a connected (topological) graph $\Gamma$ and an element of the orientation torsor $\mathrm{Or}_{\Gamma} \in \mathrm{or}_{\Gamma}$ of $\Gamma$. If we fix one element $\mathrm{Or}_{\Gamma}$ of orientation torsor of $\Gamma$, there are exactly two equivalence classes  $(\Gamma; \mathrm{Or}_{\Gamma})$ and $(\Gamma; -\mathrm{Or}_{\Gamma})$.

Now we give a detail on the canonical orientation for planar trivalent trees. Let $n$ be an integer with $n \geq 3$. For a uni-trivalent planar tree graph $T$ with $n$-external vertices, we can define a canonical orientation torsor as used in Section \ref{subsection:2.3}. In \cite{GoncharovHodge1, malkin2020shuffle}, they used the canonical orientation without details, so here we give a detailed explanation for completeness. First, choose one of the external edges of $T$ as the root edge. We number this edge as $E_1$. Then, take a small neighbourhood of $T$ whose boundary has the induced orientation from that of the ambient plane. Starting from $E_1$, number the other edges going along the oriented boundary as $E_2,\ldots, E_{2n-3}$. Then, by this ordering on $E_T$, one obtains the corresponding orientation torsor $\mathrm{Or}_T$.

\begin{lemma}\label{lemm:or_T}
    Keeping the notation as above, the element $\mathrm{Or}_T$ of orientation torsor does not depend on the choice of root edge of $T$.
\end{lemma}
\begin{proof}
Note that $\mathrm{Or}_T$ can be written as 
\begin{equation}
    \mathrm{Or}_T = E_1 \wedge E_2 \wedge E_3 \wedge \mathrm{Or}_{T_1} \wedge \mathrm{Or}_{T_2}
\end{equation}
where $T_1, T_2$ denote uni-trivalent trees obtained from $T$ be removing $E_1, E_2, E_3$. Here, $T_1$ is the resulting tree adjacent to $E_3$ and $T_2$ to $E_1$. Let us choose as  root edge the external edge whose external vertex lies right after the $E_1$ along the oriented boundary circle of the ambient disk. Since $T$ is a uni-trivalent planer tree, this edge is numbered as $E_3$ in the previous ordering on $E_T$. In the new ordering by choosing $E_3$ as root, the corresponding orientation torsor is given by the form
    \begin{equation}
        E_3 \wedge \mathrm{Or}_{T_1} \wedge E_2 \wedge \mathrm{Or}_{T_2} \wedge E_1.
    \end{equation}
    Since the number of edges of the uni-trivalent tree is odd, 
    \begin{equation}
        \begin{split}
            &E_3 \wedge \mathrm{Or}_{T_1} \wedge E_2 \wedge \mathrm{Or}_{T_2} \wedge E_1\\
            =&E_1 \wedge E_3 \wedge \mathrm{Or}_{T_1} \wedge E_2 \wedge \mathrm{Or}_{T_2}\\
            =& - E_1 \wedge E_3 \wedge E_2 \wedge \mathrm{Or}_{T_1} \wedge \mathrm{Or}_{T_2}\\
            =& E_1 \wedge E_2 \wedge E_3 \wedge \mathrm{Or}_{T_1} \wedge \mathrm{Or}_{T_2}\\
            =& \mathrm{Or}_T.
        \end{split}
    \end{equation}
    Other choices of root edge can be reduced to an iteration of the above type of replacement. Therefore, $\mathrm{Or}_T$ is independent of the choice of root edge of $T$.
\end{proof}
Lemma \ref{lemm:or_T} state that the orientation torsor $\mathrm{Or}_T$ depends only on the orientation of the plane, so this orientation torsor $\mathrm{Or}_T$ is called \textit{canonical} orientation torsor of $T$.

\subsection{Graph complex $\mathcal{D}^{\vee, \bullet}_{H,S^{\ast}}$} \label{section:4.4}
Here, we define an appropriate graph complex $\mathcal{D}^{\vee, \bullet}_{H,S^{\ast}}$ for our purpose with the following steps similar to \cite[\S 6]{GoncharovHodge1}. First, we introduce graded vector space $\mathcal{CD}^{\vee, \bullet}_{H,S^{\ast}}$ of decorated connected graphs. Then, graded commutative algebra $\mathcal{D}^{\vee, \bullet}_{H,S^{\ast}}$ generated by $\mathcal{CD}^{\vee, \bullet}_{H,S^{\ast}}$ is defined. Finally, differential operator $\partial$ is introduced on $\mathcal{D}^{\vee, \bullet}_{H,S^{\ast}}$. 

Until the end of Section \ref{section:4.5}, we keep the following notation: Let $H^{\vee}$ be a finite-dimensional symplectic vector space over $\bC$ and $S^{\ast}:= \{s_1, \ldots, s_k\}$ be a finite set. We set $V_{H,S}^{\vee} := H^{\vee} \oplus \mathbb{C}[S^{\ast}]$ where an inner product $(-, -)$ on $\bC[S]$ is defined by setting $(\{s_i\}, \{s_j\}) = \delta_{ij}$ where $\delta_{ij}$ denotes the Kronecker delta. 
\subsubsection{Graded spaces $\mathcal{CD}^{\vee, \bullet}_{H,S^{\ast}}$ and $\mathcal{D}^{\vee, \bullet}_{H,S^{\ast}}$ of decorated graphs}
To begin with, following \cite[\S 3]{KMV13}, we recall the notion of the defect and (loop) order of a graph. For a connected graph $\Gamma$, its \textit{defect} is defined as
\begin{equation}
	\mathrm{def}(\Gamma) = 2 |E_{\Gamma}| - 3 | V^{\mathrm{int}}_{\Gamma}| - | V^{\mathrm{ext}}_{\Gamma}|.
\end{equation}

 The defect may be viewed as a measure of the failure of $\Gamma$ to be uni-trivalent since $\mathrm{def}(\Gamma) = 0$ for a uni-trivalent graph $\Gamma$.  
 For a connected graph $\Gamma$, the \textit{order} is defined as
\begin{equation}
    \mathrm{ord}(\Gamma) = |E_{\Gamma}| - |V_{\Gamma}| = l(\Gamma) - 1 = -\chi(\Gamma)
\end{equation}
where $\chi(\Gamma)$ is the Euler characteristic number of $\Gamma$.
\begin{definition}[Graded space of connected decorated graphs]
	Define $\mathcal{CD}^{\vee, \bullet}_{H,S^{\vee}}$ as the graded $\mathbb{C}$-vector space generated by triples $(\Gamma, W; \mathrm{Or}_{\Gamma})$, where $\Gamma$ is a connected graph, $W$ is a $V^{\vee}_{H,S}$-decoration of $\Gamma$, and $\mathrm{Or}_{\Gamma}$ is an element of the orientation torsor of $\Gamma$ induced by a choice of ordering of $E_{\Gamma}$, subject to the relations:
	\begin{itemize}
	\item $(\Gamma, W; \mathrm{Or}_{\Gamma}) = - (\Gamma, W; -\mathrm{Or}_{\Gamma})$,
	\item $(\Gamma, \lambda W_1 + \lambda_2 W_2; \mathrm{Or}_{\Gamma}) = \lambda_1 (\Gamma, W_1; \mathrm{Or}_{\Gamma}) + \lambda_2(\Gamma, W_2; \mathrm{Or}_{\Gamma})$ for $\lambda_1, \lambda_2 \in \bC$,
 \item  $(\Gamma, W; \mathrm{Or}_{\Gamma}) =0$ if 
 \begin{enumerate}[(A)]
     \item $\Gamma$ has at least one self-loop, as Figure \ref{fig:4.3.1} (A),
     \item $\Gamma$ is the unique connected uni-trivalent graph without internal vertex and its two external vertices are labeled by the same  $\gamma = s \in S^{\ast}$ or $\gamma = \alpha \in H^{\vee}$, as Figure \ref{fig:4.3.1} (B),
     \item $\Gamma$ has an internal vertex $v$ which is connected to two external vertices labeled by the same $\gamma = s \in S^{\ast}$ or $\gamma = \alpha \in H^{\vee}$, as Figure \ref{fig:4.3.1} (C),
     \item $\Gamma$ has a non-regular edge, as Figure \ref{fig:4.3.1} (D).
 \end{enumerate}
	\end{itemize}
	The degree $\deg(\Gamma, W; \mathrm{Or}_{\Gamma}) $ of $(\Gamma, W; \mathrm{Or}_{\Gamma})$ is defined as
	\begin{equation}
 \begin{split}
		\deg(\Gamma, W; \mathrm{Or}_{\Gamma}) := \deg(\Gamma):=&\mathrm{def}(\Gamma) - l(\Gamma) + 1=|E_{\Gamma}| - 2 |V_{\Gamma}^{\mathrm{int}}|.
\end{split}
	\end{equation}
\end{definition}

\begin{figure}[h]
\captionsetup{margin=2cm}
 \centering
 \begin{tikzpicture}
 \node at (-0.5, -1.5) {(A)};
  \node at (3, -1.5) {(B)};
  \node at (7, -1.5) {(C)};
  \node at (10, -1.5) {(D)};
 \begin{scope}
 \draw[densely dotted] (0,0) circle (0.5cm);
    \draw[densely dotted] (-0.5,0.0)--(-1.5,0.5);
     \draw[densely dotted] (-0.5,0.0)--(-1,1);
     \draw[densely dotted] (-0.5,0.0)--(-1.5,-0.5);
     \draw[densely dotted] (-0.5,0.0)--(-1,-1);
     \node at (-1.25, 0.1) {$\vdots$};
     \node at (-0.5, 0) {$\bullet$};
 \end{scope}
 \begin{scope}[xshift=3cm]
 \draw[densely dotted] (-0.5,0) -- (0.5,0);
 \node at (0.5, 0) {$\bullet$};
 \node at (0.9, 0) {$\gamma$};
 \node at (-0.5,0) {$\bullet$};
 \node at (-0.9, 0) {$\gamma$};
 \end{scope}
 \begin{scope}[xshift=6cm]
 \draw[densely dotted] (1-1/1.414,0) -- (1,0);
 	\draw[densely dotted] (1,0) -- (1.5, 0.5);
 	\draw[densely dotted] (1,0) -- (1.5, -0.5);
  \node at (1.5, 0.5) {$\bullet$};
   \node at (1.5, -0.5) {$\bullet$};
   \node at (1,0) {$\bullet$};
   \node at (1,-0.3) {$v$};
  \node at (1.5+0.3, 0.5+0.3) {$\gamma$};
  \node at (1.5+0.3, -0.5-0.3) {$\gamma$};
 \end{scope}
 \begin{scope}[xshift=10cm]
  \draw[densely dotted] (0,0) circle (0.5cm);
   \draw[densely dotted] (0.5,0.0)--(1.0,0.0);
    \draw[densely dotted] (-0.5,0.0)--(-1.0,0.0);
    \node at (0.5, 0) {$\bullet$};
     \node at (-0.5, 0) {$\bullet$};
 \end{scope}
 \end{tikzpicture}
 \caption{Several typical (sub)graphs which are defined to be zero. }
 \label{fig:4.3.1}
\end{figure}
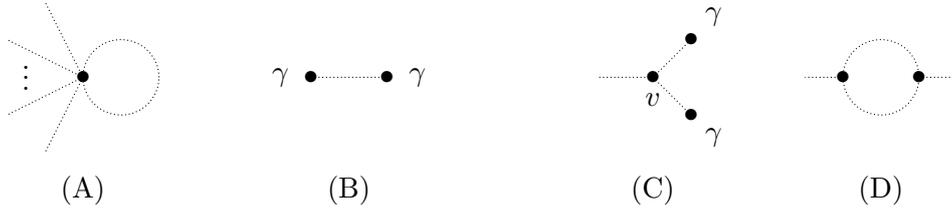

\begin{remark}
The degree $\deg$ of a triple $(T, W; \mathrm{Or}_T)$ for a planar tree $T$ is defined as
\begin{equation}
    \deg((T, W; \mathrm{Or}_T)) := 1 + \sum_{v ;  v \in V_T^{\mathrm{int}}}(\mathrm{val}(v) - 3)
\end{equation}
in \cite[(6.2)]{GoncharovHodge1}. Here, $\mathrm{val}(v)$ denotes the valency of $v$. This definition coincides with our definition using the notion of defect $\deg$. Indeed, we have
\begin{equation}
	\mathrm{def}(T, W; \mathrm{Or}_T) = \deg(T, W; \mathrm{Or}_T) -1 = \sum_{v \in V_T^{\mathrm{int}}}(\mathrm{val}(v) - 3).
\end{equation}
This can be checked as follows: Since $T$ is a planar tree, one has
	\begin{equation}
		|V^{\mathrm{int}}_T| = |V^{\mathrm{ext}}_T| - 2 -\sum_{v \in V_T^{\mathrm{int}}}(\mathrm{val}(v) - 3), \quad |E_T| = 2 |V^{\mathrm{ext}}_T| - 3 -\sum_{v \in V_T^{\mathrm{int}}}(\mathrm{val}(v) - 3).
	\end{equation}
	By substituting them for $\mathrm{def}(T) = 2 |E_T| - 3 |V_T^{\mathrm{int}}| -  |V^{\mathrm{ext}}_T|$, we obtain the desired formula.
\end{remark}

Note that the graded vector space $\mathcal{CD}^{\vee, \bullet}_{H, S^{\ast}}$ can be regarded as that with grading by the total degree of bigraded vector space graded by $\mathrm{def}$ and $-\ord$, that is,
\begin{equation}\label{eq:dec_cd}
  \mathcal{CD}^{\vee, \bullet}_{H,S^{\ast}} = \bigoplus_{k}\mathcal{CD}^{\vee, k}_{H,S^{\ast}}= \bigoplus_{k} \bigoplus_{d + l = k}  \mathcal{CD}^{\vee,(d, l)}_{H,S^{\ast}}
\end{equation}
where $\mathcal{CD}^{\vee,(d, l)}_{H,S^{\ast}}$ denotes the subspace of $\mathrm{def}=d$ and $\ord = -l$ connected graphs.

\begin{definition}[Graded commutative algebra of decorated  graphs]
	Let us define the graded commutative algebra $\mathcal{D}^{\vee, \bullet}_{H,S^{\ast}}$ as the (graded) symmetric algebra 
	\begin{equation}
		\mathcal{D}^{\vee, \bullet}_{H,S^{\ast}}:= S^{\bullet}(\mathcal{CD}^{\vee, \bullet}_{H,S^{\ast}})
	\end{equation} such that
	\begin{itemize}
	\item the underlying vector space is generated by
	\begin{equation}\label{eq:gen_vec_cd}
		(\Gamma_1 \sqcup \cdots \sqcup \Gamma_m, W_1 \sqcup \cdots \sqcup W_m; \mathrm{Or}_{\Gamma_1} \wedge \cdots  \wedge\mathrm{Or}_{\Gamma_m}) 
	\end{equation}
	where $\Gamma_1 \sqcup \cdots \sqcup \Gamma_m$ is the disjoint unions of connected graphs $\Gamma_i$, $W_1 \sqcup \cdots \sqcup W_m$ is the symmetric product of $V^{\vee}_{H,S^{\ast}}$-decorations $W_i: V_{\Gamma_i}^{\mathrm{ext}} \rightarrow  V^{\vee}_{H,S^{\ast}}$ (cf. Convention \ref{conv:sym_dec}), and $\mathrm{Or}_{\Gamma_1} \wedge \cdots  \wedge\mathrm{Or}_{\Gamma_m}$ is the wedge product of $\mathrm{Or}_{\Gamma_i} \in \mathrm{or}_{\Gamma_i}$ $(1\leq i \leq m)$. The relations between the generators are induced by those on $\mathcal{CD}^{\vee, \bullet}_{H,S^{\ast}}$,

	\item the grading is defined as
	\begin{equation} \label{eq:degree}
 \begin{split}
		\deg(\Gamma, W;\mathrm{Or}_{\Gamma}) :=& \deg(\Gamma)=\sum_{i=1}^m \deg(\Gamma_i)\\
  =& \sum_{i=1}^m \left(\mathrm{def}(\Gamma_i) - l(\Gamma_i) +1\right) = \sum_{i=1}^m (|E_{\Gamma_i}| - 2|V_{\Gamma_i}^{\mathrm{int}}|)
  \end{split}
	\end{equation}
	when $\Gamma = \Gamma_1 \sqcup \cdots \sqcup \Gamma_m$ for connected graphs $\Gamma_1,\ldots, \Gamma_m$,
	\item the product in $\mathcal{D}^{\vee, \bullet}_{H,S}$ is given by disjoint union respecting their decorations as
	\begin{equation}
		(\Gamma_1, W_1; \mathrm{Or}_{\Gamma_1}) \wedge (\Gamma_2, W_2; \mathrm{Or}_{\Gamma_2}) := (\Gamma_1 \sqcup \Gamma_2 , W_1 \sqcup W_2; \mathrm{Or}_{\Gamma_1}\wedge \mathrm{Or}_{\Gamma_2}).
	\end{equation}
	\end{itemize} 
\end{definition}
The above product $\wedge$ gives graded commutative algebra structure on $\mathcal{D}^{\vee, \bullet}_{H,S^{\ast}}$, that is,  for $(\Gamma_1, W_1; \mathrm{Or}_{\Gamma_1}),  (\Gamma_2, W_2; \mathrm{Or}_{\Gamma_2}) \in \mathcal{D}^{\vee, \bullet}_{H,S^{\ast}}$, the following equation holds:
\begin{equation}\label{eq:product_sign}
    \begin{split}
        & (\Gamma_1, W_1; \mathrm{Or}_{\Gamma_1}) \wedge (\Gamma_2, W_2; \mathrm{Or}_{\Gamma_2}) \\
        =& (-1)^{\deg(\Gamma_1)\deg(\Gamma_2)} (\Gamma_2, W_2; \mathrm{Or}_{\Gamma_2}) \wedge (\Gamma_1, W_1; \mathrm{Or}_{\Gamma_1}).
    \end{split}
\end{equation}
Indeed, noting that $\deg(\Gamma_i)=|E_{\Gamma_i}| -2 |V_{\Gamma_i}^{\mathrm{int}}|$ as defined in \eqref{eq:degree},
\begin{equation}
    \begin{split}
        & (\Gamma_1, W_1; \mathrm{Or}_{\Gamma_1}) \wedge (\Gamma_2, W_2; \mathrm{Or}_{\Gamma_2}) \\
        =& (\Gamma_1 \sqcup \Gamma_2, W_1 \sqcup W_2; \mathrm{Or}_{\Gamma_1} \wedge \mathrm{Or}_{\Gamma_2})\\
        =& (\Gamma_2 \sqcup \Gamma_1, W_2 \sqcup W_1; (-1)^{|E_{\Gamma_1}||E_{\Gamma_2}|}\mathrm{Or}_{\Gamma_2} \wedge \mathrm{Or}_{\Gamma_1})\\
        =& (-1)^{|E_{\Gamma_1}||E_{\Gamma_2}|}  (\Gamma_2, W_2; \mathrm{Or}_{\Gamma_2}) \wedge (\Gamma_1, W_1; \mathrm{Or}_{\Gamma_1})\\
        =& (-1)^{\deg(\Gamma_1)\deg(\Gamma_2)} (\Gamma_2, W_2; \mathrm{Or}_{\Gamma_2}) \wedge (\Gamma_1, W_1; \mathrm{Or}_{\Gamma_1}).
    \end{split}
\end{equation}


As in the case of $\mathcal{CD}^{\vee, \bullet}_{H, S^{\ast}}$, the underlying graded vector space of $\mathcal{D}^{\vee, \bullet}_{H, S^{\ast}}$ can be viewed as the graded vector space graded by the total degree of bigraded vector space graded by $\mathrm{def}$ and $\ord$:
\begin{equation}
    \mathcal{D}^{\vee, \bullet}_{H, S^{\ast}} = \bigoplus_{k} \mathcal{D}^{\vee, k}_{H, S^{\ast}}= \bigoplus_{k} \bigoplus_{d+l=k} \mathcal{D}^{\vee, (d,l)}_{H, S^{\ast}}
\end{equation}
where $\mathcal{D}^{\vee, (d,l)}_{H, S^{\ast}}$ denotes the subspace consisting of products of graphs with $\mathrm{def}=d$ and $\ord = -l$.
\subsubsection{CDGA $\mathcal{D}^{\vee, \bullet}_{H,S^{\ast}}$ of decorated graphs}
By formally extending the differential on connected decorated graphs given in \cite[\S 6]{GoncharovHodge1}, we shall define a differential operator
\begin{equation}
	\partial : \mathcal{D}^{\vee, \bullet}_{H,S^{\ast}} \rightarrow \mathcal{D}^{\vee, \bullet+1}_{H,S^{\ast}}
\end{equation}
which gives a commutative differential graded algebra (CDGA) structure on $\mathcal{D}^{\vee, \bullet}_{H,S}$.  The differential $\partial$ is the extension by Leibniz rule  of  $\partial$ acting on connected  decorated  graphs described as follows: The differential $\partial$ splits into three pieces
\begin{equation}
	\partial = \partial_{\Delta} + \partial_{\mathrm{Cas}} + \partial_{S^{\ast}}.
\end{equation}

\noindent
(i) The map $\partial_{\Delta}$. Let  $(\Gamma, W_{\Gamma}; \mathrm{Or}_{\Gamma})$ be a triple of a connected  graph $\Gamma$, its decoration $W_{\Gamma}$ and an element of its orientation torsor, and $E$ be an internal edge of $\Gamma$. Then, as Figure \ref{fig:part_delta}, we set $\Gamma/E$ as the connected graph obtained from $\Gamma$ by contracting the internal edge $E$, and $\mathrm{Or}_{\Gamma/E}$ so that $\mathrm{Or}_{\Gamma}=E \wedge\mathrm{Or}_{\Gamma/E}$. The map $\partial_{\Delta}$ is defined by
\begin{equation}
\partial_{\Delta}(\Gamma,W;\mathrm{Or}_{\Gamma}) := 	\sum_{\substack{E \in E_{\Gamma}^{\mathrm{int}}}}(\Gamma/E, W_{\Gamma}; \mathrm{Or}_{\Gamma/E}).
\end{equation}

\noindent
(ii) The map $\partial_{\mathrm{Cas}}$. Let $E$ be an edge of $\Gamma$. We divide into two cases whether $E$ is a cycle edge or not. First, let us assume that $E$ is not a cycle edge so that if we cut the graph $\Gamma$ along $E$ we get two connected graphs $\Gamma_1$ and $\Gamma_2$ with inherited partial decorations $W_1'$ and $W_2'$ so that $W = W_1' \sqcup W_2'$. Let $E_1$ and $E_2$ be the new external edges of the graphs $\Gamma_1$ and $\Gamma_2$ obtained by cutting the edge $E$. 

We decorate the resulting new external edges by the Casimir element $\mathrm{id}: H^{\vee} \rightarrow H^{\vee}$ as follows. We choose a basis $\{\alpha_k\}$ of $H^{\vee}$ and $\{\alpha_k^{\vee}\}$ be its dual basis with respect to the symplectic form $(-, -)$ on $H^{\vee}$, that is, $(\alpha_k, \alpha_l^{\vee}) = \delta_{kl}$. Then, $\mathrm{id}$ can be written as
\begin{equation}
    \mathrm{id} = \sum_{k} \alpha_k^{\vee} \otimes \alpha_k.
\end{equation}
Decorating $E_1$ by $\alpha_k$ and $E_2$ by $\alpha_k^{\vee}$, we obtain a decoration  on $\Gamma_1$ and $\Gamma_2$ denoted by $W_1' \sqcup \alpha_k$ and $\alpha_k^{\vee} \sqcup W_1'$. We choose their orientations so that $\mathrm{Or}_{\Gamma} = E \wedge \mathrm{Or}_{\Gamma_1\setminus E_1}\wedge \mathrm{Or}_{\Gamma_2 \setminus E_2}$ and $\mathrm{Or}_{\Gamma_1}\wedge \mathrm{Or}_{\Gamma_2} = E_1 \wedge E_2 \wedge \mathrm{Or}_{\Gamma_1\setminus E_1}\wedge \mathrm{Or}_{\Gamma_2 \setminus E_2}$. As in Figure \ref{fig:part_cas}, we set
\begin{equation}\label{eq:partial_cas1}
\partial_{\mathrm{Cas}}(\Gamma, W_{\Gamma}; \mathrm{Or}_{\Gamma}) := \sum_{\substack{E \in E_{\Gamma}\\ \text{$E$ is not cycle edge}}} \sum_{k=1}^g (\Gamma_1,W_1' \sqcup \alpha_k; \mathrm{Or}_{\Gamma_1})\wedge (\Gamma_2, \alpha^{\vee}_k \sqcup W_2'; \mathrm{Or}_{\Gamma_2}).
\end{equation}
Note that the above definition does not depend on the choice of numbering of the resulting graphs $\Gamma_1$ and $\Gamma_2$. This can be checked by using the relation
\begin{equation}
    \mathrm{id} = \sum_{k} \alpha_k^{\vee} \otimes \alpha_k=  \sum_{k} \alpha_k \otimes (-\alpha_k^{\vee})
\end{equation}
and the definition of induced orientations of $\Gamma_1$ and $\Gamma_2$.

Next, we consider the case that $E$ is a cycle edge so that if we cut the graph $\Gamma$ along $E$ we get a connected graph $\Gamma'$ with partial decoration $W'$. Let $E_1$ and $E_2$ be the new external edges of the graph $\Gamma'$ obtained by cutting the edge $E$. We decorate $E_1$ and $E_2$ by $\alpha_k$ and  $\alpha_k^{\vee}$ respectively and denote the resulting decoration by $W' \sqcup \alpha_k \sqcup \alpha_k^{\vee}$ We choose the orientation of $\Gamma'$ so that $\mathrm{Or}_{\Gamma} = E \wedge \mathrm{Or}_{\Gamma'\setminus \{E_1, E_2\}}$ and $\mathrm{Or}_{\Gamma'} = E_1 \wedge E_2 \wedge \mathrm{Or}_{\Gamma'\setminus\{E_1, E_2\}}$.
\begin{equation}\label{eq:partial_cas2}
\partial_{\mathrm{Cas}}(\Gamma, W_{\Gamma}; \mathrm{Or}_{\Gamma}) := \sum_{\substack{E \in E_{\Gamma}\\ \text{$E$ is cycle edge}}} \sum_{k=1}^g (\Gamma',W' \sqcup \alpha_k \sqcup \alpha_k^{\vee}; \mathrm{Or}_{\Gamma'}).
\end{equation}
Similar to the former case, this definition is independent of the choice of the numbering of the resulting edges $E_1$ and $E_2$.

\noindent
(iii) The map $\partial_{S^{\ast}}$. Assume that $\Gamma$ is not the tree graph with exactly two external vertices. Let $E$ be an external $S^{\ast}$-decorated edge of $\Gamma$. As in Figure \ref{fig:S-deco}, let us remove $E$ together with a little neighborhood of its vertices. Then, one of the vertices connected by $E$ is of valency $v \geq 3$. Depending on its global structure, $\Gamma$ is replaced by $\Gamma_1, \ldots, \Gamma_k$ for some $k \leq v-1$. We choose their orientations $\mathrm{Or}_{\Gamma_i}$ so that $\mathrm{Or}_{\Gamma} = E \wedge \mathrm{Or}_{\Gamma_1} \wedge \cdots \wedge \mathrm{Or}_{\Gamma_k}$.

\begin{equation}\label{eq:paartial_S1}
\frac{\partial}{\partial E} (\Gamma, W_{\Gamma}; \mathrm{Or}_{\Gamma}) := (\Gamma_1, W_{\Gamma_1}; \mathrm{Or}_{\Gamma_1}) \wedge \cdots \wedge (\Gamma_{k}, W_{\Gamma_{k}}; \mathrm{Or}_{\Gamma_{k}}) 
\end{equation}

\begin{equation}
	\partial_{S^{\ast}}:= \sum_{E \in E_{}^{\mathrm{ext}}: \text{$S$-decorated}}\frac{\partial}{\partial E}
\end{equation}

Then, we have the following proposition which is a generalization of  \cite[Theorem 6.4]{GoncharovHodge1}.

\begin{theorem}
    The  map $\partial = \partial_{\Delta} + \partial_{\mathrm{Cas}} + \partial_{S^{\ast}}$ is a differential on $\mathcal{D}^{\vee, \bullet}_{H,S}$, i.e., $\partial^2 = 0$. In particular, 
\begin{equation}
	\partial_{\mathrm{Cas}}(\partial_{\Delta} + \partial_{S^{\ast}}) = - (\partial_{\Delta} + \partial_{S^{\ast}})\partial_{\mathrm{Cas}},\quad  \partial_{\Delta}^2 = 0, \quad \partial_{\mathrm{Cas}}^2 = 0,\quad  (\partial_{\Delta} + \partial_{S^{\ast}})^2 = 0.
\end{equation} 
\end{theorem}

\begin{proof}
The proof of the assertion goes along the same lines as \cite[Theorem 6.4]{GoncharovHodge1}, where the proof is given for the planar tree case. The only different point is that the presence of the (loop) order of $\Gamma$ in the definition of the degree \eqref{eq:degree} ensures that the map $\partial$ is of degree $+1$. Except for this part, we can apply similar arguments in [loc.cit.] since the existence of loops does not affect the main points of the proof.
\end{proof}

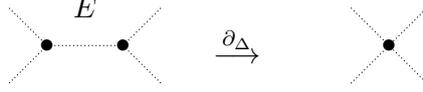
\begin{figure}[h]
 \centering
 \begin{tikzpicture}
 	\draw[densely dotted] (0,0) -- (1,0);
 	\draw[densely dotted] (0,0) -- (-0.5, 0.5);
 	\draw[densely dotted] (0,0) -- (-0.5, -0.5);
 	\draw[densely dotted] (1,0) -- (1.5, 0.5);
 	\draw[densely dotted] (1,0) -- (1.5, -0.5);
  \node at (0,0) {$\bullet$};
  \node at (1,0) {$\bullet$};
 	\node at (0.5, 0.5) {$E$};
 	\node at (2.5,0) {$\overset{\partial_{\Delta}}{\longrightarrow}$};
 	\draw[densely dotted] (4, -0.5) -- (5, 0.5);
 	\draw[densely dotted] (4, 0.5) -- (5, -0.5);
 	\node at (4.5,0) {$\bullet$};
 \end{tikzpicture}
 \caption{The map $\partial_{\Delta}$ for an internal edge $E$}
 \label{fig:part_delta}
\end{figure}

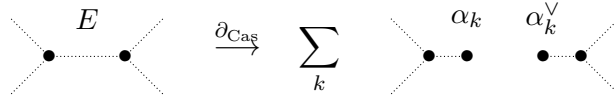
\begin{figure}[h]
 \centering
 \begin{tikzpicture}
 	\draw[densely dotted] (0,0) -- (1,0);
 	\draw[densely dotted] (0,0) -- (-0.5, 0.5);
 	\draw[densely dotted] (0,0) -- (-0.5, -0.5);
 	\draw[densely dotted] (1,0) -- (1.5, 0.5);
 	\draw[densely dotted] (1,0) -- (1.5, -0.5);
   \node at (0,0) {$\bullet$};
  \node at (1,0) {$\bullet$};
 	\node at (0.5, 0.5) {$E$};
 	\node at (3,0) {$\overset{\partial_{\mathrm{Cas}}}{\longrightarrow} \quad \displaystyle\sum_k$};
	\draw[densely dotted] (4.5,0.5) -- (5, 0);
	\draw[densely dotted] (5,0) -- (4.5, -0.5);
	\draw[densely dotted] (5,0) -- (5.5,0);
  \node at (7,0) {$\bullet$};
  \node at (5,0) {$\bullet$};
	\node at (5.5,0) {$\bullet$};
	\node at (6.5, 0) {$\bullet$};
	\draw[densely dotted] (6.5,0) -- (7,0);
	\draw[densely dotted] (7, 0) -- (7.5, 0.5);
	\draw[densely dotted] (7, 0) -- (7.5, -0.5);
	\node at (5.5, 0.5) {$\alpha_k$};
	\node at (6.5, 0.5) {$\alpha_k^{\vee}$};
 	 \end{tikzpicture}
 \caption{The map $\partial_{\mathrm{Cas}}$ for an internal edge $E$}
 \label{fig:part_cas}
\end{figure}

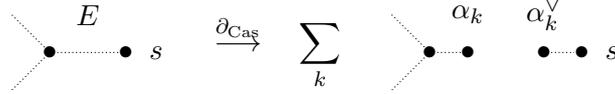
\begin{figure}[h]
 \centering
 \begin{tikzpicture}
 	\draw[densely dotted] (0,0) -- (0.9,0);
 	\draw[densely dotted] (0,0) -- (-0.5, 0.5);
 	\draw[densely dotted] (0,0) -- (-0.5, -0.5);
 	\node at (0.5, 0.5) {$E$};
  \node at (0,0) {$\bullet$};
 	\node at (1,0) {$\bullet$};
 	\node at (1.4, 0) {$s$};
 	\node at (3,0) {$\overset{\partial_{\mathrm{Cas}}}{\longrightarrow} \quad \displaystyle\sum_k$};
	\draw[densely dotted] (4.5,0.5) -- (5, 0);
	\draw[densely dotted] (5,0) -- (4.5, -0.5);
	\draw[densely dotted] (5,0) -- (5.5,0);
	\node at (5.5,0) {$\bullet$};
	\node at (6.5, 0) {$\bullet$};
	\draw[densely dotted] (6.5,0) -- (6.9,0);
	\node at (7,0) {$\bullet$};
 \node at (5,0) {$\bullet$};
	\node at (5.5, 0.5) {$\alpha_k$};
	\node at (6.5, 0.5) {$\alpha_k^{\vee}$};
	\node at (7.4, 0) {$s$};
 	 \end{tikzpicture}
 \caption{The map $\partial_{\mathrm{Cas}}$ for an $S^{\ast}$-decorated external edge $E$}
 \label{fig:Ca_S}
\end{figure}

\begin{figure}[h]
 \centering
 \begin{tikzpicture}
 	\draw[densely dotted] (0,0) -- (0, -0.9);
 	\draw[densely dotted] (0,0) -- (-0.7071, 0.7071);
 	\draw[densely dotted] (0,0) -- (0.7071, 0.7071);
 	\draw[densely dotted] (0.7071, 0.7071) -- (0.7071+0.8, 0.7071);
 	\draw[densely dotted] (0.7071, 0.7071) -- (0.7071, 0.7071+0.8);
 	\draw[densely dotted] (0.7071+0.8, 0.7071) -- (0.7071+0.8+0.424, 0.7071-0.424);
 	\draw[densely dotted] (0.7071+0.8, 0.7071) -- (0.7071+0.8+0.424, 0.7071+0.424);
 	\draw[densely dotted] (0.7071, 0.7071+0.8) -- (0.7071+0.424, 0.7071+0.424+0.8);
 	\draw[densely dotted] (0.7071, 0.7071+0.8) -- (0.7071-0.424, 0.7071+0.424+0.8);
 	\draw[densely dotted] (-0.7071, 0.7071+0.8) -- (-0.7071-0.424, 0.7071+0.424+0.8);
 	\draw[densely dotted] (-0.7071, 0.7071+0.8) -- (-0.7071+0.424, 0.7071+0.424+0.8);
 	\draw[densely dotted] (-0.7071-0.8, 0.7071) -- (-0.7071-0.8-0.424, 0.7071-0.424);
 	\draw[densely dotted] (-0.7071-0.8, 0.7071) -- (-0.7071-0.8-0.424, 0.7071+0.424);
 	\draw[densely dotted] (-0.7071, 0.7071) -- (-0.7071-0.8, 0.7071);
 	\draw[densely dotted] (-0.7071, 0.7071) -- (-0.7071, 0.7071+0.8);
 	\node at (0,-1) {$\bullet$};
  \foreach \x in {(0,0),(0.7071, 0.7071),(0.7071+0.8, 0.7071), (0.7071, 0.7071+0.8), (-0.7071, 0.7071),(-0.7071-0.8, 0.7071), (-0.7071, 0.7071+0.8)}{
   \node at \x {$\bullet$};
  }
 	\node at (0,-1.5) {$s$};
 	\node at (-0.5,-0.5) {$E$};
 	\node at (3, 0) {$\overset{\partial_{S^{\ast}}}{\longrightarrow}$};
 \begin{scope}[xshift=6cm]
 	\draw[densely dotted] (0,0) -- (-0.7071, 0.7071);
 	\draw[densely dotted] (-0.7071, 0.7071+0.8) -- (-0.7071-0.424, 0.7071+0.424+0.8);
 	\draw[densely dotted] (-0.7071, 0.7071+0.8) -- (-0.7071+0.424, 0.7071+0.424+0.8);
 	\draw[densely dotted] (-0.7071-0.8, 0.7071) -- (-0.7071-0.8-0.424, 0.7071-0.424);
 	\draw[densely dotted] (-0.7071-0.8, 0.7071) -- (-0.7071-0.8-0.424, 0.7071+0.424);
 	\draw[densely dotted] (-0.7071, 0.7071) -- (-0.7071-0.8, 0.7071);
 	\draw[densely dotted] (-0.7071, 0.7071) -- (-0.7071, 0.7071+0.8);
 	\foreach \x in {(0,0), (-0.7071, 0.7071),(-0.7071-0.8, 0.7071), (-0.7071, 0.7071+0.8)}{
   \node at \x {$\bullet$};
  }
 	\node at (0,-0.5) {$s$};
 \end{scope}
\begin{scope}[xshift=7cm]
 	\draw[densely dotted] (0,0) -- (0.7071, 0.7071);
 	\draw[densely dotted] (0.7071, 0.7071) -- (0.7071+0.8, 0.7071);
 	\draw[densely dotted] (0.7071, 0.7071) -- (0.7071, 0.7071+0.8);
 	\draw[densely dotted] (0.7071+0.8, 0.7071) -- (0.7071+0.8+0.424, 0.7071-0.424);
 	\draw[densely dotted] (0.7071+0.8, 0.7071) -- (0.7071+0.8+0.424, 0.7071+0.424);
 	\draw[densely dotted] (0.7071, 0.7071+0.8) -- (0.7071+0.424, 0.7071+0.424+0.8);
 	\draw[densely dotted] (0.7071, 0.7071+0.8) -- (0.7071-0.424, 0.7071+0.424+0.8);
 	\foreach \x in {(0,0),(0.7071, 0.7071),(0.7071+0.8, 0.7071), (0.7071, 0.7071+0.8)}{
   \node at \x {$\bullet$};
  }
 	\node at (0,-0.5) {$s$};
 \end{scope}
 \end{tikzpicture}
 \caption{The map $\partial_{S^{\ast}}$}
 \label{fig:S-deco}
\end{figure}
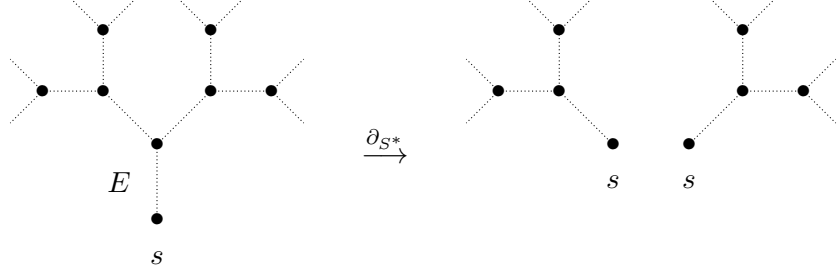

\subsection{DG Lie coalgebra  $H_{\partial_{\Delta}}^0(\mathcal{CD}^{\vee, \bullet}_{H,S^{\ast}})$ and DG Lie algebra  $H^{\partial_{\Delta}}_0(\mathcal{CD}^{\vee, \bullet}_{H,S^{\ast}})$} \label{section:4.5}
At the end of this section, we describe algebraic structures of the 0-th cohomology groups $H_{\partial_{\Delta}}^0(\mathcal{CD}^{\vee, \bullet}_{H,S^{\ast}})$ which we use in the subsequent sections.

\subsubsection{Inner product on the space of decorated graphs}
First of all, we shall introduce an inner product on the space of (decorated) graphs in a similar way as in \cite[\S 3.4]{KMV13}. 

For a decorated graph $(\Gamma, W; \mathrm{Or}_{\Gamma})$, we define $(\Gamma, W; \mathrm{Or}_{\Gamma})^{\vee}$ by 
\begin{equation}
    (\Gamma, W; \mathrm{Or}_{\Gamma})^{\vee} := (\Gamma, W^{\vee}; \mathrm{Or}_{\Gamma})
\end{equation} $W^{\vee}: V_{\Gamma}^{\mathrm{ext}} \rightarrow V_{H,S^{\ast}}$ is defined by composition of $W:  V_{\Gamma}^{\mathrm{ext}} \rightarrow V_{H,S^{\ast}}^{\vee}$ with the dualization map $V_{H,S^{\ast}}^{\vee} \overset{\sim}{\rightarrow} V_{H,S^{\ast}}$ with respect to the inner product on $\bC[S^{\ast}]$ and the symplectic form on $H$. Then, let us define a non-degenerate pairing $\langle - ,- \rangle : \mathcal{CD}^{\vee, k}_{H,S^{\ast}} \times \mathcal{CD}^{\vee, k}_{H,S^{\ast}}  \rightarrow \mathbb{R}$ on the space of decorated connected graphs by
\begin{equation}\label{eq:inner_prod_graph}
	\langle \Gamma_1, \Gamma_2 \rangle := \delta_{\Gamma_1, \Gamma_2^{\vee}} |\Aut(\Gamma_1)| (= \delta_{\Gamma_1, \Gamma_2^{\vee}} |\Aut(\Gamma_2^{\vee})|)
\end{equation}
for $\Gamma_1, \Gamma_2 \in \mathcal{CD}^{\vee, k}_{H,S^{\ast}}$, where $\delta$ is the Kronecker delta respecting decorations,  $|\Aut(\Gamma)|$ denotes the number of automorphisms of $\Gamma$ as decorated graphs.

In terms of the pairing and taking dual to $V^{\vee}_{H,S^{\ast}}$, we can define an isomorphism
\begin{equation}
	\mathcal{CD}^{\vee, k}_{H,S^{\ast}} \overset{\sim}{\rightarrow} \mathcal{CD}^{k}_{H,S^{\ast}}, \quad (\Gamma, W; \mathrm{Or}_{\Gamma}) \mapsto \langle (\Gamma, W; \mathrm{Or}_{\Gamma}), -\rangle.
\end{equation}

\subsubsection{Cohomology groups $H_{\partial_{\Delta}}^{\bullet}(\mathcal{CD}^{\vee, \bullet}_{H,S^{\ast}})$}
Next, we introduce the cohomology groups $H_{\partial_{\Delta}}^{\bullet}(\mathcal{CD}^{\vee, \bullet}_{H,S^{\ast}})$. Note that $\partial_{\Delta}$ increases defect degree by 1 and preserves connected components of graphs, one can define the complex $(\mathcal{CD}^{\vee, \bullet}_{H,S^{\ast},\mathrm{def}}, \partial_{\Delta})$ where $\mathcal{CD}^{\vee, \bullet}_{H,S^{\ast},\mathrm{def}}$ is defined as
\begin{equation}
    \mathcal{CD}^{\vee, \bullet}_{H,S^{\ast},\mathrm{def} } = \bigoplus_{d \geq 0} \mathcal{CD}^{\vee, d}_{H,S^{\ast},\mathrm{def}}
\end{equation}
where $\mathcal{CD}^{\vee, d}_{H,S^{\ast},\mathrm{def}} \subset \mathcal{CD}^{\vee, \bullet}_{H,S^{\ast}}$ denotes the subspace consisting of connected graphs with defect degree $d$. Define
\begin{equation}
    H_{\partial_{\Delta}}^{\bullet}(\mathcal{CD}^{\vee, \bullet}_{H,S^{\ast}}) := H^{\bullet}((\mathcal{CD}^{\vee, \bullet}_{H,S^{\ast},\mathrm{def}}, \partial_{\Delta}))
\end{equation}
Since $\partial_{\Delta}$ preserves the (loop) order $\ord$, the decomposition \ref{eq:dec_cd} induces the corresponding decomposition in cohomology groups as follows:
\begin{equation}
    H_{\partial_{\Delta}}^{\bullet}(\mathcal{CD}^{\vee, \bullet}_{H,S^{\ast}}) = \bigoplus_{d \geq 0} H_{\partial_{\Delta}}^{d}(\mathcal{CD}^{\vee, \bullet}_{H,S^{\ast}}) = \bigoplus_{d \geq 0} \bigoplus_{l \geq -1} H_{\partial_{\Delta}}^{d}(\mathcal{CD}^{\vee, (\bullet, -l)}_{H,S^{\ast}})
\end{equation}
where $H_{\partial_{\Delta}}^{d}(\mathcal{CD}^{\vee, (\bullet, -l)}_{H,S^{\ast}})$ denotes the $d$-th cohomology group of the subcomplex $(\bigoplus_d \mathcal{CD}^{\vee, (d, -l)}_{H,S^{\ast}}), \partial_{\Delta})$ with $\ord = -l$. 

In particular, the 0-th cohomology group is defined by
\begin{equation}
    H_{\partial_{\Delta}}^0(\mathcal{CD}^{\vee, \bullet}_{H,S^{\ast}}) = \Ker\left(\partial_{\Delta}: \mathcal{CD}^{\vee, \bullet}_{H,S^{\ast}, \mathrm{def}} \rightarrow \mathcal{CD}^{\vee, \bullet+1}_{H,S^{\ast}, \mathrm{def}}\right).
\end{equation}
and if we focus on total degree $\geq 0$ part, the $0$-th cohomology group $H_{\partial_{\Delta}}^{0}(\mathcal{CD}^{\vee, \bullet}_{H,S^{\ast}})$ has truncated decomposition as
\begin{equation}
\begin{split}
    H_{\partial_{\Delta}}^0(\mathcal{CD}^{\vee, \bullet}_{H,S^{\ast}})_{d -l \geq 0} =& H_{\partial_{\Delta}}^{0}(\mathcal{CD}^{\vee, (\bullet, 1)}_{H,S^{\ast}}) \oplus H_{\partial_{\Delta}}^{0}(\mathcal{CD}^{\vee, (\bullet, 0)}_{H,S^{\ast}})  \\
    =& H_{\partial_{\Delta}}^0(\mathcal{CD}^{\vee, \bullet}_{H,S^{\ast}})_{\text{tree}} \oplus H_{\partial_{\Delta}}^0(\mathcal{CD}^{\vee, \bullet}_{H,S^{\ast}})_{\text{$1$-loop}}
    \end{split}
\end{equation}
where $H_{\partial_{\Delta}}^0(\mathcal{CD}^{\vee, \bullet}_{H,S^{\ast}})_{\text{tree}}$ and  $H_{\partial_{\Delta}}^0(\mathcal{CD}^{\vee, \bullet}_{H,S^{\ast}})_{\text{$1$-loop}}$ denote the tree part and 1-loop part respectively. 

\begin{lemma}
    Let $\Ker \partial_{\Delta} \subset\mathcal{D}^{\vee, \bullet}_{H,S^{\ast}}$ be the kernek of the differential operator $\partial_{\Delta}$. Then, $\Ker \partial_{\Delta}$ is preserved by $\partial_{\mathrm{Cas}} + \partial_{S^{\ast}}$.
\end{lemma}
\begin{proof}
    Since $\partial_{\mathrm{Cas}}$ and $\partial_{\Delta}$ anticommutts, it suffices to show that $\partial_{S^{\ast}}$ presrves $\Ker \partial_{\Delta}$. For any $\Gamma \in  \Ker \partial_{\Delta}$, consider $\partial_{\Delta} \partial_{S^{\ast}} \Gamma$. The differential $\partial_{S^{\ast}}$ is the sum of the differentials associated with $S^{\ast}$-decorated external edges, whereas $\partial_{\Delta}$ is the sum of the differentials associated with internal edges. Therefore, the nontrivial contribution to $\partial_{\Delta}$ is contractions associated with internal edges $F$ of $\Gamma$ which is not adjacent to any $S^{\ast}$-decorated external edge of $\Gamma$. In this case, as in the proof of \cite[Theorem 6.4]{GoncharovHodge1}, $\partial_{S^{\ast}}$ and $\partial_{\Delta}$ anticommutes. Therefore, $\partial_{\Delta} \partial_{S^{\ast}} \Gamma=0$, and hence we conclude that $\partial_{\Delta} \circ (\partial_{\mathrm{Cas}}+\partial_{S^{\ast}}) \Gamma=0$ for any $\Gamma \in  \Ker \partial_{\Delta}$.
\end{proof}
Since $\partial^2 =0$  and $\partial_{\mathrm{Cas}} + \partial_{S^{\ast}}$ preserve the $\Ker \partial_{\Delta}$ in $\mathcal{D}^{\vee, \bullet}_{H,S^{\ast}}$, $\partial_{\mathrm{Cas}} + \partial_{S^{\ast}}$ induces a DG Lie coalgebra structure on $H_{\partial_{\Delta}}^0(\mathcal{CD}^{\vee, \bullet}_{H,S^{\ast}})$. We denote the associated differential and the Lie cobracket by $\delta^{\vee}$ and $[-, -]^{\vee}$ respectively. Corresponding to the decomposition $\partial_{\mathrm{Cas}} + \partial_{S^{\ast}}$, we have the decomposition
\begin{equation}
    \delta^{\vee} = \delta_{\mathrm{Cas}}^{\vee} + \delta_{S^{\ast}}^{\vee}:  H_{\partial_{\Delta}}^0(\mathcal{CD}^{\vee, (\bullet, -l)}_{H,S^{\ast}}) \rightarrow H_{\partial_{\Delta}}^0(\mathcal{CD}^{\vee, (\bullet, -(l-1)}_{H,S^{\ast}})
\end{equation}
and 
\begin{equation}
    [-, -]^{\vee} = [-, -]^{\vee}_{\mathrm{Cas}} + [-, -]^{\vee}_{S^{\ast}}: H^0_{\partial_{\Delta}}(\mathcal{CD}^{\bullet}_{H,S^{\ast}}) \rightarrow S^2(H^0_{\partial_{\Delta}}(\mathcal{CD}^{\bullet}_{H,S^{\ast}})) 
\end{equation}

Dually, we have
\begin{equation}
	H_0^{\partial_{\Delta}}(\mathcal{CD}^{\bullet}_{H,S^{\ast}}):= \mathrm{Coker}\left(\partial_{\Delta}^{\vee}: \mathcal{CD}^{1}_{H,S^{\ast}, \mathrm{def}} \rightarrow \mathcal{CD}^{0}_{H,S^{\ast},\mathrm{def}}\right)\simeq  \mathcal{CD}^{0}_{H,S^{\ast},\mathrm{def}} \ / \ \text{IHX}
\end{equation}
where IHX means the linear relation given as in Figure \ref{fig:IHX} and the last isomorphism can be shown by applying the same arguments in the proof of \cite[Proposition 3.29]{KMV13}.

\begin{figure}[h]
\captionsetup{margin=2cm}
 \centering
 \begin{tikzpicture}
\begin{scope}
\draw[densely dotted] (1.0,0.5) -- (-0,0.5);
\draw[densely dotted] (-0.0,-0.5) -- (1.0,-0.5);
\draw[densely dotted] (0.5,0.5) -- (0.5,-0.5);
\node at (0.5,0.5) {$\bullet$};
\node at (0.5,-0.5) {$\bullet$};
\node at (-0.3, 0.5) {$i$};
\node at (1.3, 0.5) {$j$};
\node at (-0.3, -0.5) {$l$};
\node at (1.3, -0.5) {$k$};
\node at (0.8, 0.0) {$m$};
\node at (2, 0) {$+$};
\node at (5, 0) {$+$};
\node at (8, 0) {$=$};
\end{scope}
\begin{scope}[xshift=3cm]
\draw[densely dotted] (-0.0,-0.5) -- (-0,0.5);
\draw[densely dotted] (1.0,0.5) -- (1.0,-0.5);
\draw[densely dotted] (-0.0,0) -- (1.0,0);
\node at (1,0) {$\bullet$};
\node at (0,0) {$\bullet$};
\node at (-0.3, 0.5) {$i$};
\node at (1.3, 0.5) {$j$};
\node at (-0.3, -0.5) {$l$};
\node at (1.3, -0.5) {$k$};
\node at (0.5, 0.3) {$m$};
\end{scope}
\begin{scope}[xshift=6cm]
\draw[densely dotted] (1.0,-0.5) --(-0,0.5);
\draw[densely dotted] (1.0,0.5) -- (-0.0,-0.5);
\draw[densely dotted] (0.2,-0.3) -- (0.8,-0.3);
\node at (0.2,-0.3) {$\bullet$};
\node at (0.8,-0.3) {$\bullet$};
\node at (-0.3, 0.5) {$i$};
\node at (1.3, 0.5) {$j$};
\node at (-0.3, -0.5) {$l$};
\node at (1.3, -0.5) {$k$};
\node at (0.5, -0.5) {$m$};
\end{scope}
\begin{scope}[xshift=9cm]
\node at (0,0) {$0.$};
\end{scope}
\end{tikzpicture}
 \caption[IHX relation]{IHX relation. Here, three decorated graphs are identical except within the depicted region. The labeling by $i,j,k,l,m$ corresponds to some numbering of edges of each graph}
 \label{fig:IHX}
\end{figure}

Similar to the above, the differential $\partial^{\vee}$ gives rise to an DG Lie algebra structure on $H_0^{\partial_{\Delta}}(\mathcal{CD}^{\bullet}_{H,S^{\ast}})$ given as
\begin{equation}
    \delta = \delta_{\mathrm{Cas}} + \delta_{S^{\ast}}: H^{\partial_{\Delta}}_0(\mathcal{CD}^{\vee, (\bullet, -(l-1))}_{H,S^{\ast}}) \rightarrow H^{\partial_{\Delta}}_0(\mathcal{CD}^{\vee, (\bullet, -l)}_{H,S^{\ast}})
\end{equation}
and 
\begin{equation}
	[- ,- ]= [-, -]_{\mathrm{Cas}} + [-,-]_{S^{\ast}} : S^2(H_0^{\partial_{\Delta}}(\mathcal{CD}^{\bullet}_{H,S^{\ast}})) \rightarrow H_0^{\partial_{\Delta}}(\mathcal{CD}^{\bullet}_{H,S^{\ast}}).
\end{equation}
More concrete definitions of these maps are described as follows. 

\noindent
(i) The map $\delta_{\mathrm{Cas}}$. For $\Gamma=[(\Gamma, W; \mathrm{Or}_{\Gamma})] \in H_0^{\partial_{\Delta}}(\mathcal{CD}^{\bullet}_{H,S^{\ast}})$, we take two special decorated external edges $e=(v,\alpha)$ and $e'=(w, \beta)$ where we denote by $v, w$  internal vertices and by $\alpha, \beta$ external vertices decorated by differential 1-forms $\alpha, \beta$ by abuse of notation. Then, we define a graph $\Gamma':=\Gamma^v_w \selfconn$ by gluing $e$ and $e'$ along their external vertices as Figure \ref{fig:dual_part_cas}. By forgetting the decorations $\alpha, \beta$, we obtain the induced decoration $W'$ on $\Gamma'=\Gamma^v_w \selfconn$. We choose the orientation so that $\mathrm{Or}_{\Gamma'} = E \wedge \mathrm{Or}_{\Gamma \setminus \{e, e'\}}$ and $\mathrm{Or}_{\Gamma} = e \wedge e' \wedge \mathrm{Or}_{\Gamma\setminus\{e, e'\}}$ where $E$ is the edge obtained by gluing $e$ and $e'$. Then, by summing over all distinct pairs of special decorated external edges with weight $\langle \alpha, \beta \rangle$, we define 
\begin{equation}
    \delta_{\mathrm{Cas}}([(\Gamma, W; \mathrm{Or}_{\Gamma})]) := \sum_{e=(v,\alpha), e'=(w, \beta) \in E_{\Gamma}^{\ext, \mathrm{sp}}} \langle \alpha, \beta \rangle \cdot [(\Gamma'=\Gamma^v_w\selfconn , W'; \mathrm{Or}_{\Gamma'})].
\end{equation}
For simplicity of notation, we often denote it as
\begin{equation}
    \delta_{\mathrm{Cas}}(\Gamma) := \sum_{e=(v,\alpha), e'=(w, \beta) \in E_{\Gamma}^{\ext, \mathrm{sp}}} \langle \alpha, \beta \rangle \cdot \Gamma^v_w\selfconn
\end{equation}
by suppressing decorations and orientations in the sequel.

\noindent
(ii) The map $\delta_{S^{\ast}}$. For $\Gamma=[(\Gamma, W; \mathrm{Or}_{\Gamma})] \in H_0^{\partial_{\Delta}}(\mathcal{CD}^{\bullet}_{H,S^{\ast}})$, we set
\begin{equation}
     \delta_{S^{\ast}}(\Gamma) := \sum_{e=(v,s), e'=(w, s') \in E_{\Gamma}^{\ext, S^{\ast}}} \langle s, s' \rangle \cdot \Gamma^v_w \selftriconn
\end{equation}
where $\Gamma^v_w \selftriconn$ denotes the uni-trivalent graph obtained from $\Gamma$ by connecting $S^{\ast}$-deocrated edges $e=(v,s)$ and $e'=(w,s')$ as Figure \ref{fig:dual_S-deco}. Here, $v,w$ means internal vertices, and $s, s'$ are external vertices decorated by $s, s'$ respectively. The induced decorations and orientations are determined so that they are compatible with \eqref{eq:paartial_S1}.

\noindent
(iii) The maps $[-, -]_{\mathrm{Cas}}$. With the similar notations with (i), for $\Gamma, \Gamma' \in H_0^{\partial_{\Delta}}(\mathcal{CD}^{\bullet}_{H,S^{\ast}})$, we set 
\begin{equation}
    [\Gamma, \Gamma']_{\mathrm{Cas}} := \sum_{e=(v,\alpha) \in E_{\Gamma}^{\ext, \mathrm{sp}}} \sum_{e'=(w,\beta) \in E_{\Gamma'}^{\ext, \mathrm{sp}}} \langle \alpha, \beta \rangle \cdot (\Gamma \connedge
    \Gamma')
\end{equation}
where $\Gamma \tikz[baseline=-0.1ex]{\draw[densely dotted] (0,0) -- (0.5,0);
    \node at (0,0.2) {$v$};
    \node at (0.5, 0.2) {$w$};
    }
    \Gamma'$ means the uni-trivalent graph obtained from $\Gamma$ and $\Gamma'$ by connecting them along special decorated external edges $e=(v,\alpha)$ of $\Gamma$ and $e'=(w,\beta)$ of $\Gamma'$, which is graphically described as Figure \ref{fig:dual_Ca_S}. The way of defining the decorations and orientations of the resulting graphs is given to be compatible with \eqref{eq:partial_cas1}.

\noindent
(iv) The map $[-,-]_{S^{\ast}}$. With the similar notations with (ii),  we set, for $\Gamma, \Gamma' \in H_0^{\partial_{\Delta}}(\mathcal{CD}^{\bullet}_{H,S^{\ast}})$,
\begin{equation}
    [\Gamma, \Gamma']_{S^{\ast}} := \sum_{e=(v,s) \in E_{\Gamma}^{\ext, S^{\ast}}} \sum_{e'=(w,s') \in E_{\Gamma'}^{\ext, S^{\ast}}} \langle s, s'\rangle \cdot (\Gamma \conntriedge
    \Gamma')
\end{equation}
where $\Gamma \conntriedge
    \Gamma'$ means similarly the uni-trivalent graph obtained from $\Gamma$ and $\Gamma'$ by connecting them along $S^{\ast}$-deocrated external edges $e=(v,s)$ of $\Gamma$ and $e'=(w,s')$ of $\Gamma'$ as Figure \ref{fig:dual_S-deco}. The decorations and orientations are given by the same manner as (i).

\begin{figure}[h]
 \centering
 \begin{tikzpicture}
 \begin{scope}[xshift=6cm]
 	\draw[densely dotted] (0,0) -- (1,0);
 	\draw[densely dotted] (0,0) -- (-0.5, 0.5);
 	\draw[densely dotted] (0,0) -- (-0.5, -0.5);
 	\draw[densely dotted] (1,0) -- (1.5, 0.5);
 	\draw[densely dotted] (1,0) -- (1.5, -0.5);
   \node at (1,0) {$\bullet$}; 
 \node at (0,0) {$\bullet$}; 
 \node at (0, -0.5) {$v$};
	\node at (1, -0.5) {$w$};
  \node at (-1.3, 0) {$\langle \alpha, \beta\rangle$};
  \node at (-2.5, 0) {$\longrightarrow$};
  \end{scope}
  \begin{scope}[xshift=-5cm]
	\draw[densely dotted] (4.5,0.5) -- (5, 0);
	\draw[densely dotted] (5,0) -- (4.5, -0.5);
	\draw[densely dotted] (5,0) -- (5.5,0);
	\node at (5.5,0) {$\bullet$}; 
	\node at (6.5, 0) {$\bullet$}; 
 \node at (7,0) {$\bullet$}; 
 \node at (5,0) {$\bullet$}; 
	\draw[densely dotted] (6.5,0) -- (7,0);
	\draw[densely dotted] (7, 0) -- (7.5, 0.5);
	\draw[densely dotted] (7, 0) -- (7.5, -0.5);
	\node at (5.5, 0.5) {$\alpha$};
	\node at (6.5, 0.5) {$\beta$};
 \node at (5, -0.5) {$v$};
	\node at (7, -0.5) {$w$};
 \end{scope}
 	 \end{tikzpicture}
 \caption{The map $\partial_{\mathrm{Cas}}^{\vee}$}
 \label{fig:dual_part_cas}
\end{figure}

\begin{figure}[h]
 \centering
 \begin{tikzpicture}
 \begin{scope}[xshift=6cm]
  \node at (0,0) {$\bullet$}; 
 \node at (0, -0.5) {$v$};
  \node at (-1.3, 0) {$\langle \alpha, \beta\rangle$};
  \node at (-2.5, 0) {$\longrightarrow$};
 	\draw[densely dotted] (0,0) -- (0.9,0);
 	\draw[densely dotted] (0,0) -- (-0.5, 0.5);
 	\draw[densely dotted] (0,0) -- (-0.5, -0.5);
 	\node at (0.5, 0.5) {$E$};
  \node at (1,0) {$\bullet$};
 	\node at (1.4, 0) {$s$};
  \end{scope}
  \begin{scope}[xshift=-5cm]
	\draw[densely dotted] (4.5,0.5) -- (5, 0);
	\draw[densely dotted] (5,0) -- (4.5, -0.5);
	\draw[densely dotted] (5,0) -- (5.5,0);
	\node at (5.5,0) {$\bullet$};
	\node at (6.5, 0) {$\bullet$};
	\draw[densely dotted] (6.5,0) -- (6.9,0);
	\node at (7,0) {$\bullet$};
	\node at (7.4, 0) {$s$};
 \node at (5.5, 0.5) {$\alpha$};
	\node at (6.5, 0.5) {$\beta$};
 \node at (5,0) {$\bullet$}; 
 \node at (5, -0.5) {$v$};
 \end{scope}
 	 \end{tikzpicture}
 \caption{The map $\partial_{\mathrm{Cas}}^{\vee}$ for an $S^{\ast}$-decorated external edge}
 \label{fig:dual_Ca_S}
\end{figure}

\begin{figure}[h]
 \centering
 \begin{tikzpicture}
 \begin{scope}[xshift=6.5cm, yshift=0.5cm]
 	\draw[densely dotted] (0,0) -- (0, -0.9);
 	\draw[densely dotted] (0,0) -- (-0.7071, 0.7071);
  \node at (-0.7071, 0.7071) {$\bullet$}; 
  \node at (-1, 0.9) {$v$}; 
  \node at (1, 0.9) {$w$}; 
   \node at (0.7071, 0.7071) {$\bullet$}; 
 	\draw[densely dotted] (0,0) -- (0.7071, 0.7071);
 	\draw[densely dotted] (0.7071, 0.7071) -- (0.7071+0.8, 0.7071);
 	\draw[densely dotted] (0.7071, 0.7071) -- (0.7071, 0.7071+0.8);
 	\draw[densely dotted] (-0.7071, 0.7071) -- (-0.7071-0.8, 0.7071);
 	\draw[densely dotted] (-0.7071, 0.7071) -- (-0.7071, 0.7071+0.8);
 	\node at (0,-1) {$\bullet$};
  \node at (0, 0) {$\bullet$};
 	\node at (0,-1.5) {$s_i$};
  \end{scope}
 	\node at (3.5, 0) {$\longrightarrow$};
  \node at (4.5, 0) {$\delta_{ij}$};
 \begin{scope}[xshift=0cm] 
 	\draw[densely dotted] (0,0) -- (-0.7071, 0.7071);
  \node at (-0.7071, 0.7071) {$\bullet$}; 
  \node at (-1, 0.9) {$v$}; 
 	\draw[densely dotted] (-0.7071, 0.7071) -- (-0.7071-0.8, 0.7071);
 	\draw[densely dotted] (-0.7071, 0.7071) -- (-0.7071, 0.7071+0.8);
 	\node at  (0,0) {$\bullet$};
 	\node at (0,-0.5) {$s_i$};
 \end{scope}
\begin{scope}[xshift=1cm] 
 \node at (1, 0.9) {$w$}; 
   \node at (0.7071, 0.7071) {$\bullet$}; 
 	\draw[densely dotted] (0,0) -- (0.7071, 0.7071);
 	\draw[densely dotted] (0.7071, 0.7071) -- (0.7071+0.8, 0.7071);
 	\draw[densely dotted] (0.7071, 0.7071) -- (0.7071, 0.7071+0.8);
 	\node at  (0,0) {$\bullet$};
  
 	\node at (0,-0.5) {$s_j$};
 \end{scope}
 \end{tikzpicture}
 \caption{The map $\partial_{S^{\ast}}^{\vee}$}
 \label{fig:dual_S-deco}
\end{figure}
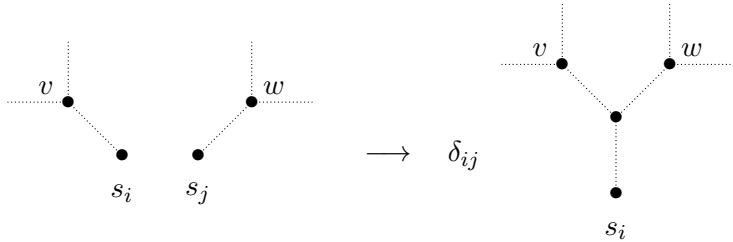

\section{Quantum master equation and (generalized) Hodge correlator twistor connection}\label{section:5}
This section introduces our generalized version of the Hodge correlator twistor connection which satisfies some Maurer--Cartan type equation analogous to the quantum master equation in BV formalism (cf. \cite{Cos11}, \cite{CW}, \cite{I}). For this, we first recall the notions of quantum master equation and master homotopy (Section \ref{ssection:QMEMH}), then we define an effective action $\mathfrak{s}$ in graph complex level (Section \ref{ssection:FEA})  and our generalized version of Hodge correlator twistor connection as its image under a morphism assigning graphs to associated integrations (Section \ref{ssection:HC_QME}).

\subsection{Quantum master equation and master homotopy}\label{ssection:QMEMH}
Following \cite{I}, we recall the quantum master equation and its variation called the quantum master homotopy equation for the case of finite dimensional vector spaces. For more details, see also \cite{CM}. 

Let $(H,  \langle \cdot, \cdot \rangle)$ be a finite dimensional $(\mathbb{Z}/2\mathbb{Z})$-graded real vector space $H = H^{\mathrm{even}} \oplus H^{\mathrm{odd}}$ equipped with an odd symplectic form $\langle \cdot, \cdot \rangle: H \times H \rightarrow \bR$, i.e.,  non-degenerate skew-symmetric bilinear form  such that if $\langle \alpha, \beta \rangle \neq 0$ then $|\alpha| + |\beta| = 1$. Note that the existence of $\langle \cdot, \cdot \rangle$ implies that $\dim(H^{\mathrm{odd}}) = \dim(H^{\mathrm{even}})=:m$ and so $H$ is $2m$-dimensional real vector space. Then, by choosing a basis for $H$, one can consider Darboux coordinates $x_1,\ldots, x_m, y_1,\ldots, y_m$ for $H$ so that $x_i$ are even and $y_i$ are odd variables respectively.

Next, let us recall the BV Laplacian $\Delta_{\mathrm{BV}}$ and the Gerstenhaber bracket $\{\cdot, \cdot\}$ associated with $(H,  \langle \cdot, \cdot \rangle)$, which are defined as follows. Let $\mathcal{O}(H)$ be the (graded) symmetric algebra $\mathcal{O}(H) = \mathrm{Sym}(H^{\vee})=\bigoplus_{k\geq 0} \mathrm{Sym}^k (H^{\vee})$ of the dual space $H^{\vee}$ to $H$. The \textit{BV Laplacian} $\Delta_{\mathrm{BV}}$ is a second order differential operator $\Delta_{\mathrm{BV}}: \mathcal{O}(H) \rightarrow \mathcal{O}(H)$ given by
\begin{equation}\label{eq:BVLaplacian}
    \Delta_{\mathrm{BV}} := \sum_{i=1}^m \frac{\partial}{\partial x_i} \frac{\partial}{\partial y_i}.
\end{equation}
Note that $\Delta_{BV}$ does not depend on the choice of basis of $H$. The \textit{Gerstenhaber bracket} $\{\cdot, \cdot\}: \mathcal{O}(H) \times \mathcal{O}(H) \rightarrow \mathcal{O}(H)$ is defined by setting for any homogeneous elements $f$ and $g$ in $\mathcal{O}(H)$,
\begin{equation}\label{eq:Gerstenhaber}
\begin{split}
    \{f, g\} &:= \Delta_{\mathrm{BV}}(fg) - \Delta_{\mathrm{BV}}(f) g - (-1)^{|f|} f \Delta_{\mathrm{BV}}(g)\\
    & =\sum_{i=1}^m\left((-1)^{|f|} \frac{\partial f}{\partial x_i} \frac{\partial g}{\partial y_i} + \frac{\partial f}{\partial y_i} \frac{\partial g}{\partial x_i}\right),
\end{split}
\end{equation}
and extending it linearly.

Now we are in the position to explain the quantum master equation. Let  $\mathcal{O}(H)[[\hbar]] :=  \mathcal{O}(H)\otimes_{\mathbb{R}} \bR[[\hbar]]$ denote the ring of formal power series in a formal variable $\hbar$ with coefficients in $\mathcal{O}(H)$. For an even element $S \in \mathcal{O}(H)[[\hbar]]$, we call that $S$ satisfies the \textit{quantum master equation} if the following equation holds
\begin{equation}
    \Delta_{\mathrm{BV}} e^{S/\hbar} = 0.
\end{equation}
By direct computation, one sees that it is equivalent to
\begin{equation}\label{eq:QME}
    \frac{1}{2} \{ S, S \} + \hbar \Delta_{\mathrm{BV}} S = 0.
\end{equation}

We end this subsection by recalling the quantum master homotopy equation which is a one-parameter family version of the quantum master equation. Let $I = [0,1]$ be the unit interval. Let us consider the space $\Omega^{\bullet}(I) \otimes_{\mathbb{R}} \mathcal{O}(H)[[\hbar]]$ and extend the operator $\Delta_{\mathrm{BV}}$ on $\Omega^{\bullet}(I) \otimes_{\mathbb{R}} \mathcal{O}(H)[[\hbar]]$ by acting trivially on $\Omega^{\bullet}(I)$. Note that, then, $\{\cdot, \cdot\}$ also extend to $\Omega^{\bullet}(I) \otimes_{\mathbb{R}} \mathcal{O}(H)[[\hbar]]$ via $\Delta_{\mathrm{BV}}$. For an even element $\widetilde{S} \in \Omega^{\bullet}(I) \otimes_{\bR} \mathcal{O}(H)[[\hbar]]$, we say that $\widetilde{S}$ satisfies the \textit{quantum master homotopy equation} if it satisfies
\begin{equation}\label{eq:QMHE}
    d_I \widetilde{S} + \frac{1}{2} \{ \widetilde{S}, \widetilde{S}\} + \hbar \Delta_{\mathrm{BV}} \widetilde{S} = 0
\end{equation}
where $d_I$ is the trivial extension of the exterior differential operator $d_I : \Omega^{\bullet}(I) \rightarrow \Omega^{\bullet+1}(I)$ to $\Omega^{\bullet}(I) \otimes_{\mathbb{R}} \mathcal{O}(H)[[\hbar]]$. Two even elements $S, S' \in \mathcal{O}(H)[[\hbar]]$ are said to be \textit{master homotopic} if there is an even element $\widetilde{S} \in \Omega^{\bullet}(I) \otimes_{\bR} \mathcal{O}(H)[[\hbar]]$ with  $\widetilde{S}|_{t=0} = S$ and $\widetilde{S}|_{t=1} = S'$, which satisfies the quantum master equation \eqref{eq:QMHE}.
\begin{remark}
    \begin{enumerate}[(1)]
    \item Corresponding to the direct sum decomposition $\Omega^{\bullet}(I) \otimes_{\bR} \mathcal{O}(H)[[\hbar]]= \Omega^{0}(I) \otimes_{\bR} \mathcal{O}(H)[[\hbar]] \oplus \Omega^{1}(I) \otimes_{\bR} \mathcal{O}(H)[[\hbar]]$, the quantum master homotopy equation \eqref{eq:QMHE} is decomposed into the following two differential equations
    \begin{equation}
        \begin{split}
            &\frac{1}{2} \{ A(t), A(t)\} + \hbar \Delta_{\mathrm{BV}} A(t) = 0,\\
            &\frac{d}{d t} A(t) + \{ B(t), A(t)\} + \hbar \Delta_{\mathrm{BV}} B(t) = 0,
        \end{split}
    \end{equation}
    where $A(t), B(t) \in \mathcal{O}(H)[[\hbar]]$ are defined by $\widetilde{S}(t) = A(t) + B(t) dt$.
    \item If $\widetilde{S}$ is a constant function on $I$ with $\widetilde{S}|_t = S$ for any $t \in I$, then the quantum master homotopy equation \eqref{eq:QMHE} reduces to the quantum master equation \eqref{eq:QME} for $S$.
    \item When $H = 0$, the quantum master equation \eqref{eq:QME} is empty. In addition, the quantum master homotopy equation \eqref{eq:QMHE} is reduced to the differential equation
    \begin{equation}
        d_I \widetilde{S} = 0
    \end{equation}
    for $\widetilde{S} \in \Omega^{\bullet}(I) \otimes_{\mathbb{R}} \bR[[\hbar]]$ so that a solution $\widetilde{S}$ to the above equation is a constant function on $I$.
    \end{enumerate}
\end{remark}

\subsection{Formal effective action}\label{ssection:FEA}

In this section, we introduce effective action as a formal series of uni-trivalent graphs in a graph complex. 
This series encodes the algebraic structure of effective actions from which one can obtain genuine effective action through CDGA morphisms to de Rham complexes.

Let $S = \{s_1, \ldots, s_k\}$ be a finite set, $\bC[S]$ be the vector space spanned by $S$, and $H^{\vee}$ be a finite-dimensional symplectic vector space over $\bC$. We define an inner product on $\bC[S]$ by setting $(\{s_i\}, \{s_j\}) = \delta_{ij}$. We set
\begin{equation}
    V_{H,S}^{\vee} := \bC[S] \oplus H^{\vee}
\end{equation}
and $ V_{H,S}$ be its dual with respect to the inner product and the symplectic form on $\bC[S]$ and $H^{\vee}$ respectively. We choose a basis $\{\alpha_1, \ldots, \alpha_g\}$  of $H^{\vee}$ and let $\{h_1, \ldots, h_g \}$ be its dual with respect to the symplectic form. Then, we define the Casimir element as
\begin{equation}
    \gamma|\gamma^{\vee}:= \sum_{i=1}^k \{s_i\} \otimes \{s_i\} + \sum_{i=1}^g \alpha_i \otimes h_i \in V_{H,S}^{\vee} \otimes V_{H,S}.
\end{equation}

For a connected uni-trivalent graph $\Gamma$, consider a decoration $W$ such that it assigns one of the elements of the (chosen) basis of $ V_{H,S}^{\vee}$ for each external vertex. We write this decorated graph by $(\Gamma, W)$. Taking dual decoration $W^{\vee}$ of $W$ results in a decorated graph $(\Gamma, W^{\vee})$. We sometimes call the sum of all such decorations weighted by the numbers of automorphisms of decorated graphs (cf. Section \ref{ss:automorphism_graph})
\begin{equation}
    \sum_{W} \frac{1}{|\Aut((\Gamma, W))|} (\Gamma, W; \mathrm{Or}_{\Gamma}) \otimes (\Gamma, W^{\vee}; \mathrm{Or}_{\Gamma})
\end{equation}
the \textit{pairwise decoration of $\Gamma \otimes \Gamma^{\vee} $ by the Casimir element $\gamma|\gamma^{\vee}$}. Note that the sign of $(\Gamma, W; \mathrm{Or}_{\Gamma}) \otimes (\Gamma, W^{\vee}; \mathrm{Or}_{\Gamma})$ does not depend on the choice of elements of orientation torsor $\mathrm{Or}_{\Gamma}$ of $\Gamma$.

In terms of this terminology, the \textit{formal effective action}  is defined by
\begin{equation}
    \mathfrak{s}:= \sum_{\Gamma} \sum_{W} \frac{ \hbar^{l(\Gamma)}}{|\Aut((\Gamma, W))|} (\Gamma, W; \mathrm{Or}_{\Gamma}) \otimes [(\Gamma, W^{\vee}; \mathrm{Or}_{\Gamma})] \in \mathcal{CD}^{\vee,0}_{H, S^{\ast},\mathrm{def}} \otimes H_0^{\partial_{\Delta}}(\mathcal{CD}^{\bullet}_{H, S^{\ast}}[-1])[[\hbar]]
\end{equation}
where the summation runs over all the topological connected uni-trivalent graph $\Gamma$, and all the decorations of $\Gamma$ by basis of $ V_{H,S}^{\vee}$. In the sequel, for simplicity of notations, $\mathfrak{s}$ will be often denoted as 
\begin{equation}
    \mathfrak{s}= \sum_{\Gamma} \frac{\hbar^{l(\Gamma)}}{|\Aut(\Gamma)|} \Gamma \otimes \Gamma^{\vee} = \sum_{n=0}^{\infty} \sum_{\Gamma; l(\Gamma)=n} \frac{\hbar^{n}}{|\Aut(\Gamma)|} \Gamma \otimes \Gamma^{\vee}
\end{equation}
where the summation taken over all connected uni-trivalent graph $\Gamma$ with $\Gamma \otimes \Gamma^{\vee}$ pairwise deocrated by the Casimir element $\gamma|\gamma^{\vee}$.

We define the bracket operation $[-,-]$ on  $\mathcal{CD}^{\vee,0}_{H, S^{\ast},\mathrm{def}} \otimes H_0^{\partial_{\Delta}}(\mathcal{CD}^{\bullet}_{H, S^{\ast}}[-1])$ by setting
\begin{equation}
    [\Gamma_1 \otimes \Gamma_1',\Gamma_2 \otimes \Gamma_2'] := \Gamma_1 \wedge \Gamma_2 \otimes [\Gamma_1', \Gamma_2']
\end{equation}
for $\Gamma_1, \Gamma_2 \in \mathcal{CD}^{\vee,0}_{H, S^{\ast},\mathrm{def}}$ and $\Gamma_1', \Gamma_2' \in H_0^{\partial_{\Delta}}(\mathcal{CD}^{\bullet}_{H, S^{\ast}}[-1])$.
\begin{theorem}[formal quantum master equation]\label{thm:formal_qme}
    The following differential equation holds:
    \begin{equation}\label{eq:qme_formal}
        (\partial \otimes \Id) \mathfrak{s} - \frac{1}{2} [\mathfrak{s}, \mathfrak{s}] - \hbar(\Id \otimes \delta ) \mathfrak{s} = 0.
    \end{equation}
\end{theorem}

\begin{proof}
    To begin with, let us show that $(\partial \otimes \Id) \mathfrak{s} = ((\partial_{\mathrm{Cas}} + \partial_{S^{\ast}}) \otimes \Id) \mathfrak{s}$, i.e., we have $(\partial_{\Delta} \otimes \Id) \mathfrak{s} =0$. This follows from IHX relation on $H_0^{\partial_{\Delta}}(\mathcal{CD}^{\bullet}_{H, S^{\ast}})$ as follows. Since $\partial_{\Delta}$ preserves the number of loops and external vertices of uni-trivalent graphs and their decorations, it is enough to prove it for
    \begin{equation}\label{eq:5.2.4}
        (\partial_{\Delta} \otimes \Id)\left(\sum_{\substack{\text{$\Gamma$: connected}\\\text{$l(\Gamma)=n, |V_{\Gamma}^{\ext}|=k$}}}\frac{1}{|\Aut(\Gamma)|} \Gamma \otimes \Gamma^{\vee} \right) \quad (n,k \geq 0)
    \end{equation}
    with a fixed decoration on external vertices. Note that, for uni-trivalent graph $\Gamma$ with $n$-loops and $k$-external vertices, the equations $|V_{\Gamma}^{\mathrm{int}}|= k + 2n -2$ and $|E_{\Gamma}^{\mathrm{int}}|= k + 3n - 3$ hold.
    As in \cite{KL23}, we rewrite the sum in \eqref{eq:5.2.4} in terms of partitions of the set of half-edges to simplify arguments. Let $P_{k,n}$ be the set of partitions of the set $h_{k,n}$ consinsting of $3|V_{\Gamma}^{\mathrm{int}}|= 3k + 6n -6$ internal half-edges and $|V_{\Gamma}^{\mathrm{int}}| =k$ external half-edges into $|E_{\Gamma}|= 2k + 3n - 3$ pairs. Then, fixing a partition $V$ of $3|V_{\Gamma}^{\mathrm{int}}|$ half-edges into the $|V_{\Gamma}^{\mathrm{int}}|$ sets of cardinality 3, we consider the surjective map
    \begin{equation}
        \pi_{k,n}: P_{k,n} \rightarrow G_{k,n}
    \end{equation}
    which maps each partition $E$ in $P_{k,n}$ to the equivalent class of a (topological) graph $\Gamma$ with $k$-external vertices and $n$-loops corresponding to partitions $(E,V)$ with the fixed decoration. Here, $G_{k,n}$ denotes the set of equivalence classes of such decorated graphs. For a topological decorated graph $\Gamma$ represented by $(h_{k,n}, (E,V))$, we denote by $G_{V}$ the subgroup of the permutation group of $h_{k,n}$ preserving the partition $V$ and respecting the decoration. Then, we have
    \begin{equation}
        |\pi_{k,n}^{-1}(\Gamma)| = \frac{|G_{V}|}{|\Aut(\Gamma)|}
    \end{equation}
    so that 
    \begin{equation}
        \begin{split}
            &\sum_{\substack{\text{$\Gamma$: connected}\\\text{$l(\Gamma)=n, |V_{\Gamma}^{\ext}|=k$}}}\frac{1}{|\Aut(\Gamma)|} \Gamma \otimes \Gamma^{\vee} \\
            =&  \sum_{\substack{\text{$\Gamma$: connected}\\\text{$l(\Gamma)=n, |V_{\Gamma}^{\ext}|=k$}}}\frac{|\pi_{k,n}^{-1}(\Gamma)|}{|G_{V}|} \Gamma \otimes \Gamma^{\vee}\\
            =& \frac{1}{|G_V| (2k + 3n -3)!} \sum_{\substack{E \in P_{n,k} \\ \pi_{n,k}(E): \text{connected}}} \Gamma(h_{k,n}, (E,V)) \otimes \Gamma(h_{k,n}, (E,V))^{\vee}.
        \end{split}
    \end{equation}
    Here, the last summation runs over all partitions $E \in P_{n,k}$ which give connected uni-trivalent graphs $\Gamma(h_{k,n}, (E,V))$ and all numberings of edges of $\Gamma(h_{k,n}, (E,V))$. Note that, after applying the map $\partial_{\Delta} \otimes \Id$, we get terms of the form $\partial_{\Delta, e} \Gamma \otimes \Gamma^{\vee}$ where $\partial_{\Delta, e} \Gamma$ is the connected graph, whose all vertices are uni-trivalent but only one vertex is valency 4, obtained from $\Gamma$ by contracting an internal edge $e \in E_{\Gamma}^{\mathrm{int}}$. Then, all such terms vanish by IHX relation on $H_0^{\partial_{\Delta}}(\mathcal{CD}_{H,S^{\ast}}^{\bullet})$, since all the possible ways of partitions which give connected uni-trivalent graphs are considered in the last equation. Indeed, setting $\Gamma':= \partial_{\Delta, e} \Gamma$,  we know that there are only 3 possible ways to insert an edge to $\Gamma'$ at the valency 4 vertex to obtain uni-trivalent graphs. We denote by $\Gamma_I, \Gamma_H, \Gamma_X$ the resulting graphs obtained by inserting IHX-shaped graphs as in Figure \ref{fig:IHX} respectively. Thus, for $\Gamma'$ we must have the terms
    \begin{equation}
        \Gamma' \otimes \Gamma_I + \Gamma' \otimes \Gamma_H + \Gamma' \otimes \Gamma_X
    \end{equation}
    with the same coefficient, and so they vanish by IHX relations.

    Therefore, the equation \eqref{eq:qme_formal} reduces to the following equation.
     \begin{equation}\label{eq:qme_formal_dash}
        ((\partial_{\mathrm{Cas}} + \partial_{S^{\ast}}) \otimes \Id) \mathfrak{s}  - \frac{1}{2} [\mathfrak{s}, \mathfrak{s}] - \hbar(\Id \otimes \delta ) \mathfrak{s} = 0.
    \end{equation}
   Note that the above equation \eqref{eq:qme_formal_dash} holds if the following two equations hold:
    \begin{equation}\label{eq:qme_formal_dash_divided}
        (\partial_{\natural}  \otimes \Id) \mathfrak{s}  - \frac{1}{2} [\mathfrak{s}, \mathfrak{s}]_{\natural} - \hbar(\Id \otimes \delta_{\natural} ) \mathfrak{s} = 0,\quad \natural \in \{\mathrm{Cas}, S^{\ast}\}.
    \end{equation}
    First, we shall consider the case of \eqref{eq:qme_formal_dash_divided} with $\natural = \mathrm{Cas}$. 
    By applying $\partial_{\mathrm{Cas}}$ to a uni-trivalent graph $\Gamma$, one obtains connected uni-trivalent graphs with exactly one less loop than $\Gamma$ or disjoint union of two connected uni-trivalent graphs whose total number of loops is the same as $\Gamma$. Let $\Gamma \setminus \{e\}$ be the connected uni-trivalent graph which is obtained from $\Gamma$ cutting along an internal edge $e$. 
    Then, consider the subgroup of $\Aut(\Gamma)$ consisting of automorphisms fixing the pair of two half-edges representing $e$ and denote it by $\Aut(\Gamma)_e$.  Therefore, we have 
    \begin{equation}
         |\Aut(\Gamma)\cdot e| = \frac{|\Aut(\Gamma)|}{|\Aut(\Gamma)_e)|}.
    \end{equation}
    where $\Aut(\Gamma) \cdot e$ denotes the $\Aut(\Gamma)$-orbit of $e$.
    By definition of automorphisms of a (Feynman) graph, one obtains $|\Aut(\Gamma)\cdot e|$-connected graphs which are isomorphic to $\Gamma \setminus \{e\}$ from $\Gamma$. This implies that we get the term (up to sign)
    \begin{equation}
        \frac{|\Aut(\Gamma)\cdot e|}{|\Aut(\Gamma)|} [\Gamma \setminus \{e\}] \otimes \Gamma^{\vee}
    \end{equation}
    from $\frac{1}{|\Aut(\Gamma)|}[\Gamma] \otimes \Gamma$ via $\partial_{\mathrm{Cas}} \otimes \Id$. This term is also obtained from $\hbar(\Id \otimes \delta_{\mathrm{Cas}} ) \mathfrak{s}$, since we have
    \begin{equation}
    \begin{split}
        & \hbar (\Id \otimes \delta_{\mathrm{Cas}})\left(\frac{1}{|\Aut(\Gamma \setminus \{e\})|} [\Gamma \setminus \{e\}] \otimes [\Gamma \setminus \{e\}^{\vee}]\right)\\
        =&  \hbar \left(\pm \frac{|\Aut(\Gamma \setminus \{e\}) \cdot e|}{|\Aut(\Gamma \setminus \{e\})|} [\Gamma \setminus \{e\}] \otimes \Gamma^{\vee} + \cdots \right)\\
        =&  \hbar \left(\pm \frac{|\Aut(\Gamma) \cdot e|}{|\Aut(\Gamma)|} [\Gamma \setminus \{e\} ]\otimes \Gamma^{\vee} + \cdots \right)
        \end{split}
    \end{equation}
    where $\Aut(\Gamma \setminus \{e\}) \cdot e$ denotes the $\Aut(\Gamma \setminus \{e\})$-orbit of $e$ as pair of half-edges. Indeed, the subgroups of $\Aut(\Gamma)$ and $\Aut(\Gamma \setminus \{e\})$ fixing the pair $e$ coincide, i.e., $\Aut(\Gamma)_e = \Aut(\Gamma \setminus \{e\})_e$ since as subgroups of permutations of the set of half-edges $\Aut(\Gamma)$ and $\Aut(\Gamma \setminus \{e\})$ differ only by the condition preserving $e$ or not. Therefore, the following equations
    \begin{equation}
        |\Aut(\Gamma)\cdot e| = \frac{|\Aut(\Gamma)|}{|\Aut(\Gamma)_e|}, \quad   |\Aut(\Gamma \setminus \{e\})\cdot e|= \frac{|\Aut(\Gamma \setminus \{e\})|}{|\Aut(\Gamma \setminus \{e\})_e|}
    \end{equation}
    and $|\Aut(\Gamma)_e| = |\Aut(\Gamma \setminus \{e\})_e|$ read
    \begin{equation}\label{eq:conn_fixing_sub}
        \frac{|\Aut(\Gamma)\cdot e|}{|\Aut(\Gamma)|} = \frac{|\Aut(\Gamma \setminus \{e\})\cdot e|}{|\Aut(\Gamma \setminus \{e\})|}.
    \end{equation}
    Next, we consider the case that $\partial_{\mathrm{Cas}}$ breaks a connected uni-trivalent graph $\Gamma$ into two connected uni-trivalent graphs $\Gamma'$ and $\Gamma''$. As we will see below, the essential idea of the proof is the same as above. Similarly, let $e$ be an internal edge of $\Gamma$ along which we cut and $\Aut(\Gamma) \cdot e$ be the $\Aut(\Gamma)$-orbit of $e$. Then, we have the relation
    \begin{equation}
        |\Aut(\Gamma) \cdot e| = \frac{|\Aut(\Gamma)|}{|\Aut(\Gamma)_e|}.
    \end{equation}
    where $\Aut(\Gamma)_e$ denotes the stabilizer of $e$ under the action of $\Aut(\Gamma)$.
    Thus,  one obtains the term (up to sign) from  $\frac{1}{|\Aut(\Gamma)|}[\Gamma] \otimes \Gamma^{\vee}$ via $\partial_{\mathrm{Cas}} \otimes \Id$
    \begin{equation}
        \pm \frac{|\Aut(\Gamma)\cdot e|}{|\Aut(\Gamma)|} [\Gamma \setminus \{e\}] \otimes \Gamma^{\vee} = \frac{|\Aut(\Gamma)\cdot e|}{|\Aut(\Gamma)|} [\Gamma' \sqcup \Gamma''] \otimes \Gamma^{\vee}.
    \end{equation}
    These terms are also obtained from $\frac{1}{2} [\mathfrak{s}, \mathfrak{s}]_{\mathrm{Cas}}$, since we have
    \begin{equation}
    \begin{split}
        & \frac{1}{2} \left[\frac{1}{|\Aut(\Gamma')|} \Gamma' \otimes (\Gamma')^{\vee}, \frac{1}{|\Aut(\Gamma'')|} \Gamma'' \otimes (\Gamma'')^{\vee}\right]_{\mathrm{Cas}}\\
        &= \pm \frac{|\Aut(\Gamma \setminus \{e\}) \cdot e|}{2} \cdot \frac{1}{|\Aut(\Gamma')|} \cdot \frac{1}{|\Aut(\Gamma'')|}  [\Gamma' \sqcup \Gamma''] \otimes \Gamma^{\vee} + \cdots\\
        &= \pm  \frac{|\Aut(\Gamma \setminus \{e\}) \cdot e|}{|\Aut(\Gamma \setminus \{e\})|} [\Gamma \setminus \{e\}] \otimes \Gamma^{\vee} + \cdot\\
        &= \pm \frac{|\Aut(\Gamma)\cdot e|}{|\Aut(\Gamma)|} [\Gamma' \sqcup \Gamma''] \otimes \Gamma^{\vee} + \cdots.
        \end{split}
    \end{equation}
    Here,  we use the identities \eqref{eq:conn_fixing_sub} and
    \begin{equation}
        |\Aut(\Gamma \setminus \{e\})| = 2 \cdot |\Aut(\Gamma')| \cdot |\Aut(\Gamma'')|. 
    \end{equation}
    Since $\partial_{\mathrm{Cas}}$ and its dual $\partial_{\mathrm{Cas}}^{\vee}$ yield compatible orientations on uni-trivalent graphs, they yields same terms with the same sign and therefore \eqref{eq:qme_formal_dash_divided} with $\natural = \mathrm{Cas}$ holds. 
    
    Second, we consider the case of \eqref{eq:qme_formal_dash_divided} with $\natural = S^{\ast}$. This case follows the essentially same idea as the former case. Instead of connecting or disconnecting two half-edges representing an internal edge, in this case, we connect two $S^{\ast}$-decorated half-edges into $Y$-shaped subgraph or cut along a small neighbourhood of one $S^{\ast}$-decorated external edge $e$. However, the proof goes along the same way as the former case. This completes the proof.
\end{proof}

\subsection{Hodge correlators and (non-commutative) quantum master equation}\label{ssection:HC_QME}
Here, we give a generalized Hodge correlator twistor connection as a version of the non-commutative quantum master homotopy equation. In our version of the quantum master equation, instead of using a symmetric algebra of cohomology groups, we consider a graph-valued generating series. By taking tree reduction, one recovers Goncharov's Hodge correlator twistor connection canonically.

\subsubsection{Green functions for a family of compact Riemann surfaces}\label{section:Green_family}
To begin with, let us recall a family of Green functions considered by Goncharov in \cite[Section 7.1]{GoncharovHodge1}, which is necessary to define the Hodge correlator twistor connection.

Let $p: X \rightarrow B$ be a family of compact Riemann surfaces, i.e., each fiber $p^{-1}(b)=X_b$ is a compact Riemann surface and it varies smoothly on $b \in B$. Then, a family of compact Riemann surfaces with $n+1$-marked points can be regarded as a pair of a family $p: X \rightarrow B$ of compact Riemann surfaces and $n+1$ distinct sections $s_0, s_1,\ldots, s_n : B \rightarrow X$ $(p \circ s_i = \Id)$. Here, intersections of each fiber $X_b$ and the sections are understood as marked points on $X_b$. Then, $S =\{ s_0(B), \ldots, s_n(B)\}\subset X$ is a smooth divisor in $X$ over $B$. Indeed, let $k+1$ be the dimension of $X$ and $k$ be the dimension of $B$. Then, $S$ is a submanifold of $X$ with codimension 1. As a section over $B$, $S_t$ varies smoothly on $t \in B$. Set $S^{\ast} := S \setminus \{s_0(B)\}$. Take a section of tangent vector $v \in T X_{/B}$ which factors through $s_0$. In general sections $s_i$ may intersect with other sections $s_j$ but, in the present article, we restrict ourselves to families such that no section intersects with other sections.

Then, we have the (geometric) variation of $\mathbb{R}$-mixed Hodge structures over $B$ as  
\begin{align}
V^{\vee}_{X, S^{\ast}} := \gr^W R^1 p_{\ast} (X - S, \underline{\mathbb{R}}) = H^1(X_{/B}, \mathbb{R}) \oplus \mathbb{R}[S^{\ast}_{/B}](-1)	\\
V_{X, S^{\ast}} := \gr^W R^1 p_{\ast} (X - S, \underline{\mathbb{R}})^{\vee} = H_1(X_{/B}, \mathbb{R}) \oplus \mathbb{R}[S^{\ast}_{/B}](1)	
\end{align}
where
\begin{align}
	\mathbb{R}[S^{\ast}_{/B}](-1):= \mathbb{R}[S^{\ast}_{/B}] \otimes \mathbb{R}(-1) \simeq   \mathbb{R}[S^{\ast}_{/B}] \otimes H^2(X_{/B}, \mathbb{R})\\
	\mathbb{R}[S^{\ast}_{/B}](1):= \mathbb{R}[S^{\ast}_{/B}] \otimes \mathbb{R}(1) \simeq   \mathbb{R}[S^{\ast}_{/B}] \otimes H_2(X_{/B}, \mathbb{R})
\end{align}
and $\mathbb{R}(n) := (2 \pi i)^n \mathbb{R}$ denotes the $n$-th Tate twists of weight $-2n$ with Hodge bidegree  $(-n,-n)$. The elements of the cohomology group of a compact Riemann surface are of weight 1 and elements corresponding to punctures are of weight 2. By taking the tensor product with $\mathbb{C}$ over $\mathbb{R}$, their Hodge degrees are (1,0), (0,1), and (1,1). Also, note that we set
\begin{equation}
	(X - S, \underline{\mathbb{R}}):= \iota_{\ast} \underline{\mathbb{R}}_{X-S}, \quad \iota: X - S \hookrightarrow X.
\end{equation}

Let $p: X \rightarrow B$ be a proper holomorphic submersion. Then, by Ehresmann's theorem, $p$ becomes a fiber bundle (cf.\cite[Theorem C.10]{mixed_hodge}). The Gauss--Manin connection $\nabla_{GM}$ in $R^k p_{\ast} \underline{\mathbb{C}}$ is described as follows. First, note that
\begin{equation}
	H^k(X_{/B}, \mathbb{C}) \simeq  R^k p_{\ast} \underline{\mathbb{C}} \otimes_{\mathbb{C}} \mathcal{O}_B.
\end{equation}
This is the local system whose fiber is
\begin{equation}
	H^k(X_{b}, \mathbb{C}) \simeq  R^k p_{\ast} \underline{\mathbb{C}}_b.
\end{equation}
Take a section $\lambda$ of $R^k p_{\ast} \underline{\mathbb{C}}$ and a smooth family $\alpha_b$ of differential $k$-forms on the fiber $X_t$ such that it represents the cohomology class $\lambda$, i.e., 
\begin{equation}
	\lambda(b) = [\alpha_b] \in H^k(X_{b}, \mathbb{C}).
\end{equation}
Let $\beta$ be any differential form on $X$ such that $\beta|_{X_b} = \alpha_b$. Then, de Rham differential of $\beta$ along the fiber of $p$ vanishes since $d\alpha_b =0 $. Thus, 
\begin{equation}
	d \beta \in \mathcal{A}^1(B) \otimes \mathcal{A}^{k}_{cl}(X_{/B})
\end{equation}
The cohomology class of closed $k$-form of $d\beta$ along the fiber is independent of the choice of $\beta$ representing $\lambda$. The resulting differential 1-form $d\beta$ on $B$ valued in $R^k p_{\ast} \underline{\mathbb{C}}$ is exactly the $\nabla_{GM}(\lambda)$.

Let 
\begin{equation}
	\mathcal{H}:=H_1(X_{/B}, \mathbb{R}):= R^1 p_{\ast}(\mathbb{R})^{\vee}
\end{equation}
be a local system over $B$ with Gauss--Manin connecion$\nabla_{GM}$. Its fiber over $b \in B$ is $H_1(X_b, \mathbb{R})$. By definition, $H_1(X_b, \mathbb{Z})$ is flat with respect to the Gauss--Manin connection (that is, $H_1(X_b, \mathbb{R})$ consists of horizontal sections with respect to $\nabla_{GM}$).

The Albanese variety $\mathrm{Alb}(X_b)$ of $X_b$ is given as a manifold by
\begin{equation}
	\mathrm{Alb}(X_b) \simeq H_1(X_b; \mathbb{R})/ H_1(X_b; \mathbb{Z}).
\end{equation}
Note that each fiber of the local system $\mathcal{H}$ is identified with the tangent space of $\mathrm{Alb}(X_b)$.

Let $\pi : X \times_B X \rightarrow B$ be the fiber product of $p : X \rightarrow B$ over $B$ and $\pi^{\ast} \mathcal{H} \rightarrow X \times_B X$ be the pullback of $\mathcal{H} \rightarrow B$ along $\pi$. Then, there is a $\pi^{\ast} \mathcal{H}$-valued smooth  differential form over $X \times_B X$ defined by
\begin{equation}
	\nu \in \mathcal{A}^1_{X \times_B X} \otimes \pi^{\ast} \mathcal{H}.
\end{equation}
For $(x_1, x_2) \in X \times_B X$ and $(v_1,v_2) \in T_{(x_1,x_2)}(X \times_B X)$, $\nu(v_1,v_2):=\nu_{(x_1,x_2)}(v_1,v_2)$ is defined as follows: Let $\gamma: ( - \epsilon, \epsilon) \to X \times_B X$ be a smooth path so that $\gamma(0) = (x_1, x_2)$ and $\gamma'(0) = (v_1,v_2)$. Note that, if $x_1(t)$ and $x_2(t)$ on $\gamma(t) = (x_1(t), x_2(t))$ lie on the same fiber $X_{b(t)}$, one can define
\begin{equation}
	X_{b(t)} \times X_{b(t)}  \rightarrow \mathrm{Alb}(X_{b(t)}); \quad (x_1(t), x_2(t)) \to (x_1(t) - x_2(t)).
\end{equation}
Since $(x_1(t) - x_2(t))$ is a curve on $\mathrm{Alb}(X_{b(t)})$, its differential at $t=0$ defines a tangent vector $v' \in T_{(x_1-x_2)} \mathrm{Alb}(X_{/B})$. Now by taking the horizontal lift $v''$ of $v'$ by using Gauss--Manin connection $\nabla_{GM}$, we obtain a vector $v' - v'' \in T_{(x_1-x_2)} \mathrm{Alb}(X_{b(0)})$. Then, we set
\begin{equation}
	\nu(v_1,v_2) := v' - v'' \in \pi^{\ast}\mathcal{H}.
\end{equation}

Let $\mathcal{H}_{\mathbb{C}} := \mathcal{H} \otimes \mathbb{C} = \mathcal{H}_{\mathbb{C}}^{-1,0} \oplus \mathcal{H}_{\mathbb{C}}^{0,-1}$ be the Hodge decomposition with a paring $\langle- , - \rangle:\mathcal{H}_{\mathbb{C}}^{-1,0} \otimes \mathcal{H}_{\mathbb{C}}^{0,-1} \rightarrow \bC$ induced by the intersection paring $H_1(X_b; \mathbb{Z})\wedge H_1(X_b; \mathbb{Z}) \rightarrow \bZ(1)$.
By composing the wedge product and the intersection paring $\langle - , - \rangle$, we obtain the canonical $(1,1)$-form
\begin{equation}
	\langle \nu\wedge \nu \rangle \in \mathcal{A}^{1,1}_{X \times_B X}.
\end{equation}
Let $\langle \nu, \nu \rangle ^{[1,1]}$ denote its $1 \times 1$ component over $X \times_B X$. Let $\Delta \subset X \times_B X$ be the relative diagonal subset and $\delta_{\Delta}$ be the delta current supported on $\Delta$. A section $a : B \rightarrow X$ determines a divisor $S_a \subset X \times_B X$ such that each fiber at $t \in B$ is $\{a_t\} \times X \cup X \times \{a_t\}$. We set $\delta_{S_a}$ the $(1,1)$-current on $X$ supported on the divisor $S_a$.

\begin{definition}[{Green current in family of Riemann surfaces (\cite[Definition 7.1]{GoncharovHodge1}})] \label{def:Green_cur}
Let $p: X \rightarrow B$ be a proper map and $a: B \rightarrow X$ be a section of $p$. A \textit{Green current} $G_a(x,y)$ is a $0$-current on $X \times_B X$ satisfying the following differential equation:
\begin{equation}\label{eq:GC_family}
	\bar{\partial} \partial G_a(x,y) = \delta_{\Delta} - \delta_{S_a} - \langle \nu \wedge \nu \rangle^{[1\times 1]}.
\end{equation}
\end{definition}

\begin{remark}
\begin{enumerate}[(1)]
\item A Green current exists fiberwise by $\partial \bar{\partial}$-lemma.
\item  A Green current $G_a(x,y)$ is unique up to a function lifted from the base $B(\mathbb{C})$.
\item By taking a nowhere vanishing section $v \in T_a X_{/B}$, one can obtain the normalized Green function $G_v(x,y)$ by $v$. Here, $v$ is a section of $TX|_{a} \rightarrow TB \rightarrow B$. For any $b \in B$, we have $v(b) \in T_{a(b)} X_{/B}$.
\end{enumerate}
\end{remark}
In the subsequent article, taking a section $v_0$ of the fiber bundle $T_{s_0} X_{/B}$, we consider the Green function  $G_{v_0}(x,y)$ normalized by $v_0$.

We note that the differential of a Green current away from $S_a \cup \Delta$ has the form
\begin{equation}
	\bar{\partial} \partial G_a(x,y) = - \langle \nu \wedge \nu \rangle^{[1\times 1]} \in \mathcal{A}^{1,1}_{X \times_B X}.
\end{equation}

\subsubsection{Goncharov's twistor propagator}
Recall the notion of twistor plane and twistor line. Let $\bC^2$ with coordinate $(z,w)$ and consider a distinguished subspace denoted by $\bC_{\twi} := \{(z,w) \in \bC^2 \mid z+w=1 \in \bC\} \subset \bC^2$. We call $\bC^2$ \textit{twistor plane} and $\bC_{\twi}$ the \textit{twistor line} respectively. We often take a coordinate of $\bC_{\twi}$ as $(1-u, u)$ by using $u \in \bC$. The twistor plane $\bC^2$ is equipped with an antiholomorphic involution map $\sigma: \bC^2 \rightarrow \bC^2$ defined by $\sigma((z,w)) = (\bar{w}, \bar{z})$ and the complex conjugation $c: \bC^2 \rightarrow \bC^2$. In this terminology, the twistor line $\bC_{\twi}$ is identified with the invariant subspace $(\bC^2)^{\sigma}$.  Then, the Goncharov's twistor version of $d^{\mathbb{C}} G_{a}$ is defined as follows.
\begin{definition}[Goncharov's twistor propagator]
Let $p: X \rightarrow B$ be a proper map,  $a: B \rightarrow X$ be a section of $p$, and  $G_a(x,y)$ be the Green current on $X \times_B X$. Consider $p\times \mathrm{id}: X \times \bC^2 \rightarrow B \times \bC^2$ and pullback of $G_a(x,y)$ to  $(X \times_B X) \times \bC^2$ denoted by the same $G_a(x,y)$. Let $\mathsf{D}^{\mathbb{C}}$ be the twistor version of $d^{\mathbb{C}}$ defined by
    \begin{equation}
        \mathsf{D}^{\mathbb{C}} = (w \partial - z \bar{\partial}) + (z dw - w dz).
    \end{equation}
    for $(z,w) \in \bC^2$. We define the current 
    \begin{equation}
        \mathsf{D}^{\mathbb{C}} G_a(x,y):= ((w \partial - z \bar{\partial}) + (z dw - w dz))G_a(x,y) \in \mathcal{D}^{\ast, \ast}(X\times_{B} X) \otimes \Omega_{\bC^2}^{\bullet}
    \end{equation}
    and call it the \textit{Goncharov's twistor propagator}. Here, $\Omega_{\bC^2}^{\bullet}$ denotes the space of algebraic differential forms on $\bC^2$.
\end{definition}
Note that $\mathsf{D}^{\mathbb{C}}$ restricts to
\begin{equation}
	\mathsf{D}^{\mathbb{C}}|_{ \mathbb{C}_{\twi}} = u \partial - (1-u) \bar{\partial} + du
\end{equation}
on the twistor line $(1-u, u) \in \mathbb{C}_{\twi} \subset \mathbb{C}^2$ so that, we have, on the twistor line,
    \begin{equation}
(\mathbf{d} + d_{\mathbb{C}_{\twi}}) \mathsf{D}^{\mathbb{C}} G_a =\bar{\partial} \partial G_a(x,y) = \delta_{\Delta} - \delta_{S_a} - \langle \nu \wedge \nu \rangle^{[1\times 1]}.
\end{equation}
The twistor propagator $\mathsf{D}^{\mathbb{C}} G_a$ is real, that is,
\begin{equation}\label{eq:reality_prop}
    (c \circ \sigma)^{\ast} \mathsf{D}^{\mathbb{C}} G_a = \mathsf{D}^{\mathbb{C}} G_a
\end{equation}
since $(c \circ \sigma)^{\ast} \mathsf{D}^{\mathbb{C}} = - \mathsf{D}^{\mathbb{C}}$ and $(c \circ \sigma)^{\ast} G_a= - G_a$.

\begin{remark}\label{rem:singularity}
Let $z$ be a local holomorphic coordinate near some point $p \in X_{/B}$ corresponding to $z=0$. Let us briefly describe the singularity of $\mathsf{D}^{\mathbb{C}} G_a(z,0)$ at $z=0$ on the twistor line $\bC_{\twi}$. Considering a polar coordinate $z=re^{i \theta}$, we have up to smooth part
\begin{equation}
\begin{split}
    \mathsf{D}^{\mathbb{C}} G_a(z,0) & \sim (u \partial - (1-u) \bar{\partial} + du)\frac{1}{2\pi \sqrt{-1}}\log|z|\\
    &= \frac{1}{2\pi \sqrt{-1}} \left( u \frac{dz}{z} -(1- u) \frac{d\bar{z}}{\bar{z}} + (\log r) du\right)\\
    &= \frac{1}{2 \pi \sqrt{-1}} (-1+2u) \frac{dr}{r} + \frac{1}{2\pi } d \theta + \frac{\log r}{2\pi \sqrt{-1}} du.
\end{split}
\end{equation}
This computation implies that $\mathsf{D}^{\mathbb{C}} G_a(z,0)$ has singularities at $r=0$ one of which is $\frac{dr}{r}$ and the other is $\log(r)$. As we will see below, these singularities are canceled when defining integrations using $\mathsf{D}^{\mathbb{C}} G(z,w)$ (cf. Lemma \ref{lem:convergence}, see also \cite{GoncharovHodge1} for planar tree case).
\end{remark}

Recall that the group $\bC^{\times} \times \bC^{\times}$ acts on the twistor plane $\bC^2$ by 
\begin{equation}
    (\bC^{\times} \times \bC^{\times}) \times \bC^2 \rightarrow \bC^2;\quad ((\lambda, \mu), (z,w)) \mapsto (\lambda z, \mu w)
\end{equation}
and on $\mathcal{D}^{\ast, \ast}(X\times_{B} X)$ by
\begin{equation}
    (\bC^{\times} \times \bC^{\times}) \times \mathcal{D}^{s, t}(X\times_{B} X)\rightarrow  \mathcal{D}^{s, t}(X\times_{B} X);\quad ((\lambda, \mu), \alpha_{s,t}) \mapsto (\lambda^s \mu^t  \alpha_{s,t}).
\end{equation}
Thus, the character of the action of the group $\bC^{\times} \times \bC^{\times}$ can define the Hodge bidegree on them. Notice that by this character one sees that $\mathsf{D}^{\mathbb{C}}$ and hence $\mathsf{D}^{\mathbb{C}} G_a$ are of degree $(1,1)$.

\subsubsection{Genralized Hodge correlator twistor connection (Effective action)} \label{subsection:gen_hdoge_corr}
In this subsection, we shall define a linear map 
\begin{equation}\label{eq:5.3.27}
    Z_{-}: \mathcal{CD}_{H,S}^{\vee, \bullet} \rightarrow \mathcal{D}^{\ast, \ast}(B) \otimes \Omega^{\bullet}_{\mathbb{C}^2}.
\end{equation}
by extending the definition in \cite{GoncharovHodge1}. We define it on the twistor plane $\bC^2$ and show its well-definedness for graphs of defect degree $\leq 1$. Then, we prove its restriction to the twistor line $Z_{-}':=Z_{-}|_{B \times \bC_{\twi}}$ commutes with differentials up to sign for defect 0 part.

We set 
\begin{equation}
    V_{X, S^{\ast}}^{\vee}{}_{/B} = H^1(X_{/B}, \mathbb{R}) \oplus \mathbb{R}[S^{\ast}_{/B}](-1)
\end{equation}
and denote by $G_{s_0}(x,y)$ the normalized Green current associated with the section $s_0: B \rightarrow X$ and a section of tangent vectors $v \in TX_{/B}$ factoring through $s_0$ (cf. Definition \ref{def:Green_cur}). To emphasize the distinguished section $s_0$, we suppress the chosen section of tangent vectors $v$ in this notation.

The definition depends on how many edges a decorated connected graph $\Gamma$ has as follows. 

\noindent
Case (I): case that $|E_{\Gamma}|=1$. In this case, $\Gamma$ is the unique tree graph consisting of exactly one edge $E$ with two external vertices. Now, since there is no internal vertex, $Z_{\Gamma}$ is defined without integration. There are following three ways of decoration.

\noindent
Case (I-1): $\Gamma = T_1$ is decorated by two distinct sections $s,t \in S^{\ast}$. These sections define a map $(s,t): B \rightarrow X \times_B X$. Then, we set
\begin{equation}
	Z_{T_1} := \mathsf{D}^{\mathbb{C}} G_{s_0} (s,t):= (s,t)^{\ast} \mathsf{D}^{\mathbb{C}} G_{s_0} \in \mathcal{D}_B^{\bullet} \otimes \Omega_{\mathbb{C}^2}^{\bullet}.
\end{equation}
When $\Gamma = T_1$ is decorated by the same  sections $s \in S^{\ast}$, we set $Z_{T_1} =0$.

\noindent
Case (I-2): $\Gamma = T_2$ is a tree with exactly one edge decorated by a section $s : B \rightarrow X$ and $1$-form  $\alpha \in \mathcal{A}^1_X$ on $X$. In this case, we set
\begin{equation}
	Z_{T_2}:= s^{\ast}(\alpha) \in \mathcal{A}^1_B.
\end{equation}

\begin{remark}
When $B =\{\ast\}$, there is no need to consider the case (I-2) since we consider the pullback of a 1-form to 0-dimensional space $B =\{\ast\}$.
\end{remark}

\noindent
Case (I-3): $\Gamma = T_3$ has exactly one edge with decoration by two 1-forms $\alpha, \beta$. In this case, we set $Z_{T_3}=0$.

\noindent
Case (II): case that $|E_{\Gamma}| \geq 2$. Assume that each external vertex of a connected graph $\Gamma$ with $|E_{\Gamma}| \geq 2$ is decorated by either a generator $\{s_i\} \in \mathbb{C}[S^{\ast}_{/B}]$ or a section $\omega \in H^1(X_{/B}, \mathbb{R})$. We choose a 1-form $\alpha$ on $X$ which represents $\omega$ so that $\alpha$ is holomorphic/anti-holomorphic on each fiber $X_t$.  Recall that, then, the set $E_{\Gamma}^{\mathrm{ext}}$ of external edges of $\Gamma$ is decomposed into the set of $S^{\ast}$-decorated edges whose external vertices are $S^{\ast}$-decorated and that of special edges whose external vertices are decorated by $H^1(X_{/B})$.

When $\Gamma=(\Gamma, W; \mathrm{Or}_{\Gamma})$ is a connected decorated graph with self-loops, we set $Z_{\Gamma} = 0$. We define the weight $\omega_{\Gamma}$ associated with a connected decorated graph $\Gamma$ without self-loops as follows: Assume that $S^{\ast}$-decoration of $\Gamma$ is given by $s_1, \ldots, s_k$. Then, recalling the notation in Remark \ref{rem:3.3.2}, according to whether $E$ is in $E_{\Gamma}^{\mathrm{int}}$,  $S^{\ast}$-decorated edge, or spacial edge, we define $\omega_{E} \in \Omega^{\bullet}(\Conf_{V^{\mathrm{int}}_{\Gamma}}(X_{/B}; (s_1,\ldots, s_k))_{s_0})\otimes \Omega_{\bC^2}^{\bullet}$ as follows:
\begin{itemize}
	\item If $E = (v_i,v_j)$ is an internal edge, then we set $\omega_{E} := p^{\ast}_{E}(\mathsf{D}^{\mathbb{C}} G_{s_0})$ where the projection $p_{E} : \Conf_{V^{\mathrm{int}}_{\Gamma}}(X_{/B}; (s_1,\ldots, s_k))_{s_0} \times \bC^2 \rightarrow \Conf_{2}(X_{/B})_{s_0} \times \bC^2$ is defined by $p_{E_i}((x_1,\ldots, x_{|V_{\Gamma}^{\mathrm{int}} \cup V_{\Gamma}^{\mathrm{ext}, S^{\ast}}|}),(z,w))=((x_i, x_j), (z,w))$.
	\item If $E = (v_i, v_j)$ is an external edge with an external vertex $v_j$ decorated by a section $s \in S^{\ast}$, we set $\omega_{E}:= p^{\ast}_{E} (\mathsf{D}^{\mathbb{C}} G_{s_0})$ where $p_{E}: \Conf_{V^{\mathrm{int}}_{\Gamma}}(X_{/B}; (s_1,\ldots, s_k))_{s_0} \times \bC^2 \rightarrow \Conf_{2}(X_{/B})_{s_0} \times \bC^2$ is defined by $p_{E}((x_1,\ldots, x_{|V_{\Gamma}^{\mathrm{int}} \cup V_{\Gamma}^{\mathrm{ext}, S^{\ast}}|}),(z,w))=((x_i, x_j),(z,w))$.
	\item If $E = (v_i, v')$ is an external edge with $v'$ decorated by a section $\omega \in H^1_{dR}(X_{/B})$, we set $\omega_{E}:=  p_{v_i}^{\ast} \alpha$ where $p_{v_i}: \Conf_{V^{\mathrm{int}}_{\Gamma}}(X_{/B};(s_1, \ldots, s_k))_{s_0} \times \bC^2 \rightarrow X$ given as $p_{v_i}((x_1,\ldots, x_{|V_{\Gamma}^{\mathrm{int}} \cup V_{\Gamma}^{\mathrm{ext}, S^{\ast}}|}),(z,w)) = x_i \in X$ and $\alpha \in \mathcal{A}^1_X$ representing $\omega$.
\end{itemize}
Choose an element
\begin{equation}
	(E_0 \wedge \cdots \wedge E_r) \wedge (E_{r+1} \wedge \cdots \wedge E_{|E_{\Gamma}|}) \in \mathrm{or}_{\Gamma}
\end{equation}
of the orientation torsor of $\Gamma$ so that the edges $E_0, \ldots, E_r$ are internal or $S$-decorated, and the others are decorated by 1-forms. Then, we define $\omega_{\Gamma} \in \Omega^{\bullet}(\Conf_{V^{\mathrm{int}}_{\Gamma}}(X_{/B}; (s_1,\ldots, s_k))_{s_0}) \otimes \Omega_{\mathbb{C}^2}^{\bullet}$ by
\begin{equation}\label{eq:def_omega_gamma}
	\omega_{\Gamma}:= \sgn(E_0 \wedge \cdots  \wedge E_{|E_{\Gamma}|}) \cdot \omega_{E_0} \wedge \cdots \wedge \omega_{E_r} \wedge \omega_{E_{r+1}} \wedge \cdots \wedge \omega_{E_{|E_{\Gamma}|}} 
 \end{equation}
where $\sgn(E_0 \wedge \cdots  \wedge E_{|E_{\Gamma}|})$ is the difference of the orientation torsors $(E_0 \wedge \cdots \wedge E_r) \wedge (E_{r+1} \wedge \cdots \wedge E_{|E_{\Gamma}|})$ and $\mathrm{Or}_{\Gamma}$. Note that the sign of $\omega_{\Gamma}$ is determined only by $\mathrm{Or}_{\Gamma}$.

Then, the element  $Z_{\Gamma}$ associated with $\Gamma$ is defined as the fiber integration
	\begin{align}
		Z_{\Gamma, (b,(z,w))} &:= (p_{\Gamma}{}_{\ast} \omega_{\Gamma})_{(b, (z,w))} = \int_{p^{-1}_{\Gamma}((b, (z,w))}\omega_{\Gamma}
	\end{align}
	along $p_{\Gamma}:C_{V_{\Gamma}^{\mathrm{int}}}(X_{/B};(s_1,\ldots, s_k))_{(s_0,v)} \times \mathbb{C}^2 \rightarrow B \times \mathbb{C}^2$ where $b \in B$ and $(z,w) \in \mathbb{C}^2$, and $s_1, \ldots, s_k$ are sections appered in the $S^{\ast}$-decorations of $\Gamma$.

For any decorated connected graph with $|E_{\Gamma}|\geq 2$, it is not obvious whether $Z_{\Gamma}$ is well-defined or not. In the following, we will show that $Z_{\Gamma}$ is well-defined for any decorated connected trivalent graph or decorated connected graph of defect one. For this, we first prepare the following lemma.

\begin{lemma}\label{lem:infinite_vanishing_lemma}
    Let $\Gamma$ be a decorated connected graph with $|E_{\Gamma}|\geq 2$. Then, the associated differential form $\omega_{\Gamma} \in \Omega^{\bullet}(\Conf_{V^{\mathrm{int}}_{\Gamma}}(X_{/B}; (s_1,\ldots, s_k))) \otimes \Omega_{\bC^2}^{\bullet}$ vanishes when one of the internal vertices goes to the point $s_0$.
\end{lemma}

\begin{proof}
    To see the cancellation of singularity at $s$, we apply similar arguments as in \cite[Theorem 2.4]{Gon05}. Here, we give a proof by direct computation. Note that it suffices to consider the case that one of the internal vertices collapses at $s_0$. First, consider the case that $E$ is an internal edge connecting a vertex collapsing at $s_0$. Then, we will obtain the diagrams as in Figure \ref{fig:5.3.1}.

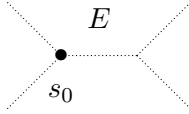
\begin{figure}[h]
 \centering
 \captionsetup{margin=2cm}
\begin{tikzpicture}
\begin{scope}
\node at (0.5,0.5) {$E$};
\draw[densely dotted] (0,0) -- (1,0);
\draw[densely dotted] (1,0) -- (1+ 1/1.41, 1/1.41);
\draw[densely dotted] (1,0) -- (1+ 1/1.41, -1/1.41);
\draw[densely dotted] (0,0) -- (- 1/1.41, 1/1.41);
\draw[densely dotted] (0,0) -- (- 1/1.41, -1/1.41);
\node at (0,0) {$\bullet$};
\node at (0,-0.5) {$s_0$};
\end{scope}
\end{tikzpicture}
\caption{A situation where a vertex connected by the internal edge $E$ collapses at $s$.}
\label{fig:5.3.1}
\end{figure}

In our situation, replacing $s_0$ with some section $s$, it suffices to show that the limit of the integrand vanishes as $s \to s_0$. To begin with, let us consider the case that there is an external edge decorated by 1-form $\alpha$ as Figure \ref{fig:5.3.2}.
\begin{figure}[h]
 \centering
 \captionsetup{margin=2cm}
\begin{tikzpicture}
\begin{scope}[xshift=5cm]
\node at (0.5,0.5) {$E$};
\draw[densely dotted] (0,0) -- (1,0);
\draw[densely dotted] (1,0) -- (1+ 1/1.41, 1/1.41);
\draw[densely dotted] (1,0) -- (1+ 1/1.41, -1/1.41);
\draw[densely dotted] (0,0) -- (- 1/1.41, 1/1.41);
\draw[densely dotted] (0,0) -- (- 1/1.41, -1/1.41);
\node at (0,0) {$\bullet$};
\node at (0,-0.5) {$s_0$};
\node at (-1, 1) {$\alpha$};
\node at(- 1/1.41, 1/1.41) {$\bullet$};
\end{scope}
\end{tikzpicture}
\caption{A situation where a vertex collapsing at $s_0$ is connected by an external edge decorated by 1-form $\alpha$}
\label{fig:5.3.2}
\end{figure}
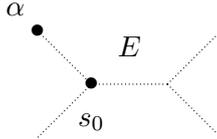
In this case, since $\alpha$ can be chosen so that $\alpha$ has compact support on the complement of a small neighbourhood of $s$, the restriction $\alpha|_{s}$ vanishes. Thus, the integrand becomes zero. Next, we consider the case that there is no external edge decorated by 1-form. Then, the integrand associated with the graph can be written as
\begin{equation}
	 \omega_{\Gamma} = \pm \mathsf{D}^{\mathbb{C}} G_{s_0}(s, x) \wedge \mathsf{D}^{\mathbb{C}} G_{s_0}(s, y) \wedge \bigwedge_{e \in E_{\Gamma} \setminus \{E, E'\}} \omega_{e}.
\end{equation}
where we set $s=x_{v}, x=x_{w_1}, y=x_{w_2}$ for $E=(v,w_1)$ and $E'=(v,w_2)$. Then, using the relation of the normalized Green function with Arakelov Green functions,
\begin{equation}
	G_{s_0}(x,y) = G_{\mathrm{Ar}}(x,y) - G_{\mathrm{Ar}}(x,s_0) - G_{\mathrm{Ar}}(s_0,y) +C
\end{equation}
we have
\begin{equation}
\begin{split}
	G_{s_0}(s,x) &= G_{\mathrm{Ar}}(s,x) - G_{\mathrm{Ar}}(s,s_0) - G_{\mathrm{Ar}}(s_0,x) +C\\
	& = -  G_{\mathrm{Ar}}(s,s_0) + g(s,s_0, x) + C,
 \end{split}
\end{equation}
and similarly, 
\begin{equation}
\begin{split}
	G_{s_0}(s,y)  = -  G_{\mathrm{Ar}}(s,s_0) + f(s,s_0, y) + C.
 \end{split}
\end{equation}
Here, we set $f(s,s_0, y)= G_{\mathrm{Ar}}(s,y)  - G_{\mathrm{Ar}}(s_0,y)$ and $g(s,s_0,x)=  G_{\mathrm{Ar}}(s,x)  - G_{\mathrm{Ar}}(s_0,x)$. Then, by definition, we have $\lim_{s \to s_0} f(s,s_0, y) = 0$ and $\lim_{s \to s_0} g(s,s_0, y) = 0$. Since the Arakelov Green function $G_{\mathrm{Ar}}(x,y)$ is smooth outside the diagonal, for $x, y \neq s_0$, we also have $\lim_{s \to s_0}\partial f(s,s_0, y) = 0$, $\lim_{s \to s_0}\partial g(s,s_0, x) = 0$ and the same formula with $\partial$ replaced by $\bar{\partial}$. Then, we have
\begin{equation}
\begin{split}
&\mathsf{D}^{\mathbb{C}} G_{s_0}(s, x) \wedge \mathsf{D}^{\mathbb{C}} G_{s_0}(s, y)\\
=& \mathsf{D}^{\mathbb{C}} (-  G_{\mathrm{Ar}}(s,s_0) + g(s,s_0, x) + C) \wedge \mathsf{D}^{\mathbb{C}} (-  G_{\mathrm{Ar}}(s,s_0) + f(s,s_0, y) + C)\\
=& \mathsf{D}^{\mathbb{C}} G_{\mathrm{Ar}}(s,s_0) \wedge  \mathsf{D}^{\mathbb{C}} G_{\mathrm{Ar}}(s,s_0) - \mathsf{D}^{\mathbb{C}}G_{\mathrm{Ar}}(s,s_0) \wedge \mathsf{D}^{\mathbb{C}} (f(s_0, s, y) +C) \\
&- \mathsf{D}^{\mathbb{C}} (g(s_0, s, x) +C) \wedge \mathsf{D}^{\mathbb{C}} G_{\mathrm{Ar}}(s,s_0) + \wedge \mathsf{D}^{\mathbb{C}} (g(s_0, s, x) +C) \wedge  \mathsf{D}^{\mathbb{C}} (f(s_0, s, y) +C)\\
=& - \mathsf{D}^{\mathbb{C}}G_{\mathrm{Ar}}(s,s_0) \wedge \mathsf{D}^{\mathbb{C}} (f(s_0, s, y) +C) \\
&- \mathsf{D}^{\mathbb{C}} (g(s_0, s, x) +C)\wedge \mathsf{D}^{\mathbb{C}} G_{\mathrm{Ar}}(s,s_0) +  \mathsf{D}^{\mathbb{C}} (g(s_0, s, x) +C) \wedge  \mathsf{D}^{\mathbb{C}} (f(s_0, s, y) +C)\\
=& - \mathsf{D}^{\mathbb{C}}G_{\mathrm{Ar}}(s,s_0) \wedge \mathsf{D}^{\mathbb{C}} f(s_0, s, y) - \mathsf{D}^{\mathbb{C}} g(s_0, s, x)  \wedge \mathsf{D}^{\mathbb{C}} G_{\mathrm{Ar}}(s,s_0)\\
& + \mathsf{D}^{\mathbb{C}} g(s_0, s, x) \wedge \mathsf{D}^{\mathbb{C}} f(s_0, s, y) \\
&+  \mathsf{D}^{\mathbb{C}} g(s_0, s, x) \wedge \mathsf{D}^{\mathbb{C}} C + \mathsf{D}^{\mathbb{C}} C \wedge \mathsf{D}^{\mathbb{C}} f(s_0, s, y)
\end{split}
\end{equation}
where we use $\omega^2 =0$ for 1-from $\omega$. From the above equation, note that every term is multiplied by $f(s,s_0, y)$, $g(s,s_0, x)$, or their differentials. Since these vanish in the order $O(s)$, one concludes that $\mathsf{D}^{\mathbb{C}} G_{s_0}(s, x) \wedge \mathsf{D}^{\mathbb{C}} G_{s_0}(s, y) \to 0$ as $s \to s_0$. Thus, we obtain, as $s \to s_0$,
\begin{equation}
	 \omega_{\Gamma} = \pm \mathsf{D}^{\mathbb{C}} G_{s_0}(s, x) \wedge \mathsf{D}^{\mathbb{C}} G_{s_0}(s, y) \wedge \bigwedge_{e \in E_{\Gamma} \setminus \{E, E'\}} \omega_{e}\to  0.
\end{equation}

Note that the above arguments can be similarly applied for the case that several internal vertices approach $s_0$. Second, consider the case that $E$ is an $S^{\ast}$-decorated external edge. In this case, we can apply the same argument as in the first case. Therefore, the assertion has been proved.
\end{proof}

\begin{lemma} \label{lem:5.3.7}
Let $\Gamma$ be a decorated connected trivalent graph with at least $2$ loops. Then,  $Z_{\Gamma}=0$ holds.
\end{lemma}
\begin{proof}
Let us consider a uni-trivalent graph with $k$-internal vertices with $l$-loops. Note that we have
\begin{equation}
    |V^{\mathrm{ext}}_{\Gamma}| + |V^{\mathrm{int}}_{\Gamma}| - |E_{\Gamma}| = 1 - l
\end{equation}
and 
\begin{equation}
    2 |E_{\Gamma}| = |V^{\mathrm{ext}}_{\Gamma}| + 3  |V^{\mathrm{int}}_{\Gamma}|,
\end{equation}
and hence
\begin{equation}
     |V^{\mathrm{ext}}_{\Gamma}| -|V^{\mathrm{int}}_{\Gamma}|  = 2 - 2l.
\end{equation}
Thus, the difference between the degree of the integrand and the dimension of the fiber is $|E_{\Gamma}| - 2|V^{\mathrm{int}}_{\Gamma}| = |V^{\mathrm{ext}}_{\Gamma}| - |V^{\mathrm{int}}_{\Gamma}| -1 + l= 1-l$. Therefore, when $l \geq 2$, the difference is strictly less than 0, implying the resulting fiber integration is zero since we integrate a differential form with a degree strictly less than the fiber dimension. Thus, the assertion has been proved. 
\end{proof}

\begin{lemma}\label{lem:convergence}
\begin{enumerate}[(1)]
\item Let $\Gamma$ be a decorated connected graph with $\mathrm{def}(\Gamma)=0$, i.e., a decorated connected uni-trivalent graph. Suppose that $|E_{\Gamma}| \geq 2$. Then, the element $Z_{\Gamma}$ is well-defined as the element of $\mathcal{D}_B^{\bullet, \bullet} \otimes \Omega_{\mathbb{C}^2}^{\leq 2}$.
\item Let $\Gamma$ be a decorated connected graph with $\mathrm{def}(\Gamma)=1$, i.e., a decorated connected graph such that except for one $4$-valent vertex all other vertices are uni-trivalent. Then, the element $Z_{\Gamma}$ is well-defined as the element of $\mathcal{D}_B^{\bullet, \bullet} \otimes \Omega_{\mathbb{C}^2}^{\leq 2}$.
\end{enumerate}
\end{lemma}
\begin{proof}
(1) First note that the integrand $\omega_{\Gamma}$ has singularity along diagonals corresponding to internal edges and also at the point $s_0$. When some internal vertices go to the point $s_0$, $\omega_{\Gamma}$ vanishes by Lemma \ref{lem:infinite_vanishing_lemma}. Thus, we need only take care of the singularities along the diagonals given by the internal edges of $\Gamma$. 

By Lemma \ref{lem:5.3.7}, we can assume that $\Gamma$ is a uni-trivalent graph with loop number $l(\Gamma) \leq 1$. Consider the case that $l(\Gamma) = 1$. Let $n$ be the number of external edges of $\Gamma$, then the number of edges is $2n$. We can replace $\mathsf{D}^{\mathbb{C}} G_a(x,y)$ with $(w \partial - z \bar{\partial}) G_a(x,y)$ since, if at least one of the internal or $S^{\ast}$-decorated edges are assigned with $(z dw - w dz)G_a(x,y)$, then the associated integral vanishes by degree reason as Lemma \ref{lem:5.3.7}. Suppose that $\Gamma$ has $p$ and $q$ special decorated edges decorated by holomorphic and anti-holomorphic 1-forms respectively. Let $E$ be an internal edge adjacent to two special decorated edges $E_1$ and $E_2$. If both of $E_1$ and $E_2$ are decorated by holomorphic 1-forms $\alpha_{E_1}, \beta_{E_2}$, then we have 
\begin{equation}
\begin{split}
    &(w \partial - z \bar{\partial}) G_E \wedge \alpha_{E_1} \wedge \beta_{E_2} \\
    =& - z \bar{\partial} G_E\wedge \alpha_{E_1} \wedge \beta_{E_2}\\
    =& z d^{\bC} G_E\wedge \alpha_{E_1} \wedge \beta_{E_2}.
    \end{split}
\end{equation}
Thus, we can think that $E$ is assigned with $z d^{\bC} G_E$. Similarly, if both of $E_1$ and $E_2$ are decorated by anti-holomorphic 1-forms, then $E$ is assigned with $w d^{\bC} G_a(x,y)$. Otherwise, $E$ is assigned with $w \partial G_a(x,y)$ or $-z \bar{\partial} G_a(x,y)$. Since the same type of forms cannot decorate all of the three edges adjacent to an internal vertex and $dz_v \wedge d\bar{z}_v$ must be assigned for every internal vertex $v$, the top degree term along the fiber direction of integrand associated with $\Gamma$ can be written as the form
\begin{equation}\label{eq:omega_g_loop}
    \omega_{\Gamma} =\pm w^{n-p} z^{n-q} d^{\bC} G_{E_1} \wedge \cdots \wedge d^{\bC} G_{E_{2n-p-q}} \wedge \omega_{E_{2n-p-q +1}} \wedge \cdots \wedge \omega_{E_{2n}}.
\end{equation}
By Corollary \ref{cor:extended_prop}, $d^{\bC} G_E$ can be extended to the compactified configuration spaces fiberwise, the fiber integration of $\omega_{\Gamma}$ is well-defined through usual configuration space integral methods.

Next, Consider the case that $l(\Gamma) = 0$. This case is shown essentially by Goncharov in \cite{GoncharovHodge1}, but here we explain it from the viewpoint of configuration space integrals. By a similar argument as above, the top degree term along the fiber direction of integrand associated with $\Gamma$ is given as 
\begin{equation}\label{eq:5.3.8-1}
    \mathsf{D}^{\mathbb{C}} \left(\pm w^{n-1-p} z^{n-1-q} G_{E_1} \wedge d^{\bC} G_{E_2}  \wedge \cdots \wedge d^{\bC} G_{E_{2n-3-p-q}} \wedge \omega_{E_{2n-3-p-q +1}} \wedge \cdots \wedge \omega_{E_{2n-3}}\right).
\end{equation}
As elements of currents $\mathcal{D}_B^{\bullet, \bullet}$, the differentials on $B$ direction commute with integration so that the only thing we need to show is the well-definedness of the form
\begin{equation}
    G_{E_1} \wedge d^{\bC} G_{E_2} \wedge \cdots \wedge d^{\bC} G_{E_{2n-3-p-q}} \wedge \omega_{E_{2n-3-p-q +1}} \wedge \cdots \wedge \omega_{E_{2n-3}}
\end{equation}
and, except for $G_{E_1}$ which have singularities of logarithm functions, other factors can be regarded as smooth forms on compactified configuration spaces. However, as is well known the logarithmic function is integrable. Thus, $Z_{\Gamma}$ is well-defined.

\noindent
(2) When $l(\Gamma) = 2$ or $l(\Gamma) = 1$, the well-definedness of $Z_{-}$ follows immediately by the same argument as (1) of 1 loop and tree case respectively. Suppose $l(\Gamma) = 0$. By similar reasoning, the well-definedness for this case also follows from the integrabilities of $G_{s_0}$ and $d^{\bC} G_{s_0}$.
\end{proof}

Note that, in the above proof, the expressions of $\omega_{\Gamma}$ in terms of $d^{\mathbb{C}} G_{s_0}$ corresponds to the cancellation of singularity of $\frac{dr}{r}$ in $\mathsf{D}^{\bC} G_{s_0}$ described in Remark \ref{rem:singularity}.

\begin{remark}
    At least for the tree graph case,  an integral of currents associated with a graph makes sense since it can be understood as an iteration of integral operators. However, in the case of diagrams possibly with loops, the associated integrals do not make sense in general. Thus, we need to take care of the singularity of integrands or use the theory of configuration space integral if it is applicable as above. In our case, for decorated connected trivalent graph $\Gamma$, $\omega_{\Gamma}$ defines a $|E_{\Gamma}|$-current on $C_{|V_{\Gamma}^{\mathrm{int}}|}(X_{/B};(s_1,\ldots, s_k))_{(s_0,v)} \times \mathbb{C}^2$ and algebraic on $\bC^2$, i.e., $\omega_{\Gamma} \in \mathcal{D}^{2|V_{\Gamma}^{\mathrm{int}}|}(C_n(X_{/B};(s_1,\ldots, s_k))_{(s_0,v)})\otimes \Omega^{\leq 2}_{\bC^2}$
\end{remark}

\begin{remark}
    For a decorated connected graph $\Gamma$ with $|E_{\Gamma}|\geq 2$, the integration $Z_{\Gamma, (b,(z,w))}$ is same as integration of $\omega_{\Gamma}$ over codimension-0 open strata of the fiber $p^{-1}_{\Gamma}((b, (z,w))$.
\end{remark}

\begin{lemma} \label{lem:5.2.2} The map $Z_{-}$ satisfies the following conditions.
\begin{enumerate}[(i)]
\item  $Z_{(\Gamma, W; \mathrm{Or}_{\Gamma})}=-Z_{(\Gamma, W; -\mathrm{Or}_{\Gamma})}$.
    \item if $\Gamma$ has a self-loop as Figure \ref{fig:5.3.3} (A) , then $Z_{\Gamma} =0$,
    \item if $\Gamma$ is the unique uni-trivalent connected graph without internal vertex and its two external vertices are labeled by the same element $\gamma = s \in S^{\ast}$ or $\gamma = \alpha \in \Omega^1(X)$ as Figure \ref{fig:5.3.3} (B),  then $Z_{\Gamma} =0$,
    \item  if $\Gamma$ has an internal vertex $v$ which is connected to two external vertices labeled by the same $\gamma = s \in S^{\ast}$ or $\gamma = \alpha \in H^1(X_{/B})$, as Figure \ref{fig:5.3.3} (C), then $Z_{\Gamma} =0$
     \item $\Gamma$ has a non-regular edge, as Figure \ref{fig:5.3.3} (D), then $Z_{\Gamma} =0$.
\end{enumerate}

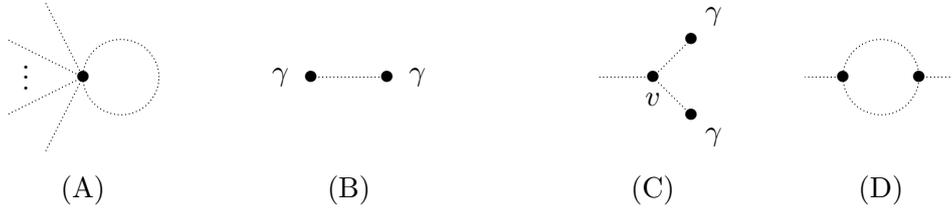
\begin{figure}[h]
\captionsetup{margin=2cm}
 \centering
 \begin{tikzpicture}
 \node at (-0.5, -1.5) {(A)};
  \node at (3, -1.5) {(B)};
  \node at (7, -1.5) {(C)};
  \node at (10, -1.5) {(D)};
 \begin{scope}
 \draw[densely dotted] (0,0) circle (0.5cm);
    \draw[densely dotted] (-0.5,0.0)--(-1.5,0.5);
     \draw[densely dotted] (-0.5,0.0)--(-1,1);
     \draw[densely dotted] (-0.5,0.0)--(-1.5,-0.5);
     \draw[densely dotted] (-0.5,0.0)--(-1,-1);
     \node at (-1.25, 0.1) {$\vdots$};
     \node at (-0.5, 0) {$\bullet$};
 \end{scope}
 \begin{scope}[xshift=3cm]
 \draw[densely dotted] (-0.5,0) -- (0.5,0);
 \node at (0.5, 0) {$\bullet$};
 \node at (0.9, 0) {$\gamma$};
 \node at (-0.5,0) {$\bullet$};
 \node at (-0.9, 0) {$\gamma$};
 \end{scope}
 \begin{scope}[xshift=6cm]
 \draw[densely dotted] (1-1/1.414,0) -- (1,0);
 	\draw[densely dotted] (1,0) -- (1.5, 0.5);
 	\draw[densely dotted] (1,0) -- (1.5, -0.5);
  \node at (1.5, 0.5) {$\bullet$};
   \node at (1.5, -0.5) {$\bullet$};
   \node at (1,0) {$\bullet$};
   \node at (1,-0.3) {$v$};
  \node at (1.5+0.3, 0.5+0.3) {$\gamma$};
  \node at (1.5+0.3, -0.5-0.3) {$\gamma$};
 \end{scope}
 \begin{scope}[xshift=10cm]
  \draw[densely dotted] (0,0) circle (0.5cm);
   \draw[densely dotted] (0.5,0.0)--(1.0,0.0);
    \draw[densely dotted] (-0.5,0.0)--(-1.0,0.0);
    \node at (0.5, 0) {$\bullet$};
     \node at (-0.5, 0) {$\bullet$};
 \end{scope}
 \end{tikzpicture}
 \caption[Several typical (sub)graphs which vanish under the map $Z_{-}$]{Several typical (sub)graphs which vanish under the map $Z_{-}$. Here, we redisplay the same graphs as  Figure \ref{fig:4.3.1} for the convenience of the reader.}
 \label{fig:5.3.3}
\end{figure}
\end{lemma}
\begin{proof}
(i) It follows from the well-definedness of $\omega_{(\Gamma, W; \mathrm{Or}_{\Gamma})}$, i.e., the sign of $\omega_{(\Gamma, W; \mathrm{Or}_{\Gamma})}$ depends only on $\mathrm{Or}_{\Gamma}$. The conditions (ii) and (iii) are just consequences of our definition of the map $Z_{-}$, so it suffices to show (iv) and (v).
    For (iv), assume that $\gamma \in H^1(X_{/B})$. In this case, the associated integrand $\omega_{\Gamma} = \beta \wedge \pi_v^{\ast}\gamma \wedge \pi_v^{\ast}\gamma$ for some forms $\beta$. Since $\pi_v^{\ast}\gamma \wedge \pi_v^{\ast}\gamma =0$, we get $\omega_{\Gamma} =0$. The case that $\gamma = s \in S^{\ast}$ is similarly proved by replacing $\pi_v^{\ast}\gamma$ with $\mathsf{D}^{\mathbb{C}} G_{s_0} (x_v, s)$ for a $s \in S^{\ast}$. (v) Since we associate the same 1-forms $\mathsf{D}^{\mathbb{C}} G_{s_0} $ on two internal edges of the graph, the associated integral vanishes since $\omega \wedge \omega =0$ for any 1-form $\omega$.
\end{proof}

The following is a key theorem that states that $Z_{-}$ defines a CDGA morphism (up to sign) though $Z_{\Gamma}$ is given as a fiber integration of non-smooth form $\omega_{\Gamma}$ on compactified configuration space.
\begin{theorem}\label{thm:key_thm}
    Let $\bC_{\twi}=\{(1-u,u) \in \bC^2\} \subset \bC^2$ be the twistor line. Then, the restriction of the map $Z_{-}$ to $\bC_{\twi}$ 
    \begin{equation}
        Z'_{-}:=  Z_{-}|_{B \times \bC_{\twi}} : \mathcal{D}_{H, S^{\ast}, \mathrm{def}}^{\vee,0} \rightarrow \mathcal{D}^{\ast, \ast}(B) \otimes \Omega_{\bC_{\twi}}^{\leq 1}
    \end{equation}
    satisfies the equation
    \begin{equation}\label{eq:cdga_v}
    (d_B + d_{\mathbb{C}_{\twi}}) Z_{\Gamma}' = - Z_{\partial_{\mathrm{Cas}}\Gamma}' + Z_{\partial_{\Delta} \Gamma}' - Z_{\partial_{S^{\ast}}\Gamma}'.
\end{equation}
\end{theorem}
\begin{proof}
   To begin with, we show that the map $Z_{-}$ is a morphism of commutative graded algebras, i.e., for two decorated connected graphs $\Gamma$ and $\Gamma'$ $Z_{\Gamma \sqcup \Gamma'} = Z_{\Gamma} \wedge Z_{\Gamma'}$ and $Z_{\Gamma \sqcup \Gamma'} = (-1)^{\deg(\Gamma) \deg(\Gamma')} Z_{\Gamma' \sqcup \Gamma}$. By construction \eqref{eq:def_omega_gamma} and suppressing $S^{\ast}$-decorations for simplicity, we have
    \begin{equation}\label{eq:omega_gg'}
        \omega_{\Gamma \sqcup \Gamma} = \iota^{\ast}(p_{V_{\Gamma}}^{\ast} \omega_{\Gamma} \wedge p^{\ast}_{V_{\Gamma'}} \omega_{\Gamma'}) \in \Omega^{\bullet}(\Conf_{V_{\Gamma \sqcup \Gamma'}^{\mathrm{int}}}(X_{/B})) \otimes \Omega_{\mathbb{C}_{\twi}}^{\leq 1}
    \end{equation}
    where $\iota: \Conf_{V_{\Gamma \sqcup \Gamma'}^{\mathrm{int}}}(X_{/B}) \rightarrow \Conf_{V_{\Gamma}^{\mathrm{int}}}(X_{/B}) \times \Conf_{V_{\Gamma'}^{\mathrm{int}}}(X_{/B})$ is the canonical inclusion map, and $p_{V_{\Gamma}}: \Conf_{V_{\Gamma}^{\mathrm{int}}}(X_{/B}) \times \Conf_{V_{\Gamma'}^{\mathrm{int}}}(X_{/B})  \rightarrow \Conf_{V_{\Gamma}^{\mathrm{int}}}(X_{/B})$ and  $p_{V_{\Gamma'}}: \Conf_{V_{\Gamma}^{\mathrm{int}}}(X_{/B}) \times \Conf_{V_{\Gamma'}^{\mathrm{int}}}(X_{/B}) \rightarrow \Conf_{V_{\Gamma'}^{\mathrm{int}}}(X_{/B})$ are projection maps forgetting points indexed by $V_{\Gamma'}^{\mathrm{int}}\cup V_{\Gamma'}^{\mathrm{ext},S^{\ast}}$ and $V_{\Gamma}^{\mathrm{int}}\cup V_{\Gamma}^{\mathrm{ext},S^{\ast}}$ respectively. This decomposition \eqref{eq:omega_gg'} implies that $Z_{\Gamma \sqcup \Gamma'} = Z_{\Gamma} \wedge Z_{\Gamma'}$. Since we have $\Gamma \sqcup \Gamma' = (-1)^{\deg(\Gamma) \deg(\Gamma')} \Gamma' \sqcup \Gamma$ and $Z_{\Gamma} \wedge Z_{\Gamma'} = (-1)^{\deg(\Gamma) \deg(\Gamma')} Z_{\Gamma'}\wedge Z_{\Gamma}$, we have shown that the $Z_{-}$ is a well-defined morphism of commutative graded algebras.

    Next, we prove that the map $Z_{-}'$ satisfies \eqref{eq:cdga_v}. For this, we will use several ideas presented in \cite{AS92, AS94}. Similar to \cite[Section 5]{AS92}, for a sufficiently small positive real number $\epsilon >0$,
    we take a collar neighbourhood $\mathcal{N}_{\epsilon,b}$ of $\partial p^{-1}_{\Gamma}((b, (1-u,u))$ (cf. \cite[Theorem 17]{hajek14}). Then, there is a diffeomorphism 
    \begin{equation}\label{eq:diff_F}
        F:  \partial p^{-1}_{\Gamma}((b, (1-u,u)) \times [0, \epsilon) \simeq \mathcal{N}_{\epsilon,b}
    \end{equation}
    such that $F(v, 0)$ is identity map for $v \in \partial p^{-1}_{\Gamma}((b, (1-u,u))$. Since $\omega_{\Gamma}$ gives convergent integrals on $p^{-1}_{\Gamma}((b, (1-u,u))$,  we have 
    \begin{equation}
       Z_{\Gamma}' = p_{\Gamma}{}_{\ast} \omega_{\Gamma} = \lim_{\epsilon \to 0} \int_{p^{-1}_{\Gamma}((b,(1-u, u)) \setminus \mathcal{N}_{\epsilon,b}}\omega_{\Gamma}.
    \end{equation}
    Therefore, by setting $p^{-1}_{\Gamma, \epsilon}((b,u)):=p^{-1}_{\Gamma}((b,(1-u, u)) \setminus \mathcal{N}_{\epsilon,b}$ we have
    \begin{equation}
        \begin{split}
        (d_B + d_{\mathbb{C}_{\twi}})  Z_{\Gamma}' & =  \lim_{\epsilon \to 0} (d_B + d_{\mathbb{C}_{\twi}})  \int_{p^{-1}_{\Gamma, \epsilon}((b,u))}\omega_{\Gamma}\\
        &=  \lim_{\epsilon \to 0}\left[ \int_{p^{-1}_{\Gamma, \epsilon}((b,u))} (\mathbf{d} + d_{\mathbb{C}_{\twi}})\omega_{\Gamma} +  \int_{\partial p^{-1}_{\Gamma, \epsilon}((b,u))} \omega_{\Gamma}|_{\partial p^{-1}_{\Gamma, \epsilon}((b,u))}  \right].
        \end{split}  
    \end{equation}
    Here, in the second equality, we use the Stokes formula on $p^{-1}_{\Gamma, \epsilon}((b, u)$. As we will see below, the two integrations on the right-hand side of the second equality are convergent as $\epsilon \to 0$, so we can justify the exchange of the order of differentials and integrations. Then, it is enough to show the following (I) and (II):
    \begin{enumerate}[(I)]
    \item $\displaystyle \lim_{\epsilon \to 0}\int_{p^{-1}_{\Gamma, \epsilon}((b, u)} (\mathbf{d} + d_{\mathbb{C}_{\twi}})\omega_{\Gamma}= -Z_{\partial_{\mathrm{Cas}}\Gamma}'$,
    \item $\displaystyle \lim_{\epsilon \to 0} \int_{\partial p^{-1}_{\Gamma, \epsilon}((b, u)} \omega_{\Gamma}|_{\partial p^{-1}_{\Gamma, \epsilon}(b,u)}= Z_{\partial_{\Delta} \Gamma}' - Z_{\partial_{S^{\ast}}\Gamma}'$
    \end{enumerate}
where note that $Z_{\partial_{\mathrm{Cas}}\Gamma}$, $Z_{\partial_{S^{\ast}}\Gamma}$ and $Z_{\partial_{S^{\ast}}\Gamma}$ are well-defined by Lemma \ref{lem:convergence}. 
For equation (I), the proof is given straightforward computation as follows. First note that $\displaystyle \lim_{\epsilon \to 0}\int_{p^{-1}_{\Gamma, \epsilon}((b, u)} (\mathbf{d} + d_{\mathbb{C}_{\twi}})\omega_{\Gamma}$ is convergent integral. Indeed, recalling that we have
\begin{equation}
(\mathbf{d} + d_{\mathbb{C}_{\twi}}) \mathsf{D}^{\mathbb{C}} G_{s_0} = d d^{\mathbb{C}} G_{s_0} = - \langle \nu \wedge \nu \rangle^{[1\times 1]} =: K
\end{equation}
off-diagonal and it canonically extends to the compactified configuration space, the differential of $\omega_{\Gamma}$ is given over the fiber of $(b,u) \in B \times \bC_{\twi}$
\begin{equation}\label{eq:diff_omega}
    (\mathbf{d} + d_{\mathbb{C}_{\twi}}) \omega_{\Gamma} = \sum_{E_i \in E_{\Gamma}^{\mathrm{int}} \sqcup E_{\Gamma}^{\mathrm{ext}}}(-1)^i p_{E_{i}}^{\ast} K \wedge \sum_{E \in E_{\Gamma}\setminus E_i} \omega_E
\end{equation}
so that the associated integral is well-defined. Here, the reason why we ignore the action of differential $\mathbf{d} + d_{\mathbb{C}_{\twi}}$ on 1-forms associated with special decorated external vertices is that on each fiber these forms are sections of local systems. Thus, one can directly see that 
\begin{equation}
\begin{split}
     \lim_{\epsilon \to 0}\int_{p^{-1}_{\Gamma, \epsilon}((b, u)} (\mathbf{d} + d_{\mathbb{C}_{\twi}})\omega_{\Gamma} & = \int_{p^{-1}_{\Gamma}((b, u)} (\mathbf{d} + d_{\mathbb{C}_{\twi}})\omega_{\Gamma}\\
     &=  \sum_{E_i \in E_{\Gamma}^{\mathrm{int}} \sqcup E_{\Gamma}^{\mathrm{ext}}} \int_{p^{-1}_{\Gamma}((b, t)}  \left( (-1)^i  p_{E_{i}}^{\ast} K \wedge \sum_{E \in E_{\Gamma}\setminus E_i} \omega_E \right) \\
     &=-Z_{\partial_{\mathrm{Cas}}\Gamma}'.
\end{split}
\end{equation}
Here, the minus sign in the last equality comes from the difference of the sign of the Casimir element $K$ above and one in $\partial_{\mathrm{Cas}}$.

For equation (II), we need more careful arguments in the evaluation of integration on the boundary, since $ \mathsf{D}^{\mathbb{C}} G_{s_0}$ does not extend to a form on $C_2(X)$. Indeed, by similar computation as in the proof of Theorem \ref{thm:ext_green}, $ \mathsf{D}^{\mathbb{C}} G_{s_0}$ near the diagonal is expressed as
\begin{equation}
\begin{split}
       \mathsf{D}^{\mathbb{C}} G_{s_0} &= \frac{1}{2}(-1 +2u) d G_{s_0} + \frac{1}{2} d^{\bC} G_{s_0} + G_{s_0} du \\
       &= -\frac{1}{2 \pi \sqrt{-1}}(-1+2u) \frac{dr}{r} + \omega_{S^1} + \frac{\log r}{2 \pi \sqrt{-1}} du + \text{(smooth part)}
       \end{split}
\end{equation}
with the same notation as Theorem \ref{thm:ext_green}. 

Note that the boundary $\partial p^{-1}_{\Gamma, \epsilon}((b,u))=\partial (p^{-1}_{\Gamma}((b,(1-u, u)) \setminus \mathcal{N}_{\epsilon,b})$ is diffeomorphic to the unions $ \bigcup_{S \in \mathcal{S}; |S|=2}  \partial_{S} C_n(X) \times \{ \epsilon\}$ where we set $C_n(X):=C_{n}(X_b;(s_1,\ldots, s_k))_{(s_0,v)}$ with $n=|V_{\Gamma}^{\mathrm{int}}|$. Therefore, the computation of the boundary contribution can be divided into that of $\partial_{S} C_n(X) \times \{ \epsilon\}$ for each $S \in \mathcal{S}$ with $|S|=2$. Moreover, to know $\lim_{\epsilon \to 0} p^{\partial}_{\Gamma}{}_{\ast} (\omega_{\Gamma}|_{\partial_S C_n(X) \times \{\epsilon\}})$, we can use the coordinate system on $\partial_S C_n(X) \times \{\epsilon\}$ induced by diffeomorphisms $\psi$ in Section \ref{ssecion:3.1} since the diffeomorphism $F$ in \eqref{eq:diff_F} and $\psi$ are identity at $\epsilon = 0$ and their difference converges to 0 as $\epsilon \to 0$. Recall that the face $\partial_{\mathcal{S}} C_n(X) $ in $\partial C_n(X)$ associated with the collision of $q$ points has the structure of fiber bundle over $C_{n-q+1}(X)$ with the $(2q- 3)$-dimensional fiber $F_{\mathcal{S}}$ isomorphic to $C_q(\mathbb{R}^2)/\mathbb{R}^2 \rtimes \mathbb{R}_{+}$. Here, $\mathbb{R}^2$ and $\mathbb{R}_{+}$ act on $C_q(\mathbb{R}^2)$ by translations and scalings respectively. 

Now, we compute the boundary contribution $p^{\partial}_{\Gamma}{}_{\ast} (\omega_{\Gamma}|_{\partial C_n(X) \times \{\epsilon\}})$. The arguments are divided into the following three cases corresponding to the types of boundary strata of $C_n(X)$ (cf. Section \ref{section:boundary_strata}):
    \begin{itemize}
    \item[(II-1)] Contributions from the principal faces. From such faces, we will obtain the terms $Z_{\partial_{\Delta} \Gamma} - Z_{\partial_{S^{\ast}}\Gamma}$.
    \item[(II-2)] Contributions from the hidden faces where the integration is shown to vanish by Kontevich's vanishing lemma with an additional argument.
    \item[(III-3)] Contributions from the infinite faces. On these faces, we show that all the integrands $\omega_{\Gamma}$ vanish since $\omega_{\Gamma}$ have tame logarithmic singularity at $s_0$.
    \end{itemize}
\noindent
Case (II-1): Case of principal faces (case that $q=2$ and two points collapse away from $s_0$). In this case, the fiber integration does not vanish in general. Assume that two points that are connected by an edge $E$ collide. First, suppose that the edge $E=(i,j)$ is internal. Then,  restriction of $\mathsf{D}^{\bC} G_E$ to the boundary stratum is given by up to smooth part
\begin{equation}
     \mathsf{D}^{\mathbb{C}} G_{s_0}(x_i,x_j)|_{ \partial_{\{\{i,j\}\}}C_2(X) \times \{\epsilon\}} \sim  \omega_{S^1} + \frac{\log \epsilon}{2 \pi \sqrt{-1}} du.
\end{equation}
Therefore, for each fixed $\epsilon$, we can assume that singularity $\frac{dr}{r}$ does not contribute when only two points collide. This means that we can replace $\mathsf{D}^{\mathbb{C}} G_{s_0}=\left((1-u) \partial - u \bar{\partial} -du \right) G_{s_0}$ with $(\frac{1}{2}d^{\mathbb{C}} - du) G_{s_0}$ for this boundary computation. 
After such replacement, the propagator is decomposed into extendable and non-extendable forms as  $\frac{1}{2} d^{\mathbb{C}}G_{s_0} - G_{s_0} du $. At least for the term $\frac{1}{2} d^{\mathbb{C}}G_{s_0}$, we know that it is extendable on $C_2(M)$. In contrast, the term $G_{s_0} du$ may diverge as $\epsilon \to 0$. Then, despite the existence of $G_{s_0} dt$ term, the usual argument of the theory of configuration space integral can be applied. Indeed, if $G_{s_0} dt$ is assigned on the edge that connects the two collapsing vertices, then no 1-forms along the $S^1$ fiber direction survive as $\epsilon \to 0$ so that the limit converges to zero. Thus, it is enough to consider the case that $\frac{1}{2} d^{\mathbb{C}} G_{s_0}$ is assigned on the edge $E$ with two vertices colliding to one point. Thus, we obtain
\begin{equation}
\begin{split}
    \lim_{\epsilon \to 0} \int_{\partial_{\mathcal{S}} p^{-1}_{\Gamma, \epsilon}} \omega_{\Gamma}|_{\partial_{\mathcal{S}} p^{-1}_{\Gamma, \epsilon}} 
    &= \int_{C_{V_{\Gamma/E}^{\mathrm{int}}}(X)} \omega_{\Gamma/E}\\
    &= Z_{\Gamma/E}
    \end{split}
\end{equation}
where in the second equality we use the fiber integration of $\frac{1}{2} d^{\mathbb{C}} G_a$ along $S^1$ fiber. By taking all over the internal edges $E$ of $\Gamma$, we obtain $Z_{\partial_{\Delta} \Gamma}$.

Next, consider the case that $E=(x,s)$ is an $S^{\ast}$-decorated external edge. In this case, the boundary contribution corresponds to the map $\partial_{S^{\ast}}$. More concretely, the integration of the term 

\begin{equation}
\begin{split}
    \omega_{\Gamma} & = \pm \prod_{E' \in E_{\Gamma}: E'=(a,b), a,b \neq x} \omega_{E'} \wedge \prod_{E=(y,x) \in E_{\Gamma}} \mathsf{D}^{\mathbb{C}} G(y,x)\wedge \left( \frac{1}{2} d^{\mathbb{C}} G(x,s)\right)\\
    & \pm \prod_{E' \in E_{\Gamma}: E'=(a,b), a,b \neq x} \omega_{E'} \wedge \prod_{E=(y,x) \in E_{\Gamma}} \mathsf{D}^{\mathbb{C}} G(y,x)\wedge G(x,s)du
\end{split}
\end{equation}
on the boundary face corresponding to the collapsing of $x$ with $s$ yields the integration of 
\begin{equation}
 \pm \prod_{E' \in E_{\Gamma}: E'=(a,b), a,b \neq x} \omega_{E'} \wedge \prod_{E=(y,x) \in E_{\Gamma}} \mathsf{D}^{\mathbb{C}} G(y,s)
\end{equation}
over $C_{V_{\Gamma}^{\mathrm{int}}}$ after fiber integration along $S^1$ fiber. By taking summation over all $S^{\ast}$-decorated external edges of $\Gamma$, we get the term $-Z_{\partial_{S^{\ast}}\Gamma}$.  
Also, note that the singular term $(\log \epsilon) du$ does not contribute to the integration since they have no forms along the $S^1$ direction as above.

Hence, from the principal faces we obtain the desired terms
\begin{equation}
    Z_{\partial_{\Delta} \Gamma}' - Z_{\partial_{S^{\ast}}\Gamma}'
\end{equation}
where the minus sign on the right-hand side reflects the difference in \eqref{eq:GC_family}.

\noindent
Case (II-2): Case of hidden faces (the case that $q \geq 3$ and $q$-points collapse away from $s$). To consider this case, we use the following theorem.

\begin{lemma}{(Kontsevich's vanishing lemma \cite[Lemma 6.4]{Kontsevich_dqpm03}} \label{thm:Kontsevich}
	Let $F_{\mathcal{S}}$ be the fiber of the face $\partial_{\mathcal{S}} C_n(X)$ corresponding to the collapse of $q$ points with coordinate $\mathbf{x}_1,\ldots, \mathbf{x}_q$. Let $d \theta \in \Omega^1(S^1; \bR)$ be standard volume form of $S^1$. Let $\pi_{ij}: F_{\mathcal{S}} \rightarrow S^1$ be the projection defined as
 \begin{equation}
     \pi_{ij}: F_{\mathcal{S}} \rightarrow S^1; \quad (\mathbf{x}_1,\ldots, \mathbf{x}_q) \mapsto \frac{\mathbf{x}_j - \mathbf{x}_i}{|\mathbf{x}_j - \mathbf{x}_i|}\quad (i \neq j)
 \end{equation}
 and $d\phi_{ij}:= \pi_{ij}^{\ast} d\theta$ be the pullback of $d\theta$ via $\pi_{ij}$. Then, for any two sequences $s_i, t_i$ $(i=1,\ldots, 2q-3)$ of integers with $s_i \neq t_i$ $(1 \leq s_i,t_i \leq q)$, the integral vanishes:
    \begin{equation}
         \int_{F_{\mathcal{S}}} \bigwedge_{i=1}^{2q-3} d\phi_{s_i t_i}  =0.
    \end{equation}
\end{lemma}

Note that, in our case, hidden faces to be considered are those corresponding to collisions of three points since we consider only uni-trivalent graphs with loop number $\leq 1$ and $X$ is of real dimension 2. Otherwise, the associated integrands are forms with vertical degree strictly less than the top degree of $(2q- 3)$-dimensional fiber $F_{\mathcal{S}}$ so that their integrations vanish by degree reason. Thus, we can assume that three points connected by three internal edges $E_1, E_2, E_3$ correspond to a hidden face.

Different from the case (II-1), we need to take care of the existence of $d G_{s_0}$ assigned to internal edges $E_1, E_2, E_3$ connecting colliding points. If one of the three edges $E_1, E_2, E_3$ is assigned with $G_{s_0} du$, then by degree reason the corresponding integral vanishes. Therefore, the singularities we need to consider are classified as follows:
\begin{enumerate}[(1)]
\item All of the edges $E_1, E_2, E_3$ are assigned with $\frac{1}{2} d^{\mathbb{C}} G_{s_0}$.
\item Two of the edges $E_1, E_2, E_3$ are assigned with $\frac{1}{2} d^{\mathbb{C}} G_{s_0}$ and one of them is assigned with  $d G_{s_0}$.
\item One of the edges $E_1, E_2, E_3$ is assigned with $\frac{1}{2} d^{\mathbb{C}} G_{s_0}$ and two of them are assigned with  $d G_{s_0}$.
\item All of the edges $E_1, E_2, E_3$ are assigned with $d G_{s_0}$.
\end{enumerate}
The case (1) immediately follows from Kontsevich's vanishing lemma \ref{thm:Kontsevich}. Note that the case (1) and (3) are related to each other as 
\begin{equation}
    d G_{E_1} \wedge d G_{E_2} \wedge d^{\mathbb{C}} G_{E_3} = d^{\mathbb{C}}G_{E_1} \wedge  d^{\mathbb{C}}G_{E_1} \wedge d^{\mathbb{C}} G_{E_3}.
\end{equation}
Therefore, the case (3) also vanishes. Similarly, the case (2) and (4) are related to each other, so we will show the vanishing of the case (4). The interior of the fiber $F_{\mathcal{S}}$ of the boundary stratum corresponding to the collision of the three points is isomorphic to $\mathrm{Conf}_3(\bR^2)/ \bR^2 \rtimes \bR_+$. As in \cite[\S 6.6]{Kontsevich_dqpm03}, we further identify $\mathrm{Int}\ F_{\mathcal{S}}$ with subspace $C_3' \subset \mathrm{Conf}_3(\bR^2)$ consisting of points $(0, \zeta, z)$ where $\zeta \in S^1$ and $z \in \bR^2 \setminus \{0,\zeta\}$. Then, we have a map $R: C_3' \rightarrow C_3''$ where $C_3'' \subset \mathrm{Conf}_3(\bR^2)$ consisting of points $(0, 1, w)$ with $w \in \bR^2 \setminus \{0,1\}$ by rotating around $0$. This map $R$ is fibration with $S^1$ fiber. The restriction of $d G_{E_1} \wedge d G_{E_2} \wedge d G_{E_3}$ to the $F_{\mathcal{S}}$ is zero under integration along $R$ since it has no volume form along the $S^1$ fiber direction. Therefore, we conclude the vanishing of the case (4) and hence (2).

\noindent
Case (II-3): Case of infinite faces (case that $q \geq 2$ and $q$-points collapse $s$) When $q=2$, the vanishing of the contributions from the strata follows from Lemma \ref{lem:infinite_vanishing_lemma}. The vanishing results follow when $q \geq 3$, using the same argument as Case (II-2).

Therefore, we conclude that 
\begin{equation}
    (d_B + d_{\mathbb{C}_{\twi}}) Z_{\Gamma}' = - Z_{\partial_{\mathrm{Cas}}\Gamma}' + Z_{\partial_{\Delta} \Gamma}' - Z_{\partial_{S^{\ast}}\Gamma}'.
\end{equation}
The assertion is proved.
\end{proof}

Next, we define the (graph valued) effective action as
\begin{equation}\label{eq:eff_action}
\begin{split}
   S^{\mathrm{eff}} &:= (Z_{-} \otimes \Id)\mathfrak{s} \\
   & =  \sum_{\Gamma} \frac{\hbar^{l(\Gamma)}}{|\Aut(\Gamma)|}  Z_{\Gamma} \cdot [\Gamma^{\vee}] \in (\mathcal{D}^{\bullet, \bullet}(B) \otimes \Omega^{\bullet}_{\mathbb{C}^2})\otimes  \pi^{\ast} H_0^{\partial_{\Delta}}(\mathcal{D}_{H_{/B}, S_{/B}^{\ast}})[[\hbar]]
\end{split}
\end{equation}
where $\Gamma$ runs over all connected uni-trivalent graphs without self-loops, $\Gamma \otimes \Gamma^{\vee}$ is pairwise decorated by the Caismir element $\gamma|\gamma^{\vee}$, and $\pi: B \times \bC^2 \rightarrow B$ is the canonical projection map. Note that the above series is independent of the choice of elements of the orientation torsor of $\Gamma$.

In \eqref{eq:eff_action}, we define $S^{\mathrm{eff}}$ as a formal power series on $\hbar$ but we see that it is in fact a polynomial of degree $1$ by dimensional reasons as follows.

\begin{lemma} \label{lem:4.2.-1}
The effective action $S^{\mathrm{eff}}$ is degree 1 polynomial on variable $\hbar$. In particular, no term is contained in $S^{\mathrm{eff}}$ which corresponds to trivalent graphs without external vertices.
\end{lemma}
\begin{proof}
The first statement is just a rephrase of Lemma \ref{lem:5.3.7}.  The second statement follows immediately from the first one because trivalent graphs without external vertices have at least 2 loops, which are coefficients of $\hbar^{l}$ with $l \geq 2$.
\end{proof}

\begin{remark}
    By Lemma \ref{lem:5.2.2} (iv), one never obtains graphs with self-loops after contracting one internal edge of graphs which appear in the effective action $S^{\mathrm{eff}}$, since $S^{\mathrm{eff}}$ does not contain graphs with non-regular edges by Lemma \ref{lem:4.2.-1}. For more details on the relation between self-loops and non-regular edges, see  \cite[Section 7]{KL23}.
\end{remark}

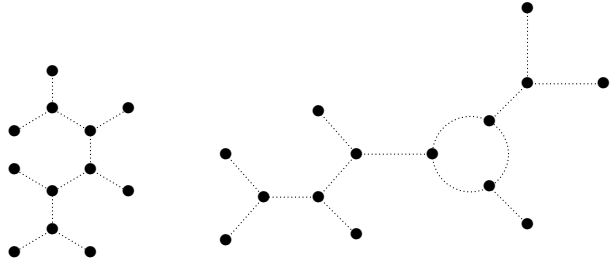
\begin{figure}[h]
\captionsetup{margin=2cm}
\centering
\begin{tikzpicture}
\begin{scope}
\draw[densely dotted] (0,0) -- (0.5, 0.3);
\draw[densely dotted] (0,0) -- (-0.5, 0.3);
\draw[densely dotted] (-0.5,0.3) -- (-1, 0);
\draw[densely dotted] (-0.5,0.3) -- (-0.5, 0.8);
    \draw[densely dotted] (0,0) -- (0,-0.5);
    \draw[densely dotted] (0,-0.5) -- (-0.5, -0.8);
     \draw[densely dotted] (0,-0.5) -- (0.5, -0.8);
      \draw[densely dotted] (-0.5,-0.8) -- (-1, -0.5);
      \draw[densely dotted] (-0.5,-0.8) -- (-0.5, -1.3);
      \draw[densely dotted] (-0.5,-1.3) -- (-1, -1.6);
      \draw[densely dotted] (-0.5, -1.3) -- (0, -1.6);
\foreach \x in {(0,0),(0.5, 0.3),(-0.5, 0.3),(-1, 0),(-0.5, 0.8),(-0.5, -0.8),(0,-0.5), (-1, -0.5),(-0.5,-1.3),(-1, -1.6),(0, -1.6),(0.5, -0.8) }{
\node at \x {$\bullet$};
}
\end{scope}
\begin{scope}[xshift=5cm, yshift=-0.3cm]
    \draw[densely dotted] (0,0) circle (0.5cm);
    \draw[densely dotted] (-0.5,0.0)--(-1.5,0.0);
    \draw[densely dotted] (-1.5,0.0)--(-2.0,0.5*1.14);
    \draw[densely dotted] (-1.5,0.0)--(-2.0,-0.5*1.14);
     \node at (-0.5, 0) {$\bullet$};
     \node at (-1.5, 0) {$\bullet$};
     \node at (-2.0,-0.5*1.14) {$\bullet$};
     \node at (-2.0,0.5*1.14) {$\bullet$};
     \draw[densely dotted] (-2.0+0.5,-0.5*1.14-0.5)--(-2.0,-0.5*1.14);
     \node at (-2.0+0.5,-0.5*1.14-0.5) {$\bullet$};
     \draw[densely dotted] (-2.0,-0.5*1.14)--(-2.72,-0.5*1.14);
     \node at (-2.72,-0.5*1.14) {$\bullet$};
     \draw[densely dotted] (0.25,0.433)--(0.25+0.5,0.433+0.5);
     \draw[densely dotted] (0.25,-0.433)--(0.25+0.5,-0.433-0.5);
     \draw[densely dotted] (-2.72,-0.5*1.14)--(-3.22,-0.5*1.14+0.5*1.14);
    \draw[densely dotted] (-2.72,-0.5*1.14)--(-3.22,-0.5*1.14-0.5*1.14);
    \node at (-3.22,-0.5*1.14-0.5*1.14) {$\bullet$};
    \node at (-3.22,-0.5*1.14+0.5*1.14) {$\bullet$};
     \node at (0.25+0.5,0.433+0.5) {$\bullet$};
     \node at (0.25,0.433) {$\bullet$};
     \node at (0.25+0.5,-0.433-0.5) {$\bullet$};
     \node at (0.25,-0.433) {$\bullet$};
      \node at (0.25+0.5+1,0.433+0.5) {$\bullet$};
      \draw[densely dotted] (0.25+0.5,0.433+0.5) -- (0.25+0.5+1,0.433+0.5);
       \node at (0.25+0.5,0.433+0.5+1) {$\bullet$};
      \draw[densely dotted] (0.25+0.5,0.433+0.5) -- (0.25+0.5,0.433+0.5+1);
\end{scope}
\end{tikzpicture}
\caption[Examples of graphs which appears in $S^{\mathrm{eff}}$]{Examples of graphs which appears in $S^{\mathrm{eff}}$. The coefficients of $S^{\mathrm{eff}}$ correspond to connected trivalent graphs such that they must have external vertices and their number of loops is less than or equal to one.}
\end{figure}

We end this section by giving our main result stating that $S^{\mathrm{eff}}$ satisfies a version of the quantum master equation. For this, we redefine the effective action $S^{\mathrm{eff}}$ by taking Tate twist $\bR(-1)$ to the uni-trivalent tree part as follows.

\begin{equation}\label{eq:eff_action}
\begin{split}
   \mathbf{S}^{\mathrm{eff}} =  \sum_{\Gamma; l(\Gamma)=0} \frac{1}{|\Aut(\Gamma)|}  Z_{\Gamma} \cdot [\Gamma^{\vee}]\otimes \bR(-1)  + \sum_{\Gamma; l(\Gamma)=1} \frac{\hbar}{|\Aut(\Gamma)|} Z_{\Gamma} \cdot [\Gamma^{\vee}]
\end{split}
\end{equation}
For later usage, we denote it as
\begin{equation}\label{eq:S_eff_dec}
    \mathbf{S}^{\mathrm{eff}} = \mathbf{S}^{\mathrm{eff}}_{\text{tree}} + \hbar \mathbf{S}^{\mathrm{eff}}_{\text{1-loop}}.
\end{equation}
\begin{theorem}\label{thm:main}
\begin{enumerate}[(1)]
\item The effective action $\mathbf{S}^{\mathrm{eff}}$ is a $\bC^{\times} \times \bC^{\times}$-equivariant real connection 1-form valued in the DG Lie algebra $\pi^{\ast} H_0^{\partial_{\Delta}}(\mathcal{CD}^{\bullet}_{H, S^{\ast}}[-1])$.
\item The restriction $\mathbf{S}^{\mathrm{eff}}{}'$ of $\mathbf{S}^{\mathrm{eff}}$ to the twistor line $\bC_{\twi} \subset \bC^2$ satisfies a non-commutative analogue of quantum master homotopy equation \eqref{eq:QMHE} 
 \begin{equation}\label{eq:5.2.33}
        ((d_B + d_{\mathbb{C}_{\twi}}) \otimes \Id) \mathbf{S}^{\mathrm{eff}}{}' + \frac{1}{2}[\mathbf{S}^{\mathrm{eff}}{}', \mathbf{S}^{\mathrm{eff}}{}'] + \hbar(\Id \otimes \delta) \mathbf{S}^{\mathrm{eff}}{}'=0 
    \end{equation}
    where we set $[-,-]$ as the wedge product of Lie algebra-valued differential forms.
\end{enumerate}

\end{theorem}

\begin{proof}
(1) For the tree part, this is the same as \cite[\S 7.3.1]{GoncharovHodge1} by Hodge bidegree counting and using the property \eqref{eq:reality_prop}. For 1-loop terms, the essential argument is the same as the tree case except that any adjustment by twisting by $\bR(n)$ is not needed since $|E_{\Gamma}| = 2 |V_{\Gamma}|$ holds for a connected uni-trivalent 1-loop graph. (2) This is a consequence of Theorem \ref{thm:formal_qme} and Theorem \ref{thm:key_thm}. Indeed, since $Z_{-}'$ satisfies the relation $(d_B + d_{\mathbb{C}_{\twi}}) Z_{-}'=-Z_{-}' \partial_{\mathrm{Cas}} -Z_{-}' \partial_{S^{\ast}} + Z_{-}' \partial_{\Delta}$, we have

\begin{equation}
    \begin{split}
         ((d_B + d_{\mathbb{C}_{\twi}}) \otimes \Id) \mathbf{S}^{\mathrm{eff}}{}' & =  ((d_B + d_{\mathbb{C}_{\twi}}) \otimes \Id) (Z_{-}' \otimes \Id)\mathfrak{s}\\
         &= -(Z_{-}'\otimes \Id) ((\partial_{\mathrm{Cas}} - \partial_{\Delta} + \partial_{S^{\ast}}) \otimes \Id)\mathfrak{s}\\
         &=  -(Z_{-}'\otimes \Id) \left(\frac{1}{2}[\mathfrak{s}, \mathfrak{s}] + \hbar(\Id \otimes \delta)\mathfrak{s}\right)\\
         &= - \frac{1}{2}[\mathbf{S}^{\mathrm{eff}}{}',  \mathbf{S}^{\mathrm{eff}}{}'] - \hbar(\Id \otimes \delta) \mathbf{S}^{\mathrm{eff}}{}'.
    \end{split}
\end{equation}
Thus, the desired equation is obtained.
\end{proof}

Note that Theorem \ref{thm:main} can be regarded as one of the generalizations of \cite[Theorem 7.8]{GoncharovHodge1} for the Hodge correlator class $\mathbf{G}$. The next section explains it explicitly (Theorem \ref{thm:tree_red}).
\section{Some Properties of generalized Hodge correlator twistor connection}\label{section:6}
This section describes several properties of our generalized version of Hodge correlator connection. To begin with, we relate our class and the original one defined by Goncharov (\cite[Definition 7.6]{GoncharovHodge1}). Then, a bit of detail of one loop term of our connection is studied. Finally, we show that a special case of our connection is related to the holomorphic/anti-holomorphic formal KZ connection along the twistor line, that is, our connection is master homotopic to the formal KZ connection. 

\subsection{Tree reduction of $\mathbf{S}^{\mathrm{eff}}$ and Hodge correlator class $\mathbf{G}$}

In this subsection, we make it explicit how $\mathbf{S}^{\mathrm{eff}}$ relates to  Goncharov's original Hodge correlator twistor connection.

For this, first we briefly recall that the cyclic envelope $\calCV_{X,S^{\ast}}$ (\ref{eq:2.1.2}) can be identified with decorated planar uni-trivalent trees which vanish under the map $\partial_{\Delta}$ following \cite[\S 6]{GoncharovHodge1}. Let $(\mathcal{T}^{\vee, 1}_{H,S^{\bullet}}, \partial_{\Delta})$ be the graded differential vector space of planar trees decorated by $V_{X,S^{\ast}}^{\vee} = H^1(X; \bC) \oplus \bC[S^{\ast}]$ where the grading structure and the differential $\partial_{\Delta}$ are defined as the same way as in Section \ref{section:4.4}. Then, there is an injective linear map
\begin{equation}\label{eq:sum_over_tree}
F: \calCV_{X,S^{\ast}} \hookrightarrow \mathcal{T}^{\vee, 1}_{H,S^{\ast}}[1];\quad W \mapsto \sum_{T} (T, W; \mathrm{Or}_T)
\end{equation}
where the summation runs over all uni-trivalent planar tree graphs decorated by the cyclic word $W$ equipped with the canonical orientation $\mathrm{Or}_T$ (cf. \ref{section:4.3}). Then, it turns out that the image of $F$ coincides with $H^0_{\partial_{\Delta}}(\mathcal{T}^{\vee, \bullet}_{H,S^{\ast}}[1])$, i.e., we have
\begin{equation}
    \calCV_{X,S^{\ast}} = H^0_{\partial_{\Delta}}(\mathcal{T}^{\vee, \bullet}_{H,S^{\ast}}[1]).
\end{equation}

Let $\mathbf{S}^{\mathrm{eff}}_{\mathrm{tree}}:= \mathbf{S}^{\mathrm{eff}}|_{\hbar =0}$ denote the tree reduction of $\mathbf{S}^{\mathrm{eff}}$, that is,  $\mathbf{S}^{\mathrm{eff}}_{\mathrm{tree}}$ is obtained from $\mathbf{S}^{\mathrm{eff}}$ by setting uni-trivalent graphs with at least one loop to be zero. 

Next, we briefly recall the Hodge correlator class $\mathbf{G}$ in our situation. For details and original definition, see \cite[\S 7.2]{GoncharovHodge1}. 
\begin{equation}
    \mathbf{G} := \sum_{W \in \mathcal{C}^{\vee}_{X,S^{\ast}}} \sum_{(T, W; \mathrm{Or}_T)} \frac{1}{|\Aut(W)|} Z_{(T, W; \mathrm{Or}_T)} \cdot W^{\vee}(-1)
\end{equation}
where $(T, W; \mathrm{Or}_T)$ runs over all planar uni-trivalent trees decorated by $W$ with the canonical orientation $\mathrm{Or}_T$, $Z_{(T, W; \mathrm{Or}_T)}$ is by \eqref{eq:5.3.27} and $|\Aut(W)|$ is the number of automorphisms of the cyclic word $W$. It turns out that the class $\mathbf{G}$ is $\pi^{\ast}(\mathcal{CL}ie_{X,S^{\ast}})(-1)$-valued 1-current on $B \times \bC^2$ (\cite[Lemma 7.7]{GoncharovHodge1}). Note that $\mathbf{G}$ has the following expansion in terms of Hodge bidegree
\begin{equation}
        \mathbf{G} = \mathbf{G}_{0,0} + \sum_{s,t \geq 0} z^s w^t ((z dw - wdz) \wedge (s + t +1) \mathbf{G}_{s+1, t+1}  + ( w \partial - z \bar{\partial}) \mathbf{G}_{s+1, t+1})
\end{equation}
where $\mathbf{G}_{s+1,t+1}$ is of Hodge bidegree $(-s-1,-t-1)$ (cf. \eqref{eq:2.3.21}). Then, we will state the following theorem relating $\mathbf{G}$ and $\mathbf{S}^{\mathrm{eff}}_{\mathrm{tree}}$.
\begin{theorem}\label{thm:tree_red}
    Keep the same notation as above. The map $F$ given in \eqref{eq:sum_over_tree} induces a canonical identification
    \begin{equation}\label{eq:tree_red}
        \mathbf{G} = (\Id \otimes F)^{\ast} \mathbf{S}^{\mathrm{eff}}_{\mathrm{tree}}.
    \end{equation}
\end{theorem}
\begin{proof}
It suffices to show that, for each cyclic word $W \in \calCV_{X,S^{\ast}}$, the evaluations of $\mathbf{G}$ and $(\Id \otimes F)^{\ast} \mathbf{S}^{\mathrm{eff}}_{\mathrm{tree}}$ at $W$ coincide. Since we use the inner product of uni-trivalent graphs given as \eqref{eq:inner_prod_graph}, taking the pairing of graphs yields the factors of numbers of automorphisms of graphs. This implies that the evaluation of $S^{\mathrm{eff}}_{\mathrm{tree}}$ at $F(W)$ is given as the form
\begin{equation}
     \sum_{(T, W; \mathrm{Or}_T)} Z_{(T, W; \mathrm{Or}_T), (b,(z,w))} 
\end{equation}
where the summation runs over all uni-trivalent tree graphs decorated by the cyclic word $W$ equipped with the canonical orientation $\mathrm{Or}_T$. Notice that this is the same as the value of $\mathbf{G}$ at $W$ by definition. This leads to the desired formula \eqref{eq:tree_red}.
\end{proof}

\subsection{One loop term of $\mathbf{S}^{\mathrm{eff}}$}

Next, this subsection studies one loop term of $\mathbf{S}^{\mathrm{eff}}$ and shows that this term is given by linear combinations of integrations defined by using a 2-dimensional analogue of Chern--Simons propagators as in \cite{CW} with a little modification.

Let $\mathbf{S}^{\mathrm{eff},1/2}$ denote the restriction of $\mathbf{S}^{\mathrm{eff}}$ on $(\frac{1}{2}, \frac{1}{2}) \in \bC_{\twi} \subset \bC^2$. 
Then, $\mathbf{S}^{\mathrm{eff},1/2}$ defines a  $H_0^{\partial_{\Delta}}(\mathcal{D}_{H, S^{\ast}})[[\hbar]]$-valued current on $B$ and has the following explicit formula.
\begin{proposition}\label{prop:6.2.1}
The restriction $\mathbf{S}^{\mathrm{eff},1/2}$ of $\mathbf{S}^{\mathrm{eff}}$ on $(\frac{1}{2}, \frac{1}{2}) \in \bC_{\twi} \subset \bC^2$ can be written as the following form:
    \begin{equation}
    \begin{split}
    \mathbf{S}^{\mathrm{eff},1/2} = \mathbf{S}^{\mathrm{eff},1/2}_{\text{tree}} + \hbar \mathbf{S}^{\mathrm{eff},1/2}_{\text{1-loop}}
    \end{split}
\end{equation}
where
\begin{equation}\label{eq:S_one_loop}
\begin{split}
&\mathbf{S}^{\mathrm{eff},1/2}_{\text{tree}} =\mathbf{G}_{0,0} + \sum_{s,t \geq 0} \left(\frac{1}{2}\right)^{s+t+1} d^{\mathbb{C}} \mathbf{G}_{s+1, t+1},\\
    & \mathbf{S}^{\mathrm{eff},1/2}_{\text{1-loop}} = \sum_{\Gamma; l(\Gamma)=1}  \frac{1}{|\Aut(\Gamma)|} I_{\Gamma} \cdot[\Gamma^{\vee}].
\end{split}
\end{equation}
Here, the summation runs over all connected decorated uni-trivalent graphs with one loop, and $I_{\Gamma}$ is defined as
\begin{equation}
    I_{\Gamma} := p_{\Gamma \ast} \omega_{\Gamma} = p_{\Gamma \ast}\left( \bigwedge_{E \in E_{\Gamma}} \omega_E\right)
\end{equation}
where $p_{\Gamma}:C_{V_{\Gamma}^{\mathrm{int}}}(X_{/B};(s_1,\ldots, s_k))_{(s_0,v)}  \rightarrow B$ is the smooth map induced by the composition of projection maps $C_{V_{\Gamma}^{\mathrm{int}}}(X_{/B};(s_1,\ldots, s_k))_{(s_0,v)} \rightarrow C_{V^{\mathrm{ext, S^{\ast}}}_{\Gamma}}(X_{/B})\rightarrow B$ and $\omega_E$ is obtained by replacing $\mathsf{D}^{\bC} G_E$ with $\frac{1}{2} \widetilde{d^{\bC} G}_E$ (cf. Corollary \ref{cor:extended_prop}) in \eqref{eq:def_omega_gamma} for internal or $S^{\ast}$-decorated external edge $E$. 
\end{proposition}

\begin{proof}
    We shall take a close look at one loop term $\mathbf{S}^{\mathrm{eff},1/2}_{\text{1-loop}}$. Note that direct computation indicates that the differential operator $\mathsf{D}^{\mathbb{C}}$ restricts to the operator
\begin{equation}
   \mathsf{D}^{\mathbb{C}}|_{(\frac{1}{2}, \frac{1}{2})} = \frac{1}{2}\partial - \frac{1}{2}\bar{\partial} = \frac{1}{2} d^{\bC}.
\end{equation}
Thus, $Z_{\Gamma, (b,(1/2,1/2))}$ for a decorated connected uni-trivalent graph is given by replacing $\mathsf{D}^{\mathbb{C}} G_{v}$ with $\frac{1}{2}d^{\bC} G_{v}$ in the definition of the map $Z_{-}$ explained in Section \ref{subsection:gen_hdoge_corr}. By Corollary \ref{cor:extended_prop}, we can also see that $I_{\Gamma}$ is well-defined. This completes the proof.
\end{proof}

\subsection{$\mathbf{S}^{\mathrm{eff}}$, $\mathbf{G}$, and the formal KZ connection}
This subsection gives an explicit relation between $\mathbf{S}^{\mathrm{eff}}$ and the formal KZ connection introduced by Drinfel'd (\cite{Dri90}). Consider the restriction $\mathbf{G}'$ of the Hodge correlator class $\mathbf{G}$ to the twistor line $(1-u, u) \in \mathbb{C}_{\twi} \subset \mathbb{C}^2$:

\begin{equation}
        \mathbf{G}' = \mathbf{G}_{0,0} + \sum_{s,t \geq 0} (1-u)^s u^t (du \wedge (s + t +1) \mathbf{G}_{s+1, t+1}  +(u \partial - (1-u) \bar{\partial}) \mathbf{G}_{s+1, t+1})
\end{equation}
This provides a gauge transformation between two flat connection $\mathbf{G}'|_{u}$ and $\mathbf{G}'|_{u'}$ over $B$ for parameters $ u, u' \in \mathbb{C}_{\twi}=\{(1-u,u) \in \mathbb{C}^2\} \subset \mathbb{C}^2$. In particular, considering the case that $u=1, 1/2, 0$ gives gauge transformations between holomorphic, real, and anti-holomorphic flat connections respectively.

Now, we take the enhanced moduli space $\mathcal{M}_{0,4}' \simeq \mathbb{P}^1(\bC) \setminus\{0,1,\infty\} (+\text{tangential base point})$ of genus 0 surface with 4 marked points as $B$ and consider the fiber $\gr^W \pi_1^{\mathrm{nil}}( \mathbb{P}^1 (\bC)\setminus\{0,1,\infty\}, z)$ over $z  \in \mathbb{P}^1(\bC) \setminus\{0,1,\infty\}$. Then, we have $\mathbf{G}_{0,0}=0$ and $\mathbf{G}$ has coefficients only of the form $\mathbf{G}_{k,k}$ $(k \geq 1)$. Therefore, 
\begin{equation}
	\mathbf{G}'|_{u=1}  = \partial \mathbf{G}_{1,1}.
\end{equation}
Since it is not explicit in \cite{GoncharovHodge1}, we compute it further. Recall that the Green function is given by 
\begin{equation}
	(2 \pi \sqrt{-1}) G_{v_z}(x,y) = \log|x-y| - \log|z - x| - \log|z-y| + C(z, v_z).
\end{equation}
Here, $C(z, v_z)$ is a function which depends smoothly on $(z,v_z) \in T \mathbb{P}^1(\mathbb{C})$. Since  we now have $S = \{0,1,\infty, z\}$, $S^{\ast}=\{0,1,\infty\}$ and $\partial \log|s_i - s_j|=0$ for $s_i,s_j \in S^{\ast}$ with $s_i \neq s_j$, 
\begin{equation}
    \begin{split}
        \partial \mathbf{G}_{1,1} & = \partial G_{v_z}(0,1)\otimes \mathcal{C}(X_0 X_1) + \partial G_{v_z}(0,\infty)\otimes \mathcal{C}(X_0 X_{\infty}) +\partial G_{v_z}(1, \infty)\otimes \mathcal{C}(X_1 X_{\infty}) \\
	 & =\frac{-1}{2 \pi \sqrt{-1}} \left( \left(\frac{dz}{z} + \frac{dz}{z-1} \right)\otimes \mathcal{C}(X_0 X_1) + \left(\frac{dz}{z} \right)\otimes \mathcal{C}(X_0 X_{\infty}) +  \left( \frac{dz}{z-1} \right)\otimes \mathcal{C}(X_1X_{\infty})\right)\\
  &+ \frac{1}{2 \pi \sqrt{-1}}\partial C(z, v_z) \otimes \left(\mathcal{C}(X_0 X_1) + \mathcal{C}(X_0 X_{\infty}) + \mathcal{C}(X_1 X_{\infty}) \right).
    \end{split}
\end{equation}
where $\mathcal{C}(X_i X_j)$ for $i,j \in S^{\ast}$ is the cyclic word in $\mathcal{C}(A_{H, S^{\ast}})$ represented by $X_i \otimes X_j$ and $A_{H, S^{\ast}}$ denotes the tensor algebra of $V_{X, S^{\ast}}$.
Consider the map $\mathsf{D}$ for cyclic words (\cite[\S 8]{GoncharovHodge1}) defined by
\begin{equation}
	\mathsf{D}: \mathcal{C}(A_{H, S^{\ast}}) \rightarrow A_{H, S^{\ast}} \otimes \mathbb{Q}[S^{\ast}], \quad \mathsf{D}(F) = \sum_{s \in S^{\ast}} \frac{\partial F}{\partial X_s} \otimes X_s.
\end{equation}
Then, the concerning cyclic words are mapped by $\mathsf{D}$ as follows:
\begin{align}
&\mathcal{C}(X_0 X_1) \mapsto X_0 \otimes X_1 + X_1 \otimes X_0,\\
& \mathcal{C}(X_0 X_{\infty}) \mapsto X_0 \otimes X_{\infty} + X_{\infty} \otimes X_0, \\
& \mathcal{C}(X_1 X_{\infty}) \mapsto X_1 \otimes X_{\infty} + X_{\infty} \otimes X_1.
\end{align}
By using the relation $X_{\infty} = - X_0 - X_1$, we have
\begin{equation}
    \begin{split}
        &X_1 \otimes X_0 + X_{\infty} \otimes X_0 = - X_0 \otimes X_0,\\
&X_0 \otimes X_1 + X_{\infty} \otimes X_1 = - X_1 \otimes X_1
    \end{split}
\end{equation}
so that $\mathsf{D}$ gives
\begin{equation}
\begin{split}
   &\mathcal{C}(X_0 X_1) + \mathcal{C}(X_0 X_{\infty}) \mapsto X_0 \otimes X_1 + X_0 \otimes X_{\infty} - X_0 \otimes X_0 \\
   & \mathcal{C}(X_0 X_1) + \mathcal{C}(X_1 X_{\infty}) \mapsto X_1 \otimes X_0 + X_1 \otimes X_{\infty} - X_1 \otimes X_1.
   \end{split}
\end{equation}
Since $H_1(\mathbb{P}^1(\bC))) =0$, for cyclic word $F \in \calLie_{H,S^{\ast}}$, we can assign the corresponding special derivation (cf. \cite[\S 8.2]{GoncharovHodge1})
\begin{equation}
    \calLie_{H,S^{\ast}} \overset{\sim}{\rightarrow} \mathrm{Der}^S(L_{H, S^{\ast}})
\end{equation}
by setting
\begin{equation}
    X_{s_0} \mapsto 0,\quad X_s \mapsto \left[X_s, \frac{\partial F}{\partial X_s}\right] \quad (s \in S^{\ast})
\end{equation}
where $L_{H, S^{\ast}}$ denotes the free Lie algera generated by $V_{X,S^{\ast}}$ over $\bC$ and $\mathrm{Der}^S(L_{H, S^{\ast}})$ the special deribation Lie algebra acting on $L_{H, S^{\ast}}$.
Therefore, as a special derivation Lie algebra valued 1-form, one obtains
\begin{equation}\label{eq:KZ=G}
    \begin{split}
        \mathbf{G}'|_{u=1} & =  \partial \mathbf{G}_{1,1} \\
	&= - \frac{1}{2 \pi \sqrt{-1}} \left( \frac{dz}{z}\otimes \ad(X_0) + \frac{dz}{z-1}\otimes \ad(X_1) \right) \in \Omega^1(\mathbb{P}^1(\bC) \setminus\{0,1,\infty\}) \otimes \mathrm{Der}^S(L_{H, S^{\ast}}).
    \end{split}
\end{equation}
Note that the regular part vanishes as spacial derivations since 
\begin{align}
	&\mathcal{C}(X_0 X_1) + \mathcal{C}(X_0 X_{\infty}) + \mathcal{C}(X_1 X_{\infty})\\
	\mapsto & (X_0 + X_{\infty}) \otimes X_1 + (X_1 + X_{\infty}) \otimes X_0 + (X_0 + X_1) \otimes X_{\infty}\\
	=& - X_0 \otimes X_0 - X_1 \otimes X_1 -  X_{\infty} \otimes X_{\infty}.
\end{align}
Recall that the formal KZ connection $\omega^{\mathrm{KZ}}$ is given by
\begin{equation}\label{eq:formal_kz}
    \omega^{\mathrm{KZ}} = \frac{1}{2 \pi \sqrt{-1}} \left( \frac{dz}{z}\otimes X_0 + \frac{dz}{z-1}\otimes X_1\right) \in \Omega^1(\mathbb{P}^1(\bC)  \setminus\{0,1,\infty\}) \otimes \bC\langle\langle  X_0, X_1 \rangle \rangle.
\end{equation}
where $\bC\langle\langle  X_0, X_1 \rangle \rangle$ denotes the ring of non-commutative formal power series on variables $X_0$ and $X_1$.
Thus, \eqref{eq:KZ=G} and \eqref{eq:formal_kz} read
\begin{equation}
    \mathbf{G}'|_{u=1}  =  \partial \mathbf{G}_{1,1} = -(\Id \otimes \mathrm{ad})\omega^{\mathrm{KZ}}.
\end{equation}

Therefore, as a special derivation Lie algebra valued 1-form, the formal KZ connection has been obtained from $\mathbf{G}'|_{u=0}$ explicitly way. When $u=1$ we obtain its anti-holomorphic version by the same argument. Then, this argument implies the following theorem.
\begin{proposition}\label{prop:M0n'_GHC}
    Take $X \rightarrow B$ as the universal family $\mathcal{M}_{0,n}'$ of compact Riemann surfaces of genus $0$ with $n$ distinct points $s_0, s_1,\ldots, s_{n-1}$ and  tangent vector at $s_0$ over $\mathcal{M}_{0,n}'$. Then, we have  $\mathbf{S}^{\mathrm{eff}}{}'|_{u=0} =\mathbf{S}^{\mathrm{eff}}_{\mathrm{tree}}{}'|_{u=0}  = \mathbf{G}{}'|_{u=0}$ and $\mathbf{S}^{\mathrm{eff}}{}'|_{u=1} =\mathbf{S}^{\mathrm{eff}}_{\mathrm{tree}}{}'|_{u=1}  = \mathbf{G}{}'|_{u=1}$. In particular, when $\mathcal{M}_{0,4}'\simeq \mathbb{P}^1(\bC)\setminus \{0,1,\infty\} (+\text{tangent vector at $s_0$})$, we have
    \begin{equation}\label{eq:prop:M0n'_GHC}
    \begin{split}
    &\mathbf{S}^{\mathrm{eff}}{}'|_{u=1}  = \mathbf{S}^{\mathrm{eff}}_{\mathrm{tree}}{}'|_{u=1}= \mathbf{G}{}'|_{u=1} = -(\Id \otimes \mathrm{ad})\omega^{\mathrm{KZ}} \\
         &\mathbf{S}^{\mathrm{eff}}{}'|_{u=0}  =\mathbf{S}^{\mathrm{eff}}_{\mathrm{tree}}{}'|_{u=0} =  \mathbf{G}{}'|_{u=0} = -(\Id \otimes \mathrm{ad})\overline{\omega}^{\mathrm{KZ}}
         \end{split}
    \end{equation}
    where $\omega^{\mathrm{KZ}}$ denotes the formal KZ connection 1-form taking values in the special derivation algebra of the free Lie algebra with two variables. Here, equalities between $\mathbf{G}{}'$ and $\mathbf{S}^{\mathrm{eff}}{}'$ means those in the sense of Theorem \ref{thm:tree_red}.
\end{proposition}

\begin{proof}
    Since all the external vertices of connected trivalent graphs are labeled by sections in $S^{\ast}$, $\mathbf{S}^{\mathrm{eff}}_{\text{1-loop}}$ must consist of terms multiplied by $z^kw^k$ for $k > 0$ as in \eqref{eq:omega_g_loop}. This implies that $\mathbf{S}^{\mathrm{eff}}_{\text{1-loop}}{}'|_{u=0} = \mathbf{S}^{\mathrm{eff}}_{\text{1-loop}}{}'|_{u=1} =0$. Therefore, $\mathbf{S}^{\mathrm{eff}}{}'|_{u=0} = \mathbf{S}^{\mathrm{eff}}_{\mathrm{tree}}{}'|_{u=0}= \mathbf{G}{}'|_{u=0}$. Similarly, we get $\mathbf{S}^{\mathrm{eff}}{}'|_{u=1} = \mathbf{S}^{\mathrm{eff} '}_{\mathrm{tree}}{}'|_{u=1}=\mathbf{G}{}'|_{u=1}$. The latter statement directly follows from the former one together with \eqref{eq:KZ=G}.
\end{proof}

\begin{remark}
    The properties described in Proposition \ref{prop:M0n'_GHC} are close to the properties satisfied by Alekseev--Torossian connections (\cite{AT}) studied in Rossi--Wilwacher (\cite[\S 8]{RW14}) and further by Furusho from the viewpoints of Drinfel'd associators as generating series of multiple zeta values (\cite{Fur18}). However, at this moment, their relations to Hodge correlator twistor connection are not clear.
\end{remark}

\appendix
\section{Goncharov's quantum field theory and one loop part of $\mathbf{S}^{\mathrm{eff}}$}\label{section:appendix}
\renewcommand{\thesubsection}{\thesection.\arabic{subsection}}
In this appendix, we recall Goncharov's quantum field theory and discuss relations between one loop part of $\mathbf{S}^{\mathrm{eff}}$ and that of the asymptotic expansion of Goncharov's quantum field theory. Besides, we also describe propagators explicitly which will be used to compute higher loop terms of the asymptotic expansion.

\subsection{Goncharov's quantum field theory}
To begin with, following \cite[\S 13]{GoncharovHodge1}, we briefly recall Goncharov's quantum field theory. Throughout this appendix $X(\mathbb{C})$ denotes a complex curve. Let $N$ be a fixed positive integer. Let $\mathrm{Mat}_N(\mathbb{C})$ denote the vector space of $N \times N$ complex matrices. Then, the space of fields we consider is $\mathrm{Mat}_N(\mathbb{C})$-valued smooth functions on $X(\mathbb{C})$, i.e., $C^{\infty}(X(\mathbb{C}); \mathrm{Mat}_N(\mathbb{C})):= C^{\infty}(X(\mathbb{C}))\otimes_{\mathbb{C}} \mathrm{Mat}_N(\mathbb{C})$. The action functional
\begin{equation}
	S: C^{\infty}(X(\mathbb{C}); \mathrm{Mat}_N(\mathbb{C})) \rightarrow \mathbb{C}
\end{equation}
of the theory is defined as
\begin{equation}
	S(\varphi) = \frac{1}{2\pi \sqrt{-1}} \int_{X(\mathbb{C})} \Tr \left(\frac{1}{2} \partial \varphi \wedge \bar{\partial} \varphi + \frac{\hbar}{6} \cdot \varphi[\partial \varphi, \bar{\partial} \varphi]\right), \quad \varphi \in C^{\infty}(X(\mathbb{C}); \mathrm{Mat}_N(\mathbb{C})).
\end{equation}
To relate Hodge correlators, we need to consider some correlation functions of the functional $\mathcal{F}_W$ associated with cyclic words $W$ as follows: Choose a cyclic word
\begin{equation}
	W = \mathcal{C}(\{a_0\}\otimes \omega^0_1 \otimes \cdots \otimes \cdots \otimes \omega^0_{n_0} \otimes \cdots \otimes  \{a_m\} \otimes \omega^m_1 \otimes \cdots \otimes \omega^m_{n_m}) \in \mathcal{C} T(V^{\vee}_{X,S}),
\end{equation}
where $a_0, \ldots, a_m \in X(\mathbb{C})$ and $\omega^i_j \in \Omega^1_X \oplus  \overline{\Omega}_X^1$. Note that $a_0, \ldots, a_m$ are not necessarily distinct points. For such a cyclic word $W$, we define a functional
\begin{equation}
	\mathcal{F}_W : C^{\infty}(X(\mathbb{C}); \mathrm{Mat}_N(\mathbb{C})) \rightarrow \mathbb{C} 
\end{equation}
as a trace of a combination of the following three types of elementary matrix-valued functionals on the space of fields:
\begin{enumerate}[(1)]
\item For each point $a \in X(\mathbb{C})$, we assign to it the functional
\begin{equation}
	\mathcal{F}_a(\varphi):= \varphi(a) \in \mathrm{Mat}_N(\mathbb{C}).
\end{equation}
\item For a 1-form $\omega$ on $X(\mathbb{C})$, we assign  to it the functional
\begin{equation}
	\mathcal{F}_{\omega}(\varphi) := \int_{X(\mathbb{C})} [ \varphi, d^{\mathbb{C}} \varphi] \wedge \omega \in \mathrm{Mat}_N(\mathbb{C})
\end{equation}
\item For a pair of 1-forms $(\omega_1, \omega_2)$ on $X(\mathbb{C})$, we assign to it the functional
\begin{equation}
	\mathcal{F}_{(\omega_1, \omega_2)}(\varphi):= \int_{X(\mathbb{C})} \varphi(x) \omega_1 \wedge \omega_2 \in \mathrm{Mat}_N(\mathbb{C}).
\end{equation}
\end{enumerate}

The functional $\mathcal{F}_W(\varphi)$ by following way: Let us choose a (possibly empty) collection $\mathcal{P}$ of consecutive pairs $\{ \omega^i_{j}, \omega^i_{j+1}\}$ of 1-forms entering the cyclic word $W$. For a field $\varphi$, according to the cyclic ordering of $W$, we assign $\mathcal{F}_a(\varphi)$ for $\{a\}$, $\mathcal{F}_{\omega}(\varphi)$ for a 1-form $\omega$ not contained in $\mathcal{P}$ and $\mathcal{F}_{(\omega^i_j, \omega^i_{j+1})}(\varphi)$ for a consecutive pair  $\{ \omega^i_{j}, \omega^i_{j+1}\}$ in $\mathcal{P}$. Thus, we obtain a cyclic word of $\mathrm{Mat}_N(\mathbb{C})$ so that its trace is well-defined and denote it by $\mathcal{F}_{W, \mathcal{P}}$. In this way, the $\mathbb{C}$-valued functional $\mathcal{F}_{W, \mathcal{P}}$ is defined. Then, We set 
\begin{equation}
	\mathcal{F}_W(\varphi) := \sum_{\mathcal{P}} \mathcal{F}_{W, \mathcal{P}}(\varphi)
\end{equation}
where the sum is taken over all possible collection $\mathcal{P}$. 
\begin{example}
(1) For $W = \mathcal{C}(\{s_0\} \otimes \{s_1\} \otimes \cdots \otimes \{s_m\})$, its corresponding functional $\mathcal{F}_W$ is given as
\begin{equation}
	\mathcal{F}_W(\varphi) = \mathcal{F}_{ \mathcal{C}(\{s_0\} \otimes \{s_1\} \otimes \cdots \otimes \{s_m\})}(\varphi) = \Tr\left(\varphi(s_0) \varphi(s_1) \cdots \varphi(s_m)\right).
\end{equation}
(2) Let us choose the cyclic word $W = \mathcal{C}(\{s_0\} \otimes \omega_1 \otimes \omega_2)$. Then, there are two possible collection of $\mathcal{P}$; $\mathcal{P}=\emptyset$ and $\mathcal{P}=\{ \{\omega_1, \omega_2\}\}$. Thus, the functional corresponding to $W$ is given as
\begin{equation}
\begin{split}
	\mathcal{F}_W(\varphi) &= \mathcal{F}_{\mathcal{C}(\{s_0\} \otimes \omega_1 \otimes \omega_2)}(\varphi) \\
	&= \Tr\left(\varphi(s_0) \int_{X(\mathbb{C})} [ \varphi, d^{\mathbb{C}} \varphi]\wedge  \omega_1  \int_{X(\mathbb{C})} [ \varphi, d^{\mathbb{C}} \varphi]\wedge  \omega_2  \right) + \Tr \left(\varphi(s_0) \int_{X(\mathbb{C})} \varphi(x) \omega_1 \wedge \omega_2 \right).
\end{split} 
\end{equation}
\end{example}
Then, the correlator corresponding to $W$ is ``defined'' as the Feynman integral
\begin{equation}
	Z_{X,N, \hbar, \mu}(W):= \langle \mathcal{F}_W \rangle_{X,N, \hbar, \mu} := \int \mathcal{F}_W(\varphi) e^{\sqrt{-1} S(\varphi)} \mathcal{D} \varphi.
\end{equation}

So far, this integral has not been defined mathematically rigorously in general, but its perturbative series may be mathematically rigorously defined if its coefficients are finite or regularized.

\subsection{Feynman rules for the theory}
Following \cite[\S 13.2]{GoncharovHodge1}, we describe the Feynman rules for the quantum field theory and we reformulate it for our purpose.

The quadratic form $B$ of $S(\varphi)$ is given as
\begin{equation}
	B(\varphi) = \frac{1}{2\pi \sqrt{-1}} \sum_{i,j=1}^N \int_{X(\mathbb{C})} \partial \varphi^{i}_j \wedge \bar{\partial} \varphi^{j}_i = - \frac{1}{2 \pi \sqrt{-1}} \sum_{i,j=1}^N \int_{X(\mathbb{C})} \varphi^i_j \cdot \partial \bar{\partial} \varphi^j_i = \sum_{i,j=1}^N \int_{X(\mathbb{C})} \varphi^i_j \cdot \Delta \varphi^j_i
\end{equation}
so the propagator is given by
\begin{equation}
	\langle \varphi(z) \varphi(w) \rangle := B^{-1} = - G_{\mu} (z,w)  \sum_{i,j=1}^N e^i_j \otimes e^j_i 
\end{equation}
where $e^i_j$ is the elementary matrix with all zero entries except that the $(i,j)$ entry is given by 1. Since the cubical term of $S(\varphi)$ is given by
\begin{align}
& \hbar \cdot \int_{X(\mathbb{C})} \Tr\left(\varphi \cdot [\partial \varphi, \bar{\partial} \varphi] \right)\\
=& \hbar \int_{X(\mathbb{C})} \sum_{i,j,k=1}^N \left( \varphi^i_j \partial{\varphi}^j_k \bar{\partial} \varphi^k_i +  \varphi^i_j \bar{\partial}{\varphi}^j_k \partial \varphi^k_i \right)\\
= & - \hbar  \int_{X(\mathbb{C})} \sum_{i,j,k=1}^N \left( \varphi^i_j d^{\mathbb{C}} \varphi^j_k \wedge d^{\mathbb{C}} \varphi^k_i \right)
\end{align}
where we use
\begin{align}
	 [\partial \varphi, \bar{\partial} \varphi] & =  \sum_{a,b,c,d=1}^N \partial \varphi^{a}_b \wedge \bar{\partial} \varphi^c_d (e^a_b e^c_d - e^c_d e^a_b)\\
	& =  \sum_{a,b,c,d=1}^N \partial \varphi^{a}_b \wedge \bar{\partial} \varphi^c_d e^a_b e^c_d  + \bar{\partial} \varphi^{c}_d \wedge \partial \varphi^a_b  e^c_d e^a_b)\\
	& = \sum_{a,b,c,d=1}^N  (\partial \varphi^{a}_b \wedge \bar{\partial} \varphi^c_d + \bar{\partial} \varphi^{a}_b \wedge \partial \varphi^c_d) e^a_b e^c_d.
\end{align}

To obtain the integration associated with a Feynman graph, we need the following formulas involving the differentials:
\begin{enumerate}[(i)]
\item $\langle \varphi(z) \partial \varphi(w) \rangle = \partial_w G_{\mu}(z,w)$,
\item $\langle \partial \varphi(z) \partial \varphi(w) \rangle = \partial_z \partial_w \log E(z,w)= \partial_z \partial_w \log \theta(z-w)$,
\item  $\langle \bar{\partial} \varphi(z) \bar{\partial} \varphi(w) \rangle = \bar{\partial}_z \bar{\partial}_w \log \overline{E(z,w)}= \partial_z \partial_w \log \overline{\theta(z-w)}$,
\item $\langle \partial \varphi(z) \bar{\partial} \varphi(w) \rangle = \frac{\sqrt{-1}}{2} \sum_{i=1}^g p^{\ast}_1 \alpha_i \wedge p^{\ast}_2 \bar{\alpha}_i$,
\item $\langle \bar{\partial} \varphi(z) \partial \varphi(w) \rangle = - \frac{\sqrt{-1}}{2}\sum_{i=1}^g p_2^{\ast} \bar{\alpha}_i \wedge p_1^{\ast}\alpha_i$,
\end{enumerate}
where $E(z,w)$ is the prime form (cf. \cite[Definition 2.1]{Fay73}, \cite{Wen91}) and $\theta(z)$ is a theta function on $X(\mathbb{C})$ (cf. \cite[Corollary 2.6]{Fay73}). 
If we use $d^{\mathbb{C}}$ instead of using $\partial$ and $\bar{\partial}$ separately, we obtain the following another formulas:
\begin{itemize}
\item $\langle \varphi(z) d^{\mathbb{C}} \varphi(w) \rangle = d^{\mathbb{C}}_w G_{\mu}(z,w)$,
\item $\langle d^{\mathbb{C}} \varphi(z)d^{\mathbb{C}} \varphi(w) \rangle = d^{\mathbb{C}}_z d^{\mathbb{C}}_w G_{\mu}(z,w)= \eta(z,w) + \varrho(z,w)$
\end{itemize}
where we set
\begin{align}
	& \eta(z,w) := \langle \partial \varphi(z) \bar{\partial} \varphi(w) \rangle + \langle \bar{\partial} \varphi(z) \partial \varphi(w) \rangle = \frac{\sqrt{-1}}{2}\sum_{i=1}^g p_1^{\ast} \bar{\alpha}_i \wedge p_2^{\ast}\alpha_i -  \frac{\sqrt{-1}}{2}\sum_{i=1}^g p_2^{\ast} \bar{\alpha}_i \wedge p_1^{\ast}\alpha_i, \label{eq:A.2.9}\\
	& \varrho(z,w) := \langle \partial \varphi(z) \partial \varphi(w) \rangle+ \langle \bar{\partial} \varphi(z) \bar{\partial} \varphi(w) \rangle = \partial_z \partial_w \log E(z,w) + \bar{\partial}_z \bar{\partial}_w \log \overline{E(z,w)}.\label{eq:A.2.10}
\end{align}

Note that except for (iv) and (v) the differential forms that appeared in the above do not extend smoothly to the compactified two-point configuration space $C_2(X)$.

\begin{example}\label{example:A.2.1}
    Let us consider the case of genus $g =1$ compact Riemann surfaces. Let $E_{\tau} = \mathbb{C}/(\bZ \oplus \tau \bZ)$ be the elliptic curve associated with $\tau \in \mathbb{H}$, the upper half complex plane. For $z \in E_{\tau}$, we define a character $\chi_z: \bZ \oplus \tau \bZ  \rightarrow U(1)$ by $\chi_z(\gamma):= \exp\left(\frac{2\pi \sqrt{-1} (z \bar{\gamma} - \bar{z} \gamma)}{\tau - \bar{\tau}} \right)$. Then, there are explicit formulas as follows:
\begin{align}
        & G(z-w) = \frac{\tau - \bar{\tau}}{2 \pi \sqrt{-1}} \sum_{\gamma \in (\mathbb{Z} \oplus \mathbb{Z} \tau)\setminus \{(0,0)\}} \frac{\chi_{z-w}(\gamma)}{|\gamma|^2}\label{eq:A.2.11}\\
& d^{\mathbb{C}} G(z-w) = \sum_{\gamma \in (\mathbb{Z} \oplus \mathbb{Z} \tau)\setminus \{(0,0)\}} \frac{\chi_{z-w}(\gamma)}{\gamma} (dz - dw) + \sum_{\gamma \in (\mathbb{Z} \oplus \mathbb{Z} \tau)\setminus \{(0,0)\}} \frac{\chi_{z-w}(\gamma)}{\bar{\gamma}} (d\bar{z} - d\bar{w}) \label{eq:A.2.12} \\
& \eta(z,w) = \frac{1}{\Im(\tau)} \left(dz\wedge d\bar{w} + d\bar{z} \wedge dw \right)\label{eq:A.2.13}\\
& \varrho(z,w) = \left(\wp(z-w) + \frac{\pi^2}{3} E^{\ast}_{2}(\tau, \bar{\tau}) \right) dz \wedge dw + (\text{its conjugation}) \label{eq:A.2.14}
   \end{align}
where $\wp(z)$ is the Weierstrass's elliptic function and $E_2^{\ast}(\tau, \bar{\tau}):=E_2(\tau) -\frac{3}{\pi \Im(\tau)}$.  Here $E_2$ is the Eisenstein series of weight 2 given by
\begin{equation}
    E_2(\tau) = 1 - 24 \sum_{n=1}^{\infty} \frac{ n q^n}{1 - q^n}, \quad q = e^{2\pi \sqrt{-1} \tau}
\end{equation}
(cf.\cite[p34, Example]{Fay73}, \cite[\S 2]{Li12}). Note that $\varrho(z,w)$ is smooth outside the diagonal but has poles of order 2 along it due to $\wp(z-w)$. 
\end{example}

\begin{remark}
    If we consider only the $(2, 0)$-form of \eqref{eq:A.2.14}, one can see a similarity to the BCOV propagator in \cite{CoSL12, Li12}, but the precise relation between Goncharov's QFT and BCOV theory seems not to be known as far as author's knowledge. See also \cite{Dijk95} and \cite{KZ95}.
\end{remark}
\subsection{Tree part}
Next, we recall the fundamental result by Goncharov on Hodge correlator integrals and the leading term of a perturbative series of some quantum field theory.

Goncharov showed that Hodge correlators are thought of as planer tree part of the formal power series as follows.
\begin{theorem}[Goncharov(\cite{GoncharovHodge1})]
Set $\hbar = N^{-1/2}$ and $Z_{X,N,\mu}(W)= Z_{X,N,N^{-1/2}, \mu}(W)$. Then, the leading term of asymptotic expansion of $\log Z_{X,N,\mu}(W)$ as $N \to \infty$ corresponds to the sum of integrals associated with planar trivalent trees. Moreover, the following holds up to sign,
	\begin{equation}
		\log Z_{X,N,\mu}(W) \sim \mathrm{Cor}_{\mathcal{H}}(W)\quad \text{as $N \to \infty$}.
	\end{equation}
\end{theorem}

\begin{remark}
    For non-planar tree uni-trivalent diagrams, the associated integrals themselves are the same as planar tree diagrams (up to sign) but their degree in $N$ are different.
\end{remark}

\subsection{One loop part}
This section studies one loop part of Goncharov's quantum field theory and we observe that it contains one loop terms of our generalized twistor connection. 

For simplicity, we restrict ourselves to the case that one loop connected trivalent graphs containing no special external vertices and they are one particle irreducible graphs. Let $\Gamma$ be such a trivalent graph. We will assign two $d^{\bC}_v$ to two of three adjacent half-edges associated with an internal vertex $v$ and $G_{\mu}(v,w)$ to each edge. There are two possibilities of assignments of $d^{\bC}_v$ as follows: (a) There is at least one edge both of whose two half-edges are assigned by $d^{\bC}_v$; (b) Every edge is assigned by exactly one $d^{\bC}_v$. In this subsection, we focus on case (b), since it is unclear how case (a) is related to our generalized twistor connection. Now, we explain the integrations associated with $\Gamma$ by assigning each edge with $d^{\bC}G_{\mu}$ can be written as linear combinations of the integrations given as case (b) by example. Consider the trivalent graph with one loop as Figure \ref{fig:A.4.1}. The leftmost graph is a one-loop term in our generalized twistor connection. Note that $X$ is manifold of real dimension 2 and  $2|V^{\mathrm{int}}_{\Gamma}| = |E_{\Gamma}|$. Then, by expanding $d^{\bC}$ as $d^{\bC} = d^{\bC}_x + d^{\bC}_y$ and so on, the associated integrals can be written by the combination given in the right-hand side of Figure \ref{fig:A.4.1} by dimensional reason.

\begin{figure}[h]
\captionsetup{margin=2cm}
\begin{center}
\begin{tikzpicture}
\node at (2,0) {$=$};
\node at (6,0) {$+$};
\begin{scope}
	\draw[densely dotted] (0,0) circle (0.8cm);
	\draw[densely dotted] (-1.5,0) -- (-0.8,0);
 \node at (-0.8, 0) {$\bullet$};
 \node at (-1.5, 0) {$\bullet$};
 \foreach \x in {60, -60}{
 \node at (\x:0.8cm) {$\bullet$};
 \node at (\x:1.5cm) {$\bullet$};
 \draw[densely dotted] (\x:0.8cm) --(\x:1.5cm);
 }
\node at (-0.5,0) {$x$};
\node at (0.4,0.4) {$y$};
\node at (0.4,-0.4) {$z$};
\node at  (-1.15,0.4) {$d^{\bC}$};
\node at  (0,1.1) {$d^{\bC}$};
\node at  (0,-1.1) {$d^{\bC}$};
\node at  (1.1,0) {$d^{\bC}$};
\node at  (0.9,1) {$d^{\bC}$};
\node at  (0.9,-0.9) {$d^{\bC}$};
\end{scope}
\begin{scope}[xshift=4cm]
	\draw[densely dotted] (0,0) circle (0.8cm);
	\draw[densely dotted] (-1.5,0) -- (-0.8,0);
 \node at (-0.8, 0) {$\bullet$};
 \node at (-1.5, 0) {$\bullet$};
 \foreach \x in {60, -60}{
 \node at (\x:0.8cm) {$\bullet$};
 \node at (\x:1.5cm) {$\bullet$};
 \draw[densely dotted] (\x:0.8cm) --(\x:1.5cm);
 }
\node at (-0.5,0) {$x$};
\node at (0.4,0.4) {$y$};
\node at (0.4,-0.4) {$z$};
\node at  (-1.15,0.4) {$d_x^{\bC}$};
\node at  (0,1.1) {$d_x^{\bC}$};
\node at  (0,-1.1) {$d_z^{\bC}$};
\node at  (1.1,0) {$d_y^{\bC}$};
\node at  (0.9,1) {$d_y^{\bC}$};
\node at  (0.9,-0.9) {$d_z^{\bC}$};
\end{scope}
\begin{scope}[xshift=8cm]
	\draw[densely dotted] (0,0) circle (0.8cm);
	\draw[densely dotted] (-1.5,0) -- (-0.8,0);
 \node at (-0.8, 0) {$\bullet$};
 \node at (-1.5, 0) {$\bullet$};
 \foreach \x in {60, -60}{
 \node at (\x:0.8cm) {$\bullet$};
 \node at (\x:1.5cm) {$\bullet$};
 \draw[densely dotted] (\x:0.8cm) --(\x:1.5cm);
 }
\node at (-0.5,0) {$x$};
\node at (0.4,0.4) {$y$};
\node at (0.4,-0.4) {$z$};
\node at  (-1.15,0.4) {$d_x^{\bC}$};
\node at  (0,1.1) {$d_y^{\bC}$};
\node at  (0,-1.1) {$d_x^{\bC}$};
\node at  (1.1,0) {$d_z^{\bC}$};
\node at  (0.9,1) {$d_y^{\bC}$};
\node at  (0.9,-0.9) {$d_z^{\bC}$};
\end{scope}
\end{tikzpicture}	
\end{center}
\caption[Example of a trivalent graph whose edges are labeled by $d^{\bC}$]{Example of a trivalent graph whose edges are labeled by $d^{\bC}$ can be decomposed into a combination of the same underlying graphs whose half-edges are labeled by $d_v^{\bC}$.}
\label{fig:A.4.1}
\end{figure}
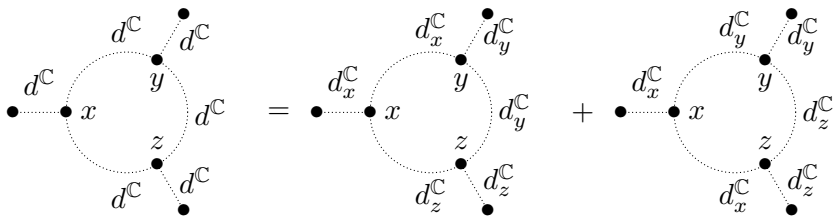

In this way, we can conclude that one-loop terms of our generalized twistor connection can be written as combinations of integrations obtained as the case (b). Therefore, in this sense, our generalized twistor connection is compatible with Goncharov's quantum field theory.

\subsection{Higher loop part}
For the higher loop part, in general, we would need some renormalizations to obtain well-defined integrals such as those developed in \cite{Cos11}. However, at least, one can expect that such integrations partially contain information on those represented by planar uni-trivalent graphs. This section briefly sketches them through examples for the case that $X$ is an elliptic curve. 

This section uses the same notation as Example \ref{example:A.2.1}. Let us consider the uni-trivalent graph $\Gamma$ as shown in Figure \ref{fig:A.5.1} whose two external vertices are labeled by $0, a \in E_{\tau}$. We assign forms to each edge of $\Gamma$ in the following way: To the leftmost external edge labeled by $0$, we assign the Green function $G$ in \eqref{eq:A.2.11}. For each loop, we assign $\eta$ \eqref{eq:A.2.13} to one of the edges consisting of a loop. To remaining edges, we assign $d^{\mathbb{C}} G$ \eqref{eq:A.2.12}.

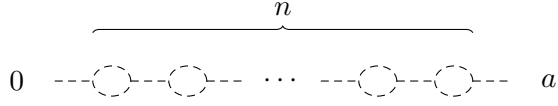
\begin{figure}[h]
\captionsetup{margin=2cm}
 \centering
 \begin{tikzpicture}
 \begin{scope}[yscale=0.8]
 \node at (3, 1.2) {$n$};
\draw[decorate,decoration={brace}] (0.5,0.8) --  (5.5,0.8);
 \draw[densely dashed] (0,0) -- (0.5,0);
 \draw[densely dashed] (0.75,0) circle (0.25cm);
 \draw[densely dashed] (1,0) -- (1.5,0);
 \draw[densely dashed] (1.75,0) circle (0.25cm);
 \draw[densely dashed] (2,0) -- (2.5,0);
 \node at (3,0) {$\cdots$};
 \draw[densely dashed] (3.5,0) -- (4,0);
 \draw[densely dashed] (4.25, 0) circle (0.25cm);
 \draw[densely dashed] (4.5,0) -- (5,0);
 \draw[densely dashed] (5.25, 0) circle (0.25cm);
 \draw[densely dashed] (5.5, 0)-- (6,0);
 \node at (-0.5,0) {$0$};
 \node at (6.5, 0) {$a$};
 \end{scope}
 \end{tikzpicture}
 \caption[Example of uni-trivalent graph with more than one loops]{Example of uni-trivalent graph with more than one loops}
\label{fig:A.5.1}
\end{figure}

Then, the associated integration to $\Gamma$ reduces to that of the uni-trivalent tree graph as in Figure \ref{fig:A.6.1} whose $2n$-external edges originated from the loop edges of $\Gamma$ are labeled by $\eta$.

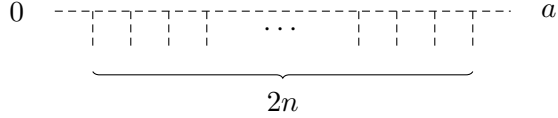
\begin{figure}[h]
\captionsetup{margin=2cm}
 \centering
 \begin{tikzpicture}
 \begin{scope}[yshift=-1.5cm]
 \draw[densely dashed] (0,0) -- (6,0);
 \foreach \x in {0.5, 1, 1.5, 2, 4, 4.5, 5, 5.5}{
 \draw[densely dashed] (\x,0) -- (\x, -0.5);
 }
 \node at (3,-0.25) {$\cdots$};
 \node at (-0.5,0) {$0$};
 \node at (6.5, 0) {$a$};
 \node at (3, -1.2) {$2n$};
\draw[decorate,decoration={brace}] (5.5,-0.8) --  (0.5,-0.8);
 \end{scope}
 \end{tikzpicture}
 \caption[Example of the uni-trivalent graph obtained by cutting along cycle edges of graphs with loops]{Example of the uni-trivalent graph obtained by cutting along cycle edges of graphs with loops}
\label{fig:A.6.1}
\end{figure}

By using the result by Goncharov on tree diagrams (\cite[Proposition 11.8]{GoncharovHodge1}), such associated integrals are given by the following Eisenstein--Kronecker series  up to sign,
\begin{equation}
	G(n+1, \chi_a) := \sum_{\gamma \in (\mathbb{Z} \oplus \mathbb{Z} \tau)\setminus \{(0,0)\}} \frac{\chi_a(\gamma)}{|\gamma|^{2n+2}}. 
\end{equation}

Therefore, one may expect that the asymptotic series of the 2-point correlation function 
\begin{equation}
	Z_{E, N,\vol_E}(\mathcal{C}(\{0\} \otimes \{a\})) = \int \Tr\left( \varphi(0)\varphi(a) \right) e^{\sqrt{-1} S(\varphi)} \mathcal{D} \varphi.
\end{equation}
 contains the generating function  of the Eisenstein--Kronecker $G(n+1, \chi_a)$ for $n \geq 0$. Note that the tree part of $Z_{E, N, \vol_E}(\mathcal{C}(\{0\} \otimes \{a\}))$ does not yield $G(n+1, \chi_a)$ for $n \geq 1$.


\bibliographystyle{amsalpha}
\bibliography{refs}

\noindent
Hisatoshi Kodani\\
Institute of Mathematics for Industry,
Kyushu University, 
744, Motooka, Nishi-ku, Fukuoka, 819-0395, Japan\\
Email address: kodani@imi.kyushu-u.ac.jp

\ 

\noindent
Yuji Terashima\\ 
Graduate School of Science, 
Tohoku University, 
6-3, Aoba, Aramaki-aza, Aoba-ku, Sendai, 980-8578, Japan\\
Email address: yujiterashima@tohoku.ac.jp

\end{document}